\documentclass[12pt,a4paper]{article}
\usepackage{amsthm}
\usepackage{amscd}
\usepackage{mathrsfs}
\usepackage{amsfonts}
\usepackage{amssymb}
\usepackage{pifont}
\usepackage{amsmath}
\usepackage{enumerate}
\usepackage{booktabs}
\usepackage{array}
\usepackage[toc,page]{appendix}
\usepackage[all]{xy}
\usepackage{txfonts}
\usepackage{latexsym}
\usepackage{amstext}
\usepackage{fancyhdr}
\usepackage{pifont}
\usepackage{lastpage}
\usepackage{hyperref}
\usepackage{verbatim}
\usepackage{bbding}
\usepackage{textcomp}

\usepackage{tikz}
\usetikzlibrary{%
  calc,
  fit,
  shapes,
  backgrounds,
  intersections
}

\newcommand{\bd}{\bar{\Delta}}
\newcommand{\td}{\bar{\Delta}}

\newcommand{\ch}{\mathrm{ch}}
\newcommand{\glt}{\widetilde{\mathrm{GL}^+(2,\mathbb R)}}
\newcommand{\pkp}{\mathrm{P}\big(\mathrm{K}_{\mathbb R}(\textbf {P}^2)\big)}
\newcommand{\mssq}{\mathfrak M^{s}_{\sigma_{s,q}}}
\newcommand{\msq}{\mathfrak M_{\sigma_{s,q}}}
\newcommand{\mgs}{\mathfrak M^{\mathrm{s}}_{\mathrm{GM}}}
\newcommand{\msssq}{\mathfrak M^{ss}_{\sigma_{s,q}}}
\newcommand{\mss}{\mathfrak M^{\mathrm{s}}_{\sigma}}

\newcommand{\repqa}{\mathbf{Rep}\left(Q_\mathcal E, \vec n,\alpha_{\mathcal E}\right)}
\newcommand{\repqwa}{\mathbf{Rep}\left(Q_\mathcal E, \vec n_w,\alpha_{\mathcal E}\right)}
\newcommand{\repqw}{\mathbf{Rep}\left(Q_\mathcal E, \vec n_w\right)}
\newcommand{\repq}{\mathbf{Rep}\left(Q_\mathcal E, \vec n \right)}
\newcommand{\Rep}{\mathbf{Rep}}

\newcommand{\pp}{\mathbf {P}^2}

\newcommand{\Hom}{\mathrm{Hom}}
\newcommand{\Geo}{\mathrm{Geo}}
\newcommand{\ext}{\mathrm{ext}}
\newcommand{\Ext}{\mathrm{Ext}}
\newcommand{\Ker}{\mathrm{Ker}}
\newcommand{\HH}{\mathrm{H}}
\newcommand{\Coh}{\mathrm{Coh}}
\newcommand{\Arg}{\mathrm{Arg}\text{ }}
\newcommand{\Stab}{\mathrm{Stab}}
\newcommand{\Cone}{\mathrm{Cone}}
\newcommand{\lw}{L^{\text{last}}}
\newcommand{\lrw}{L^{\text{right-last}}}
\newcommand{\cccp}{\{1,\frac{\mathrm{ch}_1}{\mathrm{ch}_0},\frac{\mathrm{ch}_2}{\mathrm{ch}_0}\}\text{-plane}}
\newcommand{\vecm}{(m_1,m_2,m_3)}

\newcommand{\Epm}{E_{\left(\frac{p}{2^m}\right)}}
\newcommand{\chy}{\frac{\mathrm{ch}_1}{\mathrm{ch}_0}}
\newcommand{\che}{\frac{\mathrm{ch}_2}{\mathrm{ch}_0}}
\newcommand{\clp}{\mathrm{Cone}_{LP}}

\newcommand{\lbs}{\text{\Large$\backslash$}}
\newcommand{\lfs}{\text{\Large$/$}}
\newcommand{\lfss}{\text{\Large{//}}}
\renewcommand\appendix{\par
  \setcounter{section}{0}
  \setcounter{subsection}{0}
  \setcounter{figure}{0}
  \setcounter{table}{0}
  \renewcommand\thesection{Appendix \Alph{section}}
  \renewcommand\thefigure{\Alph{section}\arabic{figure}}
  \renewcommand\thetable{\Alph{section}\arabic{table}}
}

\newtheorem{theorem}{Theorem}[section]
\newtheorem{defn}[theorem]{Definition}
\newtheorem{prop}[theorem]{Proposition}
\newtheorem{cor}[theorem]{Corollary}
\newtheorem{lemma}[theorem]{Lemma}
\newtheorem{rem}[theorem]{Remark}

\newtheorem{pd}[theorem]{Proposition and Definition}

\title{Birational models of moduli spaces of coherent sheaves on the projective plane}
\author{Chunyi Li and Xiaolei Zhao}
\date{\today}

\begin{document}
\maketitle
\begin{abstract}
In this paper, we study the birational geometry of moduli spaces
of semistable sheaves on the
projective plane via Bridgeland stability conditions. We show that the entire MMP of their moduli spaces can be run via wall-crossing. Via a description of the walls, we give a numerical description of their movable cones, along with its chamber decomposition corresponding to minimal models. As an application, we show that for primitive vectors, all birational models corresponding to open chambers in the movable cone are smooth and irreducible.
\end{abstract}

\section*{Introduction}

Birational geometry of moduli space of sheaves on surfaces has been studied a lot in recent years, see \cite{ABCH, BM1, BM2, BMW, CH, CH2, CH3, CHW, LZ, Woolf}. The milestone work in \cite{BM1,BM2} completes the whole picture for K3 surfaces. In this paper, we give a complete description for the minimal model program of the moduli space of semistable sheaves on the projective plane via wall-crossings in the space of Bridgeland stability conditions. In particular, we deduce a description of their nef cone, movable cone and the chamber decomposition for their minimal models.\\

\textbf{Geometric stability conditions on $\pp$:} The notion of \emph{stability condition} on a $\mathbb C$-linear triangulated category was first introduced in \cite{Bri07} by Bridgeland. A stability condition consists of a slicing $\mathcal P$ of semistable objects in the triangulated category and a central charge $Z$ on the Grothendieck group, which is compatible with the slicing. In particular in this paper, we consider the bounded derived category of coherent sheaves on the projective plane. A stability condition $\sigma=(Z,\mathcal P)$ is called {\em geometric} if it satisfies the support property and all sky scraper sheaves are $\sigma$-stable with the same phase, see Definition \ref{defn:stabgeo}.

The Grothendieck group K$(\pp)$ of D$^b(\pp)$ is of rank $3$ and K$_\mathbb R(\pp)(=\mathrm{K}(\pp)\otimes \mathbb R)$ is spanned by the Chern characters $\ch_0$, $\ch_1$ and $\ch_2$. Due to the work of Drezet and Le Potier, there is a {\em Le Potier cone} (see the picture below Definition \ref{defn:lpcurve}) in the space K$_\mathbb R(\pp)$, such that there exists slope stable coherent sheaves with character $w=\left(\ch_0(>0),\ch_1,\ch_2\right)\in \mathrm{K}(\pp)$ if and only if either $w$ is the character of an exceptional bundle, or it is not inside the Le Potier cone. By taking the kernel of the central charge, the space of all geometric stability conditions can be realized as a principal $\glt$-bundle over Geo$_{LP}$, which is an open region above the {\em Le Potier curve}, see Proposition \ref{thm:geostab}.  Note that the $\glt$-action does not affect the stability of objects. We will write a geometric stability condition as $\sigma_{s,q}$ with $(s,q)\in\Geo_{LP}$ indicating the kernel of its central charge. Let $\msq^{s(ss)}(w)$ be the moduli space of $\sigma_{s,q}$-(semi)stable objects in the heart Coh$_{\#s}$ of the stability condition $\sigma_{s,q}$ with character $w\in \mathrm{K}(\pp)$, we address the following questions:
\begin{enumerate}
\item For a Chern character $w$ and a geometric stability condition $\sigma_{s,q}$, when $\msssq(w)$ is non-empty?

\item How does $\msssq(w)$ change when $\sigma_{s,q}$ varies in Geo$_{LP}$?
\end{enumerate}
The first question is answered step by step in several parts of the paper. Similar to the result of Drezet and Le Potier, when the character $w$ is inside the Le Potier cone and not exceptional (see Corollary \ref{cor:regionofEstab}), there is no $\sigma_{s,q}$-semistable object with character $w$ (or $-w$) for any geometric stability condition $\sigma_{s,q}$. In other words, $\msssq(w)$ is always empty. When the character $w$ is proportional to an exceptional character (see Corollary  \ref{cor:nostabincone}), both $\msssq(w)$ and $\msssq(-w)$ are empty if and only if: the point $(1,s,q)$ and the reduced character $\tilde w$ are on different sides of the vertical line $L_{e\pm}$ for some exceptional character $e$, and  $(1,s,q)$ is below the line $L_{we}$ in the Geo$_{LP}$.

The main case of the first question is when the character is not inside the Le Potier cone.
 
\begin{theorem}[Lemma \ref{lemma:lw}, Theorem \ref{prop:lastwall}]
Let $w\in \mathrm K(\pp)$ be a character not inside the Le Potier cone, then $\msssq(w)$ is empty if and only if  $\sigma_{s,q}$ is not above $\lw_w$ or $\lrw_w$ in the $\Geo_{LP}$.
\label{thm:mainthminintro}
\end{theorem}

The notations $\lw_w$ and $\lrw_w$ are defined in Definition \ref{def:lastwall}. As we will see, the description for the last wall is equivalent to that for the effective boundary of the moduli space. This is first solved in \cite{CHW} for the $\ch_0\geq 1$ case and in \cite {Woolf} for the torsion case by studying the effective cone of the moduli space. In this paper, we reprove these results in our set-up in a different way.\\

For the second question, we have the following result:
\begin{theorem}[Theorem \ref{smoothness and irred property}, Theorem \ref{left half upper plane's main theorem in
the body}]
Let $w$ be a primitive character. The moduli space $\mssq(w)$ is  smooth and connected for any generic geometric stability condition $\sigma_{s,q}$ when it is non-empty. Any two non-empty moduli spaces $\mss(w)$ and $\mathfrak M^s_{\sigma'}(w)$ are birational to each other. The actual walls (chambers) is one-to-one corresponding to the
stable base locus decomposition walls (chambers) of the divisor
cone of $\mathfrak M^s_{GM}(w)$. In particular, one can run the whole minimal model program for $\mathfrak M^s_{GM}(w)$ via wall crossing in the space of geometric stability conditions.
\label{thm:mainthm2inintro}
\end{theorem}

The smoothness result can be proved easily for moduli of slope stable sheaves. However, for Bridgeland stable objects, they may not remain stable after twisting by $\mathcal O(-3)$. The key point is to develop a method to compare slopes with respect to different Bridgeland stability conditions, and conclude the vanishing of Hom group. This is achieved first in \cite{LZ}, and generalized to the current situation in Section 2. The following consequence seems new to the theory of MMP of moduli of sheaves.

\begin{cor}
Let $w\in \mathrm{K}(\pp)$ be a primitive character not inside the Le Potier cone, then each minimal model of $\mathfrak M^s_{GM}(w)$ corresponding to an open chamber in the movable cone is smooth.
\label{cor:mmsmooth}
\end{cor}

The moduli space $\mssq(w)$ can be constructed via the geometric invariant theory, because there are the so-called \emph{algebraic stability conditions}. The space Stab$^{\text{Alg}}$ of algebraic stability conditions is large enough in the sense that for any Chern character $w$ outside the Le Potier cone and stability condition $\sigma_{s,q}$ in Geo$_{LP}$, the line segment $l_{\sigma w}$ always intersects Stab$^{\text{Alg}}$. The moduli space $\mathfrak M^s_{\sigma'}(w)$ does not depend on the choice of $\sigma'$ on $l_{\sigma w}$, we may assume that $\sigma_{s,q}$ is an algebraic stability condition.  After a suitable homological shift, the heart of an algebraic stability condition is the same as the representation space of a path algebra with relations. The moduli space of such kind of quiver representations can be constructed via the geometric invariant theory. As an immediate corollary, the GIT construction ensures that $\msssq(w)$ is projective. This construction first appears in \cite{ABCH} for Hilbert schemes of point on $\pp$, and in \cite{BMW} for moduli of sheaves, both using quivers associated to line bundles. In this paper, we further study the relationship between geometric stability and algebraic stability, and make available this construction for any full strong exceptional collections of $\pp$. This generalization will be important for some arguments in this paper.

When $q_0\gg 0$ and $s_0<\ch_1(w)$,  $\mathfrak M^s_{\sigma_{s_0,q_0}}(w)$ is the same as the moduli space $\mathfrak M^s_{GM}(w)$ of slope stable sheaves. For any Chern character $v$ right orthogonal to $w$ (i.e. $\chi( w,v)=0$), the Donaldson morphism provides a divisor class $\mathcal L_v$ in Pic$_{\mathbb R}(\mathfrak M^s_{GM}(w))$. On the other hand, for each moduli space $\mssq(w)$, the GIT construction provides a divisor class $[\mathcal L_{s,q}]$ in Pic$_{\mathbb R}(\mssq(w))$ up to a positive scalar. When the exceptional locus of the natural map $\mssq(w)\dashrightarrow\mathfrak M^s_{\sigma_{s_0,q_0}}(w)\simeq \mathfrak M^s_{GM}(w)$ (which is constructed via the variation of GIT concretely, see Section 2) has codimension greater than $1$, the divisor class $[\mathcal L_{s,q}]$  is also defined on Pic$_{\mathbb R}(\mathfrak M^s_{GM}(w))$. Suppose the line $L_{w\sigma_{s,q}}$ is given by $ ^\perp v_{s,q}$ for some $v_{s,q}\in w^\perp$, then the divisor class $[\mathcal L_{s,q}]$ from the GIT construction is the same as $[\mathcal L_{v_{s,q}}]$ given by the Donaldson morphism up to a positive scalar.\\

Based on the explicit correspondence between walls in the space of stability conditions and walls in the divisor cone, we may describe all stable base locus walls (including the boundaries of nef cone, effective cone and movable cone) as actual walls in the space of stability conditions. Here an actual wall for a Chern character $w$ is a potential wall $L_{\sigma w}$ such that curves are contracted from either side of $\mathfrak M^{ss}_{\sigma_\pm}(w)\dashrightarrow \mathfrak M^{ss}_{\sigma}(w)$. So it becomes an important question to ask when a potential wall is an actual wall. In Section 3, we give a numerical criteria on actual walls, which depends on only numerical data, and provides an effective algorithm to compute all actual walls for $w$.

\begin{theorem}[Theorem \ref{thm: actual wall}]
Let $w\in\mathrm K(\pp)$ be a Chern character with $\ch_0(w)\geq 0$ not inside the Le Potier cone. For any stability condition $\sigma$ inside $\bd_{< 0}\subset \mathrm{Geo}_{LP}$ between the last wall $\lw_w$ and the vertical wall $L_{w\pm}$, the wall $L_{\sigma w}$ is an actual wall for $w$ if and only if there exists a Chern character $v\in \mathrm{K}(\pp)$ on the line segment $l_{\sigma w}$ such that: $\ch_0(v)>0$, $\frac{\ch_1(w)}{\ch_0(w)}>\frac{\ch_1(v)}{\ch_0(v)}$, the characters $v$ and $w-v$ are either exceptional or not inside the Le Potier cone, and both of them are not in $\mathrm{TR}_{wE}$ for any exceptional bundle $E$.
\label{thm:main3inintro}\end{theorem}

TR$_{wE}$ is defined in Definition \ref{def:trew}. It is a small triangle area decided by $w$ and an exceptional character $E$. On an actual wall $L_{w\sigma}$, the Chern character $w$ can always be written as the sum of two proper Chern characters $w'$ and $w-w'$ satisfying the conditions in the theorem. The key is to prove the inverse direction: when such two characters exist, two stable objects with the corresponding characters extend to a stable object with character $w$, which is destabilized on the wall.  Roughly speaking, three main steps are involved:  1. $\mss(w')$ and $\mss(w-w')$ are non-empty; 2. the extension group ext$^1$ between two generic objects in $\mss(w')$ and $\mss(w-w')$ is non-zero; 3. the extension of two stable objects will produce $\sigma_+$ or $\sigma_-$ stable object with character $w$. 

The conditions in the theorem are mainly used in step 1, and step 3 follows from general computations. For step 2, based on the characters, we can only aim to show $\chi(w-w',w')< 0$. However, one may wonder about the case that on an actual wall, generic objects in $\mss(w')$ and $\mss(w-w')$ do not have non-trivial extensions, but objects on some jumping loci extend to $\sigma_\pm$-stable objects. Should this happen, $\chi(w-w',w')\geq 0$ but objects in $\mss(w')$ and $\mss(w-w')$ still extend to $\sigma_\pm$-stable objects. From this point of view, it is a bit surprising to have a numerical criteria for actual walls. The main point in Theorem \ref{thm:main3inintro} is to rule out this possibility, and this is done by gaining a good understanding of the last wall.

Moreover, the criteria decides all the actual walls effectively, in the sense that it involves only finitely many steps of decisions, and one may write a computer program to output all the actual walls with a given Chern character $w$ as the input. We compute the example for $w=(4,0,-15)$ by hands to show some details of this computation. A cartoon for the actual walls in this case is as follows.

\begin{center}
\begin{tikzpicture}[domain=1:5]

\tikzset{%
    add/.style args={#1 and #2}{
        to path={%
 ($(\tikztostart)!-#1!(\tikztotarget)$)--($(\tikztotarget)!-#2!(\tikztostart)$)%
  \tikztonodes},add/.default={.2 and .2}}
}

\newcommand\XA{0.1}

\coordinate (W) at (0,-3.75);
\node at (W) {$\bullet$};


\coordinate (V) at (-1.5,0.75);
\node [opacity=\XA] at (V) {$\bullet$};
\draw [add= 0 and 1] (W) to (V) node[above]{Eff};

\coordinate (V) at (-1/4,-5/8);
\node [opacity=\XA]at (V) {$\bullet$};
\draw [add= 0 and 2] (W) to (V) node[above]{Nef};

\coordinate (V) at (-1,.5);
\node [opacity=\XA] at (V) {$\bullet$};
\draw [add= 0 and 1.2][opacity=\XA] (W) to (V);

\coordinate (V) at (-1,-.5);
\node [opacity=\XA]at (V) {$\bullet$};
\draw [add= 0 and 1.5][opacity=\XA] (W) to (V);

\coordinate (V) at (-.5,-.25);
\node [opacity=\XA] at (V) {$\bullet$};
\draw [add= 0 and 1.5][opacity=\XA] (W) to (V);

\coordinate (V) at (-.5,-.75);
\node[opacity=\XA]at (V) {$\bullet$};
\draw [opacity=\XA][add= 0 and 1.5] (W) to (V);

\coordinate (V) at (-.5,-1.25);
\node [opacity=\XA]at (V) {$\bullet$};
\draw [add= 0 and 2.5] [opacity=\XA] (W) to (V);

\
\coordinate (V) at (-.5,-1.75);
\node[opacity=\XA] at (V) {$\bullet$};
\draw[opacity=\XA] [add= 0 and 3.5] (W) to (V);

\coordinate (V) at (-1/3,-0.5);
\node[opacity=\XA] at (V) {$\bullet$};
\draw[opacity=\XA] [add= 0 and 1.5] (W) to (V);

\coordinate (V) at (-1/3,-5/6);
\node[opacity=\XA] at (V) {$\bullet$};
\draw[opacity=\XA] [add= 0 and 1.5] (W) to (V);

\coordinate (V) at (-1/3,-7/6);
\node[opacity=\XA] at (V) {$\bullet$};
\draw[opacity=\XA] [add= 0 and 1.5] (W) to (V);

\coordinate (V) at (-1/3,-9/6);
\node[opacity=\XA] at (V) {$\bullet$};
\draw[opacity=\XA] [add= 0 and 1.5] (W) to (V);

\coordinate (V) at (-1/3,-11/6);
\node[opacity=\XA] at (V) {$\bullet$};
\draw[opacity=\XA] [add= 0 and 1.5] (W) to (V);

\coordinate (V) at (-1/3,-13/6);
\node[opacity=\XA] at (V) {$\bullet$};
\draw[opacity=\XA] [add= 0 and 1.5] (W) to (V);

\coordinate (V) at (-1/3,-15/6);
\node[opacity=\XA]  at (V) {$\bullet$};
\draw[opacity=\XA] [add= 0 and 1.5] (W) to (V);

\coordinate (V) at (-2/3,-1/3);
\node [opacity=\XA] at (V) {$\bullet$};
\draw[opacity=\XA] [add= 0 and 1.5] (W) to (V);

\coordinate (V) at (-2/3,-2/3);
\node [][opacity=\XA] at (V) {$\bullet$};
\draw[opacity=\XA] [add= 0 and 1.5] (W) to (V);

\coordinate (V) at (-2/3,-4/3);
\node [opacity=\XA] at (V) {$\bullet$};
\draw[opacity=\XA] [add= 0 and 1.5] (W) to (V);

\coordinate (V) at (-2/3,-5/3);
\node [][opacity=\XA] at (V) {$\bullet$};
\draw[opacity=\XA] [add= 0 and 1.5] (W) to (V);

\coordinate (V) at (-4/3,1/3);
\node[opacity=\XA] at (V) {$\bullet$};
\draw[opacity=\XA] [add= 0 and 1] (W) to (V);


\coordinate (V) at (-1/4,-7/8);
\node[opacity=\XA] at (V) {$\bullet$};
\draw[opacity=\XA][add= 0 and 1.5] (W) to (V);

\coordinate (V) at (-1/4,-9/8);
\node[opacity=\XA] at (V) {$\bullet$};
\draw[opacity=\XA] [add= 0 and 1.5] (W) to (V);
\coordinate (V) at (-1/4,-11/8);
\node[opacity=\XA] at (V) {$\bullet$};
\draw[opacity=\XA] [add= 0 and 1.5] (W) to (V);
\coordinate (V) at (-1/4,-13/8);
\node[opacity=\XA] at (V) {$\bullet$};
\draw[opacity=\XA] [add= 0 and 1.5] (W) to (V);
\coordinate (V) at (-1/4,-15/8);
\node[opacity=\XA] at (V) {$\bullet$};
\draw[opacity=\XA] [add= 0 and 1.5] (W) to (V);
\coordinate (V) at (-1/4,-17/8);
\node[opacity=\XA] at (V) {$\bullet$};
\draw[opacity=\XA] [add= 0 and 1.5] (W) to (V);
\coordinate (V) at (-1/4,-19/8);
\node[opacity=\XA] at (V) {$\bullet$};
\draw[opacity=\XA] [add= 0 and 1.5] (W) to (V);
\coordinate (V) at (-1/4,-21/8);
\node[opacity=\XA] at (V) {$\bullet$};
\draw[opacity=\XA] [add= 0 and 1.5] (W) to (V);
\coordinate (V) at (-1/4,-23/8);
\node[opacity=\XA] at (V) {$\bullet$};
\draw[opacity=\XA] [add= 0 and 1.5] (W) to (V);

\coordinate (V) at (-2/4,-4/4);
\node[opacity=\XA] at (V) {$\bullet$};
\draw[opacity=\XA] [add= 0 and 1.5] (W) to (V);

\coordinate (V) at (-2/4,-6/4);
\node[opacity=\XA] at (V) {$\bullet$};
\draw[opacity=\XA] [add= 0 and 1.5] (W) to (V);

\coordinate (V) at (-2/4,-8/4);
\node[opacity=\XA] at (V) {$\bullet$};
\draw[opacity=\XA] [add= 0 and 1.5] (W) to (V);

\coordinate (V) at (-3/4,-3/8);
\node [opacity=\XA]at (V) {$\bullet$};
\draw[opacity=\XA] [add= 0 and 1.5] (W) to (V);

\coordinate (V) at (-3/4,-9/8);
\node [opacity=\XA] at (V) {};
\draw[opacity=\XA] [add= 0 and 1.5] (W) to (V);

\coordinate (V) at (-3/4,-11/8);
\node [opacity=\XA] at (V) {};
\draw[opacity=\XA] [add= 0 and 1] (W) to (V);

\coordinate (V) at (-5/4,1/8);
\node[opacity=\XA] at (V) {};
\draw[opacity=\XA] [add= 0 and 1] (W) to (V);


\coordinate (V) at (-5/5,-7/10);
\node[opacity=\XA] at (V) {$\bullet$};
\draw[opacity=\XA] [add= 0 and 1] (W) to (V);

\coordinate (V) at (-6/5,0);
\node[opacity=\XA] at (V) {$\bullet$};
\draw[opacity=\XA] [add= 0 and 1] (W) to (V);

\coordinate (V) at (-7/5,5/10);
\node[opacity=\XA] at (V) {$\bullet$};
\draw[opacity=\XA] [add= 0 and 1] (W) to (V);

\draw[->] (-4,-3.75) -- (0,-3.75) node[above right] {$w$}-- (1.5,-3.75) node[above right] {$B$};

\draw[->,opacity =0.3] (-4,0) -- (2.5,0) node[above right] {$\frac{ch_1}{ch_0}$};

\draw[->][] (0,-4.25)-- (0,0) node [above right] {O} --  (0,6) node[right] {$H$};

\draw[->,opacity=0.3] (0,-4.25)-- (0,0) node [above right] {O} --  (0,4) node[right] {$\frac{ch_2}{ch_0}$};

\draw [thick](-3,4.5) parabola bend (0,0) (1.5,1.125);

\end{tikzpicture}

\end{center}

As two quick applications of Theorem \ref{thm:main3inintro}, we decide the boundary (on the primitive side) of the nef cone and the movable cone  of $\mathfrak M^s_{GM}(w)$ for a primitive character $w=(\ch_0,\ch_1,\ch_2)$. The other side is dually decided by $w'=(\ch_0,-\ch_1,\ch_2)$.

\begin{theorem}[Theorem \ref{thm:movcone}, the movable cone]
Let $w$ be a primitive Chern character with $\ch_0(w)\geq 0$ not inside the Le Potier cone. When $\chi(E,w)\neq 0$ for any exceptional bundle $E$ with $\frac{\ch_1(E)}{\ch_0(E)}<\frac{\ch_1(w)}{\ch_0(w)}$, the movable cone boundary on the primitive side coincides with the effective cone boundary.

When $\chi(E_\gamma,w)=0$ for an exceptional bundle $E_\gamma$ with $\frac{\ch_1(E_\gamma)}{\ch_0(E_\gamma)}<\frac{\ch_1(w)}{\ch_0(w)}$,
let  $E_\alpha$, $E_\beta$, $E_\gamma$ be exceptional bundles  corresponding to dyadic numbers $\frac{p-1}{2^n}$, $\frac{p+1}{2^n}$, $\frac{p}{2^n}$ respectively, then $w$ can be uniquely written as $n_2e_\alpha-n_1e_{\beta-3}$ for some positive integers $n_1$, $n_2$. Define the character $P$ accordingly as follows:
\begin{enumerate}
\item $P:=e_\gamma - (3\ch_0(E_\beta)-n_2)e_\alpha$, if $1 \leq n_2 < 3 \ch_0(E_\beta)$;
\item $P:=e_\gamma$, if $n_2 \geq 3 \ch_0(E_\beta)$.
\end{enumerate}
Then the wall $L_{Pw}$ corresponds to the boundary of the movable cone of $\mathfrak M^s_{GM}(w)$.
\label{thm:thm4inintro}
\end{theorem}

\begin{theorem}[Theorem \ref{thm:nefcone}, the nef cone]
Let $w$ be a primitive Chern character with $\ch_0(w) > 0$ and $\bd(w) \geq 10$, then the first actual wall to the left of vertical wall (i.e. \textit{nef cone boundary} for $\mathfrak M_{GM}^{s}(w)$) is the first lower rank wall $L_{vw}$ such that
\begin{enumerate}
\item $\frac{\ch_1(v)}{\ch_0(v)}$ is the greatest rational number less than $\frac{\ch_1(w)}{\ch_0(w)}$ with $0<\ch_0(v)\leq \ch_0(w)$;
\item given the first condition, if $\ch_1(v)$ is even (odd resp.), let $\ch_2(v)$ be the greatest integer ($2\ch_2(v)$ be the greatest odd integer resp.), such that the point $v$ is either an exceptional character or not inside $\mathrm{Cone}_{LP}$.
\end{enumerate}
\label{thm:nefinintro}
\end{theorem}

The result on the nef cone is not hard to see from the $\cccp$ picture. First of all, when the Chern character $w$ has certain distance from  Geo$_{LP}$, the first wall is not of higher rank. On the $\cccp$,  although the Le Potier curve for the stable objects is zigzag, when the Chern character $w$ is not very close to Geo$_{LP}$, the first wall $L_{wv}$ is still given by a point $v$ with $\frac{\ch_1}{\ch_0}$-coordinate closest to the vertical wall.

The result on the movable cone is more subtle. When the Chern character $w$ is right orthogonal to an exceptional bundle $E_\gamma$, the jumping locus 
\[\{[F]\in \mathfrak M^s_{GM}(w)\, |\,\Hom(E_\gamma,F) \neq 0\}\]
has codimension $1$ and is the exceptional divisor that contracted on the movable cone boundary. However, the wall $L_{we_\gamma}$ for $\Hom(E_\gamma,F)\neq 0$ may not always be the wall for this contraction. In the case when $n_2 < 3 \ch_0(E_\beta)$, the exceptional divisor is already contracted at a wall prior to the wall $L_{we_\gamma}$. One simple example of such $w$ is when $(\ch_0,\ch_1,\ch_2)$ $=$ $(1,0,-4)$, in other words, the ideal sheaf of four points. The Chern character $w$ is right orthogonal to the cotangent bundle $\Omega$, whose dyadic number is $-\frac{3}{2}$. The other exceptional bundles $E_\alpha$ and $E_\beta$ are $\mathcal O(-2)$ and $\mathcal O(-1)$ respectively, and $w$ can be written as $2[\mathcal O(-2)]-[\mathcal O(-4)]$. The jumping locus of $\Hom(\Omega,w)\neq 0$  is the exceptional divisor, and it is the same as the jumping locus $\Hom(\mathcal I_1(-1),w)\neq 0$, where $\mathcal I_1(-1)$ stands for the ideal sheaf of one point tensor $\mathcal O(-1)$. Since the wall $L_{\mathcal I_1(-1)w}$ is between $L_{\Omega w}$ and $L_{w\pm}$ in the $\cccp$, the boundary of movable cone should be given by $L_{\mathcal I_1(-1)w}$. Geometrically, the exceptional locus is where any three points are collinear.\\

\textbf{Related Work.} There are several papers \cite{ABCH, BMW, CH, CH2, CH3, CHW, LZ, Woolf} studying the birational geometry of moduli of sheaves on the projective plane via wall crossing.

The 
study for Hilbert schemes of points on $\pp$ first appears in \cite{ABCH}, and the wall crossing behavior is explicitly carried out for small numbers of points. It is also firstly suggested in \cite{ABCH} that there is a correspondence between the wall crossing picture in the Bridgeland space and the minimal model program of the moduli space. In \cite{CH}, the correspondence between walls in the Bridgeland space and stable locus decomposition walls in MMP is established for monomial schemes in the plane. In \cite{LZ}, we proved the full correspondence for Hilbert schemes of points, by establishing similar results as in Section 2 of this paper, and further generalize this correspondence to deformations of Hilbert schemes, or Hilbert schemes of non-commutative projective planes.

For moduli of torsion sheaves, the effective cone and the nef cone are computed in \cite{Woolf}. For general moduli of sheaves on $\pp$, the theory is built up in \cite{BMW}. Among other results, the projectivity of moduli of Bridgeland stable objects is proved in \cite{BMW}. The effective cone and the ample cone are computed in \cite{ CH2, CHW} respectively. Also \cite{CHW} gives the criteria on when the movable cone coincides with the effective cone. We refer to the beautiful lecture notes \cite{CH3} for details of these results.

Compared with these papers, the smoothness and irreducibility of moduli of Bridgeland stable objects with primitive characters are only proved in this paper, which enables us to deduce the equivalence between wall crossing and MMP for moduli of sheaves on $\pp$ suggested in \cite{ABCH}. Our result on the effective cone is essentially equivalent to that in \cite{CHW}, however, the proof is very different. Since this proof is very closely related to the proof of the criteria on actual walls, we decide to include it here. The numerical criteria on actual walls and the result on the movable cone are  new. Our result on the nef cone (Theorem \ref{thm:nefinintro}) follows from our numerical criteria. The nef cone was first proved in \cite{CH2} when $\Delta$ is large enough with respect to $\ch_0$ and $\frac{\ch_1}{\ch_0}$ (see Remark 8.7 in \cite{CH2} for a lower bound), the bound in Theorem \ref{thm:nefinintro} is explicitly given by $\bd \geq 10$. Moreover, as a benefit of our set-up, in a large part of the paper we can treat the torsion case and the positive rank case uniformly. We make careful remarks on this through the paper.

Another important application of the wall-crossing machinery is towards the Le Potier strange duality conjecture. A special case is studied in \cite{Abe}.\\

\textbf{Organization.} In Section \ref{sec1.1}, we review some classical work by Drezet and Le Potier for stable sheaves on the projective plane. We prove some useful lemmas by visualizing the geometric stability conditions in the $\cccp$ in Section \ref{sec1.3}. These properties will play crucial roles for the arguments in the paper. In Section \ref{sec2}, we prove that the moduli space $\mss (w)$ is smooth and irreducible for generic $\sigma$ and primitive $w$. In this way, one can run the minimal model program for $\mgs(w)$ on the $\cccp$. In Section \ref{sec3}, we first compute the last wall, and then prove the main theory in the paper: a criteria for actual walls of $\mss (w)$. In Section \ref{sec4}, we compute the nef and movable cone boundary as an application of the criteria for actual walls. Moreover, in Section \ref{sec4.4}, we work out a particular example for the character $(4,0,-15)$.\\

\textbf{Acknowledgments.} The authors are greatly indebted to Arend Bayer, who offered tremendous assistance during the
preparation of this work.  We are grateful to Aaron Bertram, Izzet
Coskun,  Zheng Hua, Jack Huizenga, Wanmin Liu, Emanuele Macr\`{\i}, Matthew Woolf and Ziyu Zhang for helpful conversations. We also had useful discussions with our advisors Herbert Clemens, Thomas Nevins and Karen Smith, and we would like to thank all of them. The author CL is supported by ERC starting grant no. 337039 ``WallXBirGeom''.

\section{Stability conditions on $D^b(\textbf P^2)$}\label{sec1}

In this section, we will recall some properties of the bounded derived category of coherent sheaves on the projective plane, and the construction of stability conditions on it. In Section \ref{sec1.1}, we will explain the structure of D$^b(\textbf P^2)$ given by exceptional triples, and the numerical criteria on the existence of stable sheaves. A slice of the space of geometric stability conditions is discussed in Section \ref{sec1.2}, and the wall-chamber structure on it is studied in Section \ref{sec1.3}. In Section \ref{sec1.4}, we study the algebraic stability conditions, i.e. the stability conditions given by the exceptional triples. We also explain how they are glued to the slice of geometric stability conditions. In Section \ref{sec1.5}, we explain in detail the difference and advantage of our set-up over the one used in other papers. Finally in Section \ref{sec1.6}, we derive some easy numerical conditions on the existence of stable objects.

\subsection{Review and notations: Exceptional objects, triples and the Le Potier curve}\label{sec1.1}
Let $\mathcal T$ be a $\mathbb C$-linear triangulated category of
finite type. In this article, $\mathcal T$ will always be
D$^b(\textbf P^2)$: the bounded derived category of coherent sheaves
on the projective plane over $\mathbb C$. We first recall the following definitions
from \cite{AKO, GorRu, Or}.
\begin{defn}
An object $E$ in $\mathcal T$ is called \emph{exceptional} if
\begin{center} $\Hom(E,E[i]) = 0,$ \text{ for } $i\neq 0;$
$\Hom(E,E)=\mathbb C$.\end{center} An ordered collection of
exceptional objects $\mathcal E = \{E_0,\dots,E_m\}$ is called an
\emph{exceptional collection} if
\begin{center}
$\Hom(E_i,E_j[k])=0$, for $i>j$, any $k$.
\end{center}
\end{defn}
\begin{defn}
Let $\mathcal E$ $=$ $\{ E_0,\dots,E_n\}$ be an exceptional
collection. We say this collection $\mathcal E$ is is \emph{strong}, if
\[\Hom(E_i,E_j[q]) = 0,\] for  all $i$, $j$ and $q\neq 0$.
This collection $\mathcal E$ is called \emph{full}, if $\mathcal E$
generates $\mathcal T$ under homological shifts, cones and direct
sums.
\end{defn}


An exceptional coherent sheaf on $\pp$ is locally free since it is rigid.
We summarize some results on the classification of exceptional
bundles on \textbf P$^2$ and introduce some notations, for details we refer
to \cite{DP, GorRu, LeP}. 

The Picard group of \textbf P$^2$ is of rank one with generator $H$ $=$ $[\mathcal O(1)]$, and we will, by abuse of notation, identify the $i$-th Chern character $\ch_i$ with its degree $H^{2-i}\ch_i$. There is a one-to-one correspondence
between exceptional bundles and dyadic integers, $\frac{p}{2^m}$, with integer $p$ and non-negative integer $m$. Denote the exceptional bundle corresponding to $\frac{p}{2^m}$ by $E_{\left(\frac{p}{2^m}\right)}$.  We write Chern characters of $E_{\left(\frac{p}{2^m}\right)}$ as \[\tilde{v}\left(\frac{p}{2^m}\right):=\tilde{v}\left(\Epm\right)=
\left(\ch_0(\Epm),\ch_1(\Epm),\ch_2(\Epm)\right).\]
They are inductively (on $m$) given by the formulas:
\begin{itemize}
\item $\tilde{v}(n)$ $=$ $(1,n,\frac{n^2}{2})$, for $n\in \mathbb Z$.
\item When $m>0$ and $p\equiv 3 $(mod $4$), the Chern character is given
by
\[\tilde{v}\left(\frac{p}{2^m}\right) = 3\ch_0\left(E_{\left(\frac{p+1}{2^m}\right)}\right)\tilde{v}\left(\frac{p-1}{2^m}\right)-\tilde{v}\left(\frac{p-3}{2^m}\right).\]
\item When $m>0$ and $p\equiv 1 $(mod $4$), the character is given
by
\[\tilde{v}\left(\frac{p}{2^m}\right) = 3\ch_0\left(E_{\left(\frac{p-1}{2^m}\right)}\right)\tilde{v}\left(\frac{p+1}{2^m}\right)-\tilde{v}\left(\frac{p+3}{2^m}\right).\]
\end{itemize}
\begin{rem} Here are some observations from the definition.
\begin{enumerate}
\item $\tilde{v}(p)$ is the character of the line bundle $\mathcal
O(p)=E_{(p)}$.
\item $\tilde{v}(\frac{3}{2})$ is the character of the tangent
sheaf $\mathcal T_{\pp}=E_{\left(\frac{3}{2}\right)}$.
\item The exceptional bundle
$E_{\left(\frac{p}{2^m}+1\right)}$ is
$E_{\left(\frac{p}{2^m}\right)}\otimes \mathcal O(1)$.
\item $\frac{\ch_1(E_{(a)})}{\ch_0(E_{(a)})}<\frac{\ch_1(E_{(b)})}{\ch_0(E_{(b)})}$ if and only if $a<b$.
\end{enumerate}
\end{rem}

For the rest of this section, we recall the construction of the Le Potier curve $C_{LP}$, which is greatly related to the existence of semistable sheaves.

The Grothendieck group K$(\pp)$ has rank $3$. We denote $\mathrm{K}(\pp)\otimes \mathbb R$ by K$_{\mathbb R}(\pp)$. Consider the real projective space $\pkp$ with homogeneous coordinate $[\ch_0,\ch_1,\ch_2]$, we view the locus $\ch_0=0$ as the line at infinity. The complement forms an affine real plane, which is referred to as the $\cccp$. We call $\pkp$ the projective $\cccp$. For any object $F$ in $\mathrm D^b(\pp)$, we write 
\[\tilde{v}(F):=\big(\ch_0(F),\ch_1(F),\ch_2(F)\big)\]
as the (degrees of) Chern characters of $F$. When $\tilde{v}(F)\neq 0$, use $v(F)$ to denote the corresponding point in the projective $\cccp$. In particular, when $\ch_0(F)\neq 0$, $v(F)$ is in the $\cccp$.

\begin{rem}
In this article, in all arguments on the $\cccp$,
we assume the $\frac{\ch_1}{\ch_0}$-axis to be horizontal and the
$\frac{\ch_2}{\ch_0}$-axis to be vertical. The term `above' means `$\frac{\ch_2}{\ch_0}$ coordinate is greater than'. Other terms such as `below', `to the right' and `to the left' are understood in the similar
way.
\label{rem:cccp}
\end{rem}

 Let
$e(\frac{p}{2^m})$ be the point in the
$\cccp$ with coordinate $(1,\chy(E_{(\frac{p}{2^m})}),\che(E_{(\frac{p}{2^m})}))$.  We associate to $E_{\left(\frac{p}{2^m}\right)}$ three points
$e^+(\frac{p}{2^m})$, $e^l(\frac{p}{2^m})$ and $e^r(\frac{p}{2^m})$ in the $\cccp$. The coordinate of
$e^+(\frac{p}{2^m})$ is given by:
\[e^+\left(\frac{p}{2^m}\right)\; := \; v\left(E_{(\frac{p}{2^m})}\right)\; -
\; \left(0,0,\frac{1}{\left(\ch_0(E_{\left(\frac{p}{2^m}\right)})\right)^2}\right). \] For any real number $a$,
let $\bd_{a}$ be the parabola:
\[\left\{\left(1,\chy,\che\right)\;\text{\Large {$|$} }\; \bd:= \frac{1}{2}\left(\frac{\ch_1}{\ch_0}\right)^2-\frac{\ch_2}{\ch_0}\; = \; a\right\}\]
in the $\cccp$. The point
$e^l(\frac{p}{2^m})$ is defined to be the intersection of
$\bd_{\frac{1}{2}}$ and the line segment
$l_{e^+(\frac{p}{2^m})e(\frac{p-1}{2^m})}$, and  $e^r(\frac{p}{2^m})$ is
defined to be the intersection of $\bd_{\frac{1}{2}}$ and the line
segment $l_{e^+(\frac{p}{2^m})e(\frac{p+1}{2^m})}$.

\begin{rem}
In this paper we always use $l_{**}$ to denote a line segment and $L_{**}$ to denote a line in the $\cccp$. Let $E$ be the exceptional bundle $E_{\left(\frac{p}{2^m}\right)}$. The characters on the line $L_{e^+(\frac{p}{2^m})e^l(\frac{p}{2^m})e(\frac{p-1}{2^m})}$ satisfy the equation $\chi(E,-)=\chi(-,E(-3))=0$. Symmetrically, the line $L_{e^+(\frac{p}{2^m})e^r(\frac{p}{2^m})}$ is given by the equation $\chi(-,E)=\chi(E(3),-)=0$ in the $\cccp$.
\label{rem:leeline}
\end{rem}

In the $\cccp$, consider the open region below all the line segments $l_{e^+(\frac{p}{2^m})e^l(\frac{p}{2^m})}$,
$l_{e^r(\frac{p}{2^m})e^+(\frac{p}{2^m})}$ and the curve $\bd_{\frac{1}{2}}$. The boundary of this open region is a fractal curve consisting of line segments $l_{e^+(\frac{p}{2^m})e^l(\frac{p}{2^m})}$,
$l_{e^r(\frac{p}{2^m})e^+(\frac{p}{2^m})}$ for all dyadic numbers
$\frac{p}{2^m}$ and fractal pieces of points on $\bd_{\frac{1}{2}}$. This curve is in the region between $\bd_{\frac{1}{2}}$ and $\bd_{1}$.

\begin{defn}
The above boundary curve is called the \emph{Le
Potier curve} in the $\cccp$, and denoted by $C_{LP}$. The cone in $\mathrm{K}_{\mathbb
R}(\pp)$ spanned by the origin and $C_{LP}$ is defined to be the \emph{Le Potier cone}, denoted by $\Cone_{LP}$.

We say a character $v\in \mathrm{K}(\pp)$ is \emph{not inside} $\Cone_{LP}$ if either $\ch_0(v)\neq 0$ and the corresponding point $\tilde{v}$ is not above $C_{LP}$ in the $\cccp$; or $\ch_0(v)=0$ and $\ch_1> 0$.
\label{defn:lpcurve}
\end{defn}

\begin{rem}
The line segments $l_{e^+(\frac{p}{2^m})e^l(\frac{p}{2^m})}$,
$l_{e^r(\frac{p}{2^m})e^+(\frac{p}{2^m})}$ do not cover the whole $C_{LP}$, the complement forms a Cantor set on $\bd_{\frac{1}{2}}$.
The cartoon for $C_{LP}$ in the $\cccp$ is shown as follows.
\label{rem:lpcone}
\end{rem}


\begin{center}
\begin{tikzpicture}[domain=1:5]

\tikzset{%
    add/.style args={#1 and #2}{
        to path={%
 ($(\tikztostart)!-#1!(\tikztotarget)$)--($(\tikztotarget)!-#2!(\tikztostart)$)%
  \tikztonodes},add/.default={.2 and .2}}
}

\newcommand\XA{0.02}

\draw [name path =C0, opacity=0.1](-3,4.5) parabola bend (0,0) (3,4.5)
 node[right, opacity =0.5] {$\bd_0$};

\draw [name path = C1, opacity=0.5](-3,4) parabola bend (0,-0.5) (3,4)
 node[right] {$\bd_{\frac{1}{2}}$};

\draw [name path =C2, opacity=0.5](-3,3.5) parabola bend (0,-1) (3,3.5)
 node[right] {$\bd_1$};


\coordinate (B3) at (-3,4.5);
\coordinate (B2) at (-2,2);

\coordinate (B1) at (-1,0.5);
\coordinate (A0) at (0,0);
\coordinate (A1) at (1,0.5);
\coordinate (A2) at (2,2);
\coordinate (A3) at (3,4.5);

\coordinate (E0) at (0,-1);
\draw (E0) node [below right] {$e^+(0)$};

\coordinate (E1) at (1,-0.5);
\draw (E1) node [below right] {$e^+(1)$};

\coordinate (E2) at (2,1);
\draw (E2) node [below right] {$e^+(2)$};

\coordinate (E3) at (3,3.5);
\draw (E3) node [below right] {$e^+(3)$};

\coordinate (F1) at (-1,-0.5);
\draw (F1) node [below left] {$e^+(-1)$};

\coordinate (F2) at (-2,1);
\draw (F2) node [below left] {$e^+(-2)$};

\coordinate (F3) at (-3,3.5);
\draw (F3) node [below left] {$e^+(-3)$};

\draw [name path =L0,opacity =\XA] (E0) -- (B2);
\draw [name intersections={of=C1 and L0},  thick] (E0) -- (intersection-1);
\draw [name path =R13,opacity =\XA] (F3) -- (B2);
\draw [name intersections={of=C1 and R13},  thick] (F3) -- (intersection-1);

\draw [name path =L1,opacity =\XA] (E1) -- (B1);
\draw [name intersections={of=C1 and L1},  thick] (E1) -- (intersection-1);
\draw [name path =R12,opacity =\XA] (F2) -- (B1);
\draw [name intersections={of=C1 and R12},  thick] (F2) -- (intersection-1);

\draw [name path =L2,opacity =\XA] (E2) -- (A0);
\draw [name intersections={of=C1 and L2},  thick] (E2) -- (intersection-1);
\draw [name path =R11,opacity =\XA] (F1) -- (A0);
\draw [name intersections={of=C1 and R11},  thick] (F1) -- (intersection-1);

\draw [name path =L3,opacity =\XA] (E3) -- (A1);
\draw [name intersections={of=C1 and L3},  thick] (E3) -- (intersection-1);
\draw [name path =R0,opacity =\XA] (E0) -- (A1);
\draw [name intersections={of=C1 and R0},  thick] (E0) -- (intersection-1);

\draw [name path =L12,opacity =\XA] (F2) -- (B3);
\draw [name intersections={of=C1 and L12},  thick] (F2) -- (intersection-1);
\draw [name path =L11,opacity =\XA] (F1) -- (B2);
\draw [name intersections={of=C1 and L11},  thick] (F1) -- (intersection-1);

\draw [name path =R1,opacity =\XA] (E1) -- (A2);
\draw [name intersections={of=C1 and R1},  thick] (E1) -- (intersection-1);
\draw [name path =R2,opacity =\XA] (E2) -- (A3);
\draw [name intersections={of=C1 and R2},  thick] (E2) -- (intersection-1);

\coordinate (S3) at (-2.5,2.5);
\draw (S3) node [below left] {$e^+(-\frac{5}{2})$};

\coordinate (S2) at (-1.5,0.5);

\coordinate (S1) at (-.5,-0.5);

\coordinate (T1) at (.5,-0.5);

\coordinate (T2) at (1.5,0.5);

\coordinate (T3) at (2.5,2.5);
\draw (T3) node [below right] {$e^+(\frac{5}{2})$};

\draw [name path =RS3,opacity =\XA] (S3) -- (B3);
\draw [name intersections={of=C1 and RS3},  thick] (S3) -- (intersection-1);
\draw [name path =LS3,opacity =\XA] (S3) -- (A0);
\draw [name intersections={of=C1 and LS3},  thick] (S3) -- (intersection-1);

\draw [name path =RS2,opacity =\XA] (S2) -- (B2);
\draw [name intersections={of=C1 and RS2},  thick] (S2) -- (intersection-1);
\draw [name path =LS2,opacity =\XA] (S2) -- (A1);
\draw [name intersections={of=C1 and LS2},  thick] (S2) -- (intersection-1);

\draw [name path =RS1,opacity =\XA] (S1) -- (B3);
\draw [name intersections={of=C1 and RS1},  thick] (S1) -- (intersection-1);
\draw [name path =LS1,opacity =\XA] (S1) -- (A2);
\draw [name intersections={of=C1 and LS1},  thick] (S1) -- (intersection-1);

\draw [name path =RT1,opacity =\XA] (T1) -- (B2);
\draw [name intersections={of=C1 and RT1},  thick] (T1) -- (intersection-1);
\draw [name path =LT1,opacity =\XA] (T1) -- (A3);
\draw [name intersections={of=C1 and LT1},  thick] (T1) -- (intersection-1);

\draw [name path =RT2,opacity =\XA] (T2) -- (B1);
\draw [name intersections={of=C1 and RT2},  thick] (T2) -- (intersection-1);
\draw [name path =LT2,opacity =\XA] (T2) -- (A2);
\draw [name intersections={of=C1 and LT2},  thick] (T2) -- (intersection-1);

\draw [name path =RT3,opacity =\XA] (T3) -- (A0);
\draw [name intersections={of=C1 and RT3},  thick] (T3) -- (intersection-1);
\draw [name path =LT3,opacity =\XA] (T3) -- (A3);
\draw [name intersections={of=C1 and LT3},  thick] (T3) -- (intersection-1);


\draw[->,opacity =0.3] (-4,0) -- (4,0) node[above right] {$\frac{\ch_1}{\ch_0}$};

\draw[->,opacity=0.3] (0,-2)-- (0,0) node [above right] {O} --  (0,6) node[right] {$\frac{\ch_2}{\ch_0}$};

\end{tikzpicture}
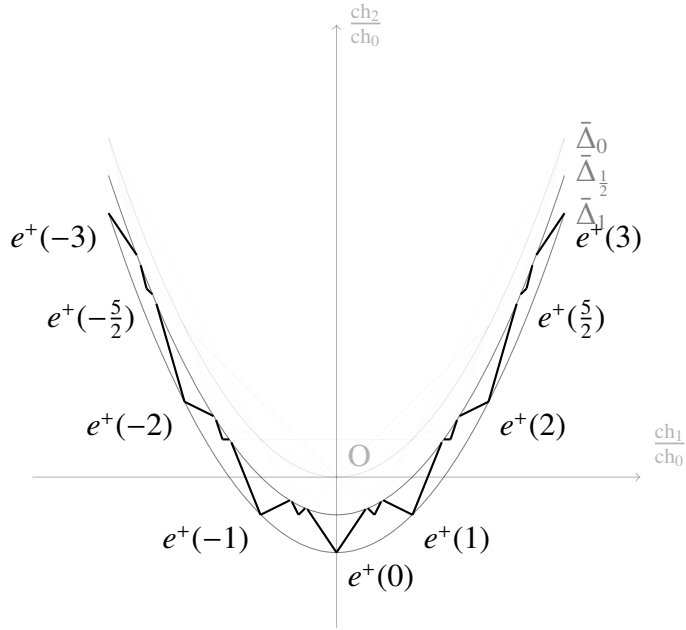
Figure: The Le Potier curve  $C_{LP}$.
\end{center}


Given the Le Potier curve, we can now state the numerical condition on the existence of stable sheaves.

\begin{theorem}[Drezet, Le Potier]
There exists a slope semistable coherent sheaf with character
$w=(\ch_0(>0),\ch_1,\ch_2)\in \mathrm K(\pp)$ if and only if one of the following two conditions holds:
\begin{enumerate}
\item $w$ is proportional to an exceptional character;
\item The point $\left(1,\frac{\ch_1}{\ch_0},\frac{\ch_2}{\ch_0}\right)$ is on or below $C_{LP}$ in the $\cccp$.
\end{enumerate}
\label{thm:lp}
\end{theorem}

\subsection{Geometric stability conditions}\label{sec1.2}
In this section, we follow \cite{BM, Bri08}
and recall that the space of geometric stability conditions on \textbf P$^2$ is
a $\glt$ principal bundle over a subspace Geo$_{LP}$ of the $\cccp$.




In applications to geometry, the following type of stability conditions are always most relevant.

\begin{defn}
A stability condition $\sigma$ on $\mathrm D^b(\pp)$ is called
\emph{geometric} if it satisfies the support property and all
skyscraper sheaves $k(x)$ are $\sigma$-stable of the same phase.
We denote the set of all geometric stability conditions by
$\mathrm{Stab}^{\Geo}(\pp)$. \label{defn:stabgeo}
\end{defn}

In order to construct geometric stability conditions, we want to first introduce the appropriate $t$-structure. Fix a real number $s$, a torsion pair of coherent sheaves on
$\pp$ is given by:
\begin{itemize}
\item[] $\Coh_{\leq s}$: subcategory of $\Coh(\pp)$ generated by
semistable sheaves of slope $\leq s$.
\item[] $\Coh_{>
s}$: subcategory of $\Coh(\pp)$ generated by semistable
sheaves of slope $> s$ and torsion sheaves.

\item[] $\Coh_{\# s}$ $:=$ $\langle\Coh_{\leq s}[1]$, $\Coh_{>
s}\rangle$.
\end{itemize}

We define the \emph{geometric area} Geo$_{LP}$ in the $\cccp$ to be the open set:
\[\mathrm{Geo}_{LP}:= \{(1,a,b)\; |\; (1,a,b) \text{ is above }C_{LP}
\text{ and not on } l_{ee^+} \text{ for any exceptional } e \}.\]

\begin{pd}
For a point $(1,s,q)$ $\in$ $\Geo_{LP}$, there exists a geometric stability condition
$\sigma_{s,q}:=(Z_{s,q},\Coh_{\#s})$ on $\mathrm D^b(\pp)$, where the central charge is given by
\[ Z_{s,q}(E):=(-\ch_2(E)+q\cdot \ch_0(E))+ i(\ch_1(E)-s\cdot \ch_0(E)).\]
In this case, $\Ker(Z_{s,q})$ consists of the characters corresponding to the point $(1,s,q)$. We write $\phi_{\sigma_{s,q}}$ or $\phi_{s,q}$ for the phase function of $\sigma_{s,q}$.
\label{defn:geostabsq}
\end{pd}
For the proof that $\sigma_{s,q}$ is indeed a geometric stability condition, we refer to \cite{BM} Corollary 4.6 and \cite{Bri08}, which also work well for \textbf P$^2$.
Here the phase function $\phi_{s,q}$ can be also defined for objects in
$\Coh_{\#s}$: \[\phi_{s,q}(E) := \frac{1}{\pi} \Arg(Z_{s,q}(E)).\] It is
well-defined in the sense that it coincides with the phase function on
$\sigma_{s,q}$-semistable objects.

\begin{rem}
The definition of  $\sigma_{s,q}$ here is different from the
usual one as that in \cite{ABCH}, which is given as $(Z'_{s,t},\mathcal P_s)$ (see Section \ref{sec1.5} for the explicit formulae). When $q>\frac{s^2}{2}$, $Z_{s,q}$ has the same kernel as that of $Z'_{s,q-\frac{s^2}{2}}$. Their formula are slightly different. The imaginary parts are the same, but the real parts differ by a multiple of the imaginary part. We would like to
use the version here because $q-\frac{s^2}{2}$ is allowed to be negative, and the kernel of the central charge on the $\cccp$ is
clearly $(1,s,q)$.
\label{rem:anglephase}
\end{rem}

\begin{rem}
Given a point $P=(1,s,q)$ in $\Geo_{LP}$, we will also write $\sigma_P$, $\phi_P$, $\Coh_{P}(\pp)$ and $Z_P$ for the stability condition $\sigma_{s,q}$, the phase function $\phi_{s,q}$, the tilt heart $\Coh_{\#s}(\pp)$ and the central charge $Z_{s,q}$ respectively. 
\label{rem:sigmap}
\end{rem}

Up to the $\glt$-action, geometric
stability conditions are all of the form given in Proposition and Definition \ref{defn:geostabsq}.

\begin{prop}[\cite{Bri08} Proposition 10.3, \cite{BM} Section 3] Let $\sigma=(Z,\mathcal P((0,1]))$ be a
geometric stability condition such that  all skyscraper sheaves $k(x)$ are contained in
$\mathcal P(1)$. Then the heart $\mathcal P((0,1])$ is $\Coh_{\#s}$
for some real number $s$.  The central charge $Z$ can be written in
the form of
\[-\ch_2+a\cdot \ch_1+b\cdot \ch_0,\] where $a, b \in \mathbb C$
satisfy the following conditions:
\begin{itemize}
\item $\Im a$ $>$ $0$, $\frac{\Im  b}{\Im a}$ $=$ $s$;
\item $(1, \frac{\Im b}{\Im a}$, $\frac{\Re a\Im b}{\Im a}$+$\Re b)$ is in $\Geo_{LP}$.
\end{itemize}
\label{thm:geostab}
\end{prop}

Thanks to the classification of characters of semistable sheaves on $\textbf{P}^2$ \cite{DP}, this property is proved in the same way as in cases of local \textbf P$^2$ \cite{BM} and K3 surfaces \cite{Bri07}. Since all discussions in this paper are invariant under the $\glt$-action, geometric stability conditions will be identified with the corresponding points in Geo$_{LP}$. We will always visualize Stab$^{\Geo}(\pp)$ as Geo$_{LP}$ in this paper.

\subsection{Potential walls and phases}\label{sec1.3}
We collect some well-known and useful results about the potential walls in this section. Since our set-up is slightly different from the usual one (see Remark \ref{rem:anglephase}), we give statements and proofs for completeness. We hope this can also illustrate the advantage of our set-up.
\begin{defn}
A stability condition is said to be \emph{non-degenerate}, if it satisfies
the support property and the image of
its central charge is not contained in any real line in $\mathbb C$. We write $\Stab^{\mathrm{nd}}$ for the space of
non-degenerate stability conditions.
\end{defn}
The kernel map for the central charges
is well-defined on $\mathrm{Stab}^{nd}$: \[\text{Ker}: \text{Stab}^{\mathrm{nd}}\rightarrow
\mathrm P_{\mathbb R}\left(\mathrm{K}_{\mathbb R}(\mathbf P^2)\right).\]
\begin{lemma}[\cite{Bri07}]
$\glt$ acts freely on $\mathrm{Stab}^{\mathrm{nd}}$ with
closed orbits, and 
\[\mathrm{Ker}: \mathrm{Stab}^{\mathrm{nd}}/\glt\rightarrow\mathrm P_{\mathbb R}\left(\mathrm K_{\mathbb R}(\mathbf P^2)\right)\]
is a local homeomorphism. \label{lemma:localhomeo}
\end{lemma}
\begin{proof}
By \cite{Bri07}, Stab$^{\mathrm{nd}}$ $\rightarrow$ $\Hom_{\mathbb
Z}(\mathrm K(\pp),\mathbb C)$ is a local homeomorphism, whose image lies
in the subspace of non-degenerate morphisms in $\Hom_{\mathbb Z}(\mathrm K(\pp),\mathbb C)$. When taking the quotient by GL$^+(2,\mathbb R)$, $\Hom^{\mathrm{nd}}_{\mathbb Z}(\mathrm K(\pp),\mathbb
C)/$GL$^+(2,\mathbb R)$ can be identified with the quotient Grassmannian Gr$_2(3)\cong \mathrm P_{\mathbb R}\big(\mathrm K_{\mathbb R}(\pp)\big)$
as a topological space. The statement clearly follows.
\end{proof}

We have the following description of the potential wall, i.e. the locus of stability conditions for which two given characters are of the same slope.

\begin{lemma}[Potential walls]
 Let $P=(1,s,q)$ be a point in $\Geo_{LP}$; $E$ and $F$ be two objects in $\mathrm{Coh}_{P}(\pp)$ such that their Chern characters $v$ and $w$ are not zero, then
\begin{center}
$Z_P(E)$ and  $Z_P(F)$ are on the same ray
\end{center}
if and only if $v$, $w$ and $P$ are collinear in the projective $\cccp$.
\label{lemma:paraandspanplane}\end{lemma}
\begin{proof}
$Z(v)$ and $Z(w)$ are on the same ray if and only if $Z(av-bw)$  $=$ $0$ for some
$a,b\in\mathbb R_+$. This happens only when $v$, $w$ and $\mathrm{Ker}Z$ are collinear in the projective $\cccp$.
\end{proof}

Note that this statement holds even when $v$, $w$ are torsion, i.e. $\ch_0=0$.

We make some notations for lines and rays on the (projective) $\cccp$. Consider objects $E$ and $F$ such that $v(E)$ and $v(F)$ are not zero, and let $\sigma_{s,q}=\sigma_P$ be a geometric stability condition. Let $L_{EF}$ be the straight line on the projective $\cccp$ across $v(E)$    and $v(F)$. $L_{EP}$, as well as $L_{E\sigma}$, is the line across $v(E)$ and $P$. $l_{EF}$, as well as $l_{E\sigma}$, are the line segments on the $\cccp$ when both $v(E)$ and $v(F)$ are not at infinity. $\mathcal H_P$ is the right half plane with either $\frac{\ch_1}{\ch_0}>s$, or $\frac{\ch_1}{\ch_0}=s$ and $\frac{\ch_2}{\ch_0}>q$. $l_{PE}^+$ is the ray along $L_{PE}$ starting from $P$ and completely contained in $\mathcal H_P$. $L_{E\pm}$ is the vertical wall $L_{E(0,0,1)}$. $l_{E+}$ is the vertical ray along $L_{E(0,0,1)}$ from $E$ going upward. $l_{E-}$ is the vertical ray along $L_{E(0,0,-1)}$ from $E$ going downward.

The following lemma translates the comparison of slopes into a geometric comparison of the positions of two rays. This simplifies a lot of computations and will be used throughout the paper.

\begin{lemma}
Let $P = (1,s,q)$ be a point in $\mathrm{Geo}_{LP}$, $E$ and
$F$ be two objects in $\Coh_{\#s}$. The inequality
\[\phi_{s,q}(E)>\phi_{s,q}(F)\] holds if and only if the ray $l^+_{PE}$ is
above $l^+_{PF}$. \label{lemma:slopecompare}
\end{lemma}
\begin{proof}
By the formula of $Z_{s,q}$, the angle between the
rays $l^+_{PE}$ and $l_{P-}$ at the point $P$ is
$\pi\phi_{s,q}(E)$. The statement follows from this observation.
\end{proof}


\begin{center}
\tikzset{%
    add/.style args={#1 and #2}{
        to path={%
 ($(\tikztostart)!-#1!(\tikztotarget)$)--($(\tikztotarget)!-#2!(\tikztostart)$)%
  \tikztonodes},add/.default={.2 and .2}}
}

\begin{tikzpicture}[domain=2:1]
\newcommand\XA{0.1}
\newcommand\obe{-0.3}

\coordinate (E) at (-3,-2);
\node  at (E) {$\bullet$};
\node [below] at (E) {$E$};

\coordinate (F) at (3,-1);
\node [left] at (F) {$F$};
\node at (F) {$\bullet$};

\coordinate (P) at (0.5,1.5);
\node [above] at (P) {$\nu_P$};
\node at (P) {$\bullet$};

\draw [dashed] (P) -- (F);
\node at (2,0.5) {$l^+_{PF}$};

\draw  (P) -- (E);
\draw [add =0 and 1,dashed] (E) to (P) node [right]{$l^+_{PE}$};




\draw[->] [opacity=\XA] (-4,0) -- (,0) node[above right] {$w$}-- (4,0) node[above right, opacity =1] {$\frac{\ch_1}{\ch_0}$};

\draw[->][opacity=\XA] (0,-2.5)-- (0,0) node [above right] {O} --  (0,5) node[right, opacity=1] {$\frac{\ch_2}{\ch_0}$};

\draw [thick](-3,4.5) parabola bend (0,0) (3,4.5) node [above, opacity =1] {$\bd=0$};

\end{tikzpicture} \\
Figure: comparing the slopes at $P$.
\end{center}


An important problem is to study the existence of stable objects with respect to given stability condition and character. This will be solved in several steps in this paper. Now we can make the first observation.

\begin{prop}
Let $E \in \Coh_{\#s}$ be a $\sigma_{s,q}$-stable object, then one of the
following cases happens:
\begin{enumerate}
\item The character $\tilde{v}(E)$ is not in the cone spanned by $\Geo_{LP}$ and the origin.
\item There exists a slope semistable sheaf $F$ such that in the $\cccp$ the
point $v(F)$ is above $L_{EP}$ and  between the vertical walls $l_{E+}$ and $l_{P+}$.
\end{enumerate}
In either case, the line segment \textbf{\emph{$l_{EP}$ is not entirely contained inside}}
$\Geo_{LP}$. In particular, at least one of $v(E)$ and $(1,s,q)$ is
outside the negative discriminant area $\bd_{<0}$.
\label{prop:epintersectsgeolp}
\end{prop}

\begin{proof}
Assume that Case 1 does not happen, i.e. $\tilde{v}(E)$ is inside the Geo$_{LP}$-cone, we need to show case 2 happens. In
particular, $\ch_0(E)$ is not $0$. When $\ch_0(E)>0$, H$^0(E)$ is non-zero. Let $F=$H$^0(E)_{\min}$ be the
quotient sheaf of H$^0(E)$ with the minimum slope $\frac{\ch_1}{\ch_0}$. Then $F$ is a slope semistable sheaf, so $v(F)$ is outside Geo$_{LP}$. Let $D$ be
H$^{-1}(E)$ and $G$ be the kernel of H$^0(E)\rightarrow F$. We may compare the slopes of $E$ and $F$
\[\frac{\ch_1(E)}{\ch_0(E)}\; = \; \frac{\ch_1(F)+\ch_1(G)-\ch_1(D)}{\ch_0(F)+\ch_0(G)-\ch_0(D)}\; \geq\; \frac{\ch_1(F)}{\ch_0(F)}.\]

The inequality holds because when $D$ and $G$ are non-zero, we have
\[\frac{\ch_1(D)}{\ch_0(D)}<\frac{\ch_1(F)}{\ch
_0(F)}<\frac{\ch_1(G)}{\ch_0(G)}.\]
Note here the equality
\[\frac{\ch_1(E)}{\ch_0(E)}\; = \; \frac{\ch_1(F)}{\ch_0(F)}\]
holds only when $D$ and $G$ are both zero. In this case, $v(E)$ $=$ $v(F)$, hence $v(F)$ is inside Geo$_{LP}$, which contradicts to our assumption. Therefore, we have a strict inequality, i.e. $v(F)$ is to the left of $v(E)$. As $F$ $\in$ $\Coh_{>
s}$, $P$ is to the left of $v(F)$. In addition, as
$\phi_{s,q}(E)$ $<$ $\phi_{s,q}(F)$, by Lemma
\ref{lemma:slopecompare}, $F$ is above $l_{PE}$, so case 2 happens.

When $\ch_0(E)<0$, let $F=$H$^{-1}(E)_{max}$ be the subsheaf of
$\HH^{-1}(E)$ with the maximum slope $\frac{\ch_1}{\ch_0}$. By the same argument, $v(F)$
is to the right of $v(E)$. As $F$ $\in$ $\Coh_{\leq s}$, $v$ is to
the left  of $L_{P\pm}$ or on the ray $l_{P-}$. In
addition, since $\phi_{s,q}(F[1])$ $<$ $\phi_{s,q}(E)$, by Lemma
\ref{lemma:slopecompare}, $F$ is above $l_{EP}$. As $l_{F-}$ does
not intersect Geo$_{LP}$, $F$ is to the left of $L_{P\pm}$.\\

For the last statement, since the region
$\bd_{<0}$  is convex, for any $v(E)$ and $P=(1,s,q)$ that are
both in $\bd_{<0}$, the line segment $l_{EP}$ is also in $\bd_{<0}$, which is
contained in Geo$_{LP}$.
\end{proof}

This induces some useful corollaries. First we get the stability of exceptional bundles for some stability conditions.

\begin{cor}
Let $E$ be an exceptional bundle, and $P=(1,s,q)$
be a point in $\Geo_{LP}$, then  $E$ is $\sigma_{s,q}$-stable if $s <
\chy(E)$ and $l_{EP}$ is contained in $\Geo_{LP}$ (not include the endpoints). In the homological shifted
case, $E[1]$ is $\sigma_{s,q}$-stable, if $s \geq \chy(E)$ and
$l_{EP}$ is contained in $\Geo_{LP}$. \label{cor:regionofEstab}
\end{cor}
\begin{proof}
We will prove the first statement. If $E$ is not $\sigma_{s,q}$-stable, then there
is a $\sigma_{s,q}$-stable object $F$ destabilizing $E$. We have the
exact sequence
\[0\rightarrow \HH^{-1}(F)\rightarrow \HH^{-1}(E)\rightarrow
\HH^{-1}(E/F)\rightarrow \HH^0(F)\rightarrow \HH^0(E)\rightarrow
\HH^0(E/F)\rightarrow 0.\]
Since $\HH^{-1}(E)=0$, we see that $\HH^{-1}(F)=0$ and $v(F)$ lies between the vertical lines
$L_{P\pm}$ and $L_{E\pm}$. Since $\phi_{s,q}(F)>\phi_{s,q}(E)$, by Lemma \ref{lemma:slopecompare}, $v(F)$ is in the
region bounded by $l_{P+}$, $l_{PE}$ and $l_{E+}$. As $l_{EP}$ is contained in
Geo$_{LP}$, $l_{FP}$ is also contained in Geo$_{LP}$,. By Proposition
\ref{prop:epintersectsgeolp}, $F$ is not $\sigma_{s,q}$-stable,
which is a contradiction. The second statement can be proved similarly.
\end{proof}
\begin{rem}
The condition that `$l_{EP}$ is contained in $\mathrm{Geo}_{LP}$' is also
necessary. Any ray starting from $v(E)$ may only intersect $C_{LP}$ at most once, and only intersect with finitely many $l_{ee^+}$ segments. Suppose that $s<\chy(E)$, and $l_{EP}$ intersects some $l_{ee^+}$
segments, we may choose the one (denoted by $F$) with minimum
$\frac{\ch_1}{\ch_0}$-coordinate. The segment $l_{FP}$ is contained in
$\mathrm{Geo}_{LP}$, and $\phi_{s,q}(F)>\phi_{s,q}(E)$. By Corollary \ref{cor:regionofEstab}, $F$ is $\sigma_{s,q}$-stable. By \cite{GorRu},
$\Hom(F,E)$ $\neq$ $0$ when $\chy(F)$ $<$ $\chy(E)$. This shows that $E$ is not $\sigma_{s,q}$-stable.
\end{rem}

The second corollary roughly says when we vary the stability condition in 
$\mathrm{Geo}_{LP}$, stable objects remain stable if the slopes do not change.

\begin{cor}
Let $\sigma_{s,q}$ be a geometric stability condition and $F$ be a
$\sigma$-stable object, then for any geometric stability condition
$\tau$ on the line $L_{F\sigma}$ such that $l_{\tau\sigma}$ is
contained in $\mathrm{Geo}_{LP}$, $F$ is also
$\tau$-stable.
\label{cor:stabonthesegment}
\end{cor}

\subsection{Algebraic stability conditions}\label{sec1.4}

The structure of $\mathrm{D}^b(\pp)$ can be studied via full strong exceptional collections. First recall the following definition.

\begin{defn}
An ordered set $\mathcal E$ $=$ $\{E_1,E_2,E_3\}$ is an
\emph{exceptional triple} in $\mathrm{D}^b(\pp)$ if $\mathcal E$ is a
full strong exceptional collection of coherent sheaves in
$\mathrm D^b(\pp)$. \label{def:exceptionaltriple}
\end{defn}

\begin{rem}
The exceptional triples in $\mathrm{D}^b(\textbf P^2)$ have been classified in  \cite{GorRu} by Gorodentsev and Rudakov. In
particular, up to a cohomological shift, each collection consists of
exceptional bundles on $\pp$. In terms of dyadic numbers,
their labels are given by one of the following three cases ($p$ is an odd integer when $m\neq 0$):
\begin{equation*}
\left\{\frac{p-1}{2^m},\frac{p}{2^m},\frac{p+1}{2^m}\right\};\;
\left\{\frac{p}{2^m},\frac{p+1}{2^m},\frac{p-1}{2^m}+3\right\};\;
\left\{\frac{p+1}{2^m}-3,\frac{p-1}{2^m},\frac{p}{2^m}\right\}.
\tag{$\clubsuit$} \label{eq:dyadictriples}
\end{equation*}
\label{rem:exctrip}
\end{rem}

We recall the construction of algebraic stability
conditions associated to an exceptional triple.
\begin{prop} [\cite{Mac} Section 3]
Let $\mathcal E$ be an exceptional triple in $\mathrm D^b(\pp)$. For
any positive real numbers $m_1$, $m_2$, $m_3$ and real numbers
$\phi_1$, $\phi_2$, $\phi_3$ such that:
\[\phi_1<\phi_2<\phi_3,\text{ and }\phi_1+1<\phi_3,\] there exists a
unique stability condition $\sigma$ $=$ $(Z,\mathcal P)$ such that
\begin{enumerate}
\item $E_j$'s are $\sigma$-stable of phase $\phi_j$;
\item $Z(E_j)=m_je^{\pi i\phi_j}$.
\end{enumerate}
\label{prop:thetaE}
\end{prop}

\begin{defn}
Let $\mathcal E$ be  an exceptional triple $\{E_1,E_2,E_3\}$ in
$\mathrm D^b(\pp)$, we write $\mathcal A_{\mathcal E}$ for the heart   $\langle E_1[2]$, $E_2[1]$, $E_3\rangle$, and $\Theta_{\mathcal E}$ for the space of all
stability conditions in Proposition \ref{prop:thetaE}.
$\Theta_{\mathcal E}$ is parametrized by
\[\{(m_1,m_2,m_3,\phi_1,\phi_2,\phi_3)\in (\mathbb R_{>0})^3\times
\mathbb R^3\text{ $\big |$
}\phi_1<\phi_2<\phi_3,\phi_1+1<\phi_3\}.\] We consider the following two subsets of $\Theta_{\mathcal E}$.
\begin{itemize}
\item $\Theta^{\triangledown}_{\mathcal E}:=$ $\{ \sigma \in \Theta_{\mathcal E}$ $|$ $\phi_2-\phi_1<1$, $\phi_3-\phi_2< 1\}$;
\item $\Theta^{\Geo}_{\mathcal E}:=$ $\Theta_{\mathcal E}\cap $
$\Stab^{\Geo}$;
\end{itemize}
We denote $\mathrm{Stab}^{Alg}$ as the union of $\Theta_{\mathcal E}$ for all
exceptional triples in $\mathrm D^b(\pp)$. A stability condition in $\mathrm{Stab}^{Alg}$ is called an \emph{algebraic
stability condition}. \label{def: theta E}
\end{defn}

Let TR$_{\mathcal E}$ be the inner points in the triangle bounded by  $l_{e_1e_2}$, $l_{e_2e_3}$ and $l_{e_3e_1}$ in the $\cccp$, for $1\leq i < j\leq 3$. Let $e^*_i$ be the points associated to $e_i$ defined in the first section, where $i=1,2,3$ and $*$ could be $+,$
$l$, or $r$. The points $e^+_1,e^r_1,e_2,e_3$ are on the line
$\chi(-,E_1)=0$, and $e^+_3,e^l_3,e_2,e_1$ are on the line
$\chi(E_3,-)=0$. Let MZ$_\mathcal E$ be the inner points of the region bounded by the line segments $l_{e_1e^+_1}$, $l_{e^+_1e_2}$,$l_{e_2e^+_3}$, $l_{e^+_3e_3}$ and $l_{e_3e_1}$.


\begin{center}
\begin{tikzpicture}[domain=1:5]

\tikzset{%
    add/.style args={#1 and #2}{
        to path={%
 ($(\tikztostart)!-#1!(\tikztotarget)$)--($(\tikztotarget)!-#2!(\tikztostart)$)%
  \tikztonodes},add/.default={.2 and .2}}
}

\newcommand\XA{0.02}

\draw [name path =C0, opacity=0.1](-1,0.5) parabola bend (0,0) (2.5,3.125)
 node[right, opacity =0.5] {$\bd_0$};

\coordinate (E1) at (0,0);
\draw (E1) node {$\bullet$} node [above left] {$e_1$};

\coordinate (E2) at (1,0.5);
\draw (E2) node {$\bullet$} node [above] {$e_2$};

\coordinate (E3) at (2,2);
\draw (E3) node {$\bullet$} node [above] {$e_3$};

\coordinate (F1) at (0,-1);
\draw (F1) node {$\bullet$} node [below left] {$e^+_1$};


\coordinate (F3) at (2,1);
\draw (F3) node {$\bullet$}node [below right] {$e^+_3$};

\coordinate (K1) at (0.4,-0.4);
\draw (K1) node {$\bullet$} node [below right] {$e^r_1$};

\coordinate (K3) at (1.6,0.8);
\draw (K3) node {$\bullet$} node [below] {$e^l_3$};


\draw[->,opacity =0.3] (-1,0) -- (3.5,0) node[above right] {$\frac{\ch_1}{\ch_0}$};

\draw[->,opacity=0.3] (0,-1.5)-- (0,0) node [above right] {O} --  (0,3) node[right] {$\frac{\ch_2}{\ch_0}$};

\draw (E1) -- (E3);
\draw (F1) -- (E3);
\draw (E1) -- (F3);
\draw (F3) -- (E3);
\draw (E1) -- (F1);

\end{tikzpicture}
Figure: TR$_\mathcal E$ and MZ$_\mathcal E$.
\end{center}


The next proposition explains how the algebraic part $\Theta_\mathcal E$ `glues' onto the geometric part Stab$^{\Geo}$.

\begin{prop}
Let $\mathcal E$ be an exceptional triple, then we have:
\begin{enumerate}
\item $\Theta^{\triangledown}_{\mathcal E} = \glt\cdot \{\sigma_{s,q}\in$ $\Stab^{\Geo}(\pp)\; |\; (1,s,q) \in \mathrm{TR}_{\mathcal E}\}$.
\item $\Theta^{\Geo}_{\mathcal E} = \glt\cdot \{\sigma_{s,q}\in$ $\Stab^{\Geo}(\pp) \; | \; (1,s,q)\in$
$\mathrm{MZ}_{\mathcal E} \}$.
\end{enumerate}
In particular,
$\Theta^{\triangledown}_{\mathcal E}$ is contained in $\Theta^{\Geo}_{\mathcal
E}$.
\label{prop:commomareaofalggeo}
\end{prop}

\begin{proof}
We will first prove the second statement. As MZ$_{\mathcal E}$ is contained in Geo$_{LP}$, by Corollary
\ref{cor:regionofEstab}, $E_2$ or $E_2[1]$ is $\sigma_{s,q}$-stable for any
point $(1,s,q)$ in MZ$_{\mathcal E}$. As $e^+_1,e^r_1,e_2,e_3$ are
collinear on the line of $\chi(-,E_1)=0$, for any point $P$ in
MZ$_\mathcal E$, $l_{E_3P}$ is contained in Geo$_{LP}$. By Corollary
\ref{cor:regionofEstab}, $E_3$ is stable for any stability conditions in MZ$_{\mathcal E}$. For the same reason, $E_1[1]$ is stable for any stability conditions in MZ$_{\mathcal E}$.

For any $(1,s,q)$ in MZ$_{\mathcal E}$, $E_3$ and $E_1[1]$ are in
the heart $\Coh_{\#s}$. By Lemma \ref{lemma:slopecompare},
$\phi_{s,q}(E_1[1])<\phi_{s,q}(E_3)$, hence
\[\phi_{s,q}(E_3)-\phi_{s,q}(E_1) >1.\]
When $s\geq \chy(E_2)$, $E_3$ and $E_2[1]$ are in the heart
$\Coh_{\#s}$, we have \[\phi_{s,q}(E_3)-\phi_{s,q}(E_2)>0.\] As
$(1,s,q)$ is above $L_{e_1e_2}$, by Lemma \ref{lemma:slopecompare},
we also have \[\phi_{s,q}(E_2)-\phi_{s,q}(E_1)>0.\] When
$s<\chy(E_2)$, by a similar argument we have the same
inequalities for $\phi_{s,q}(E_i)$'s. By Proposition
\ref{prop:thetaE}, we get the embedding
\[\mathrm{Ker}^{-1}(\mathrm{MZ}_\mathcal E)\cap\text{Stab}^{\Geo}
\hookrightarrow \Theta_\mathcal E\cap\text{Stab}^{nd}
\xrightarrow{\text{Ker}} \mathrm P(\mathrm K_\mathbb R(\textbf P^2)).
\]
For $(1,s,q)$ outside the area MZ$_\mathcal E$, by Lemma \ref{lemma:slopecompare}, at least one of the inequalities: 
\[\phi_{s,q}(E_2)\leq \phi_{s,q}(E_1),\;\; \phi_{s,q}(E_3)\leq \phi_{s,q}(E_2),\text{ or } \phi_{s,q}(E_3)-\phi_{s,q}(E_1)\leq 1\]
holds. Hence $\sigma_{s,q}$ is not contained in $\Theta_\mathcal E$, the second statement of the proposition holds.

For statement 1, as $\phi_2-\phi_1$ is not an integer,
$\Theta^{\triangledown}_{\mathcal E}$ $\in$ Stab$^{nd}$. The image
of Ker$\big(\Theta^{\triangledown}_{\mathcal E}\big)$ is in
TR$_\mathcal E$. By the previous argument, we also have the embedding
\[\big(\mathrm{Ker}^{-1}(\mathrm{TR}_\mathcal
E)\cap\text{Stab}^{\Geo}\big)/\glt
\hookrightarrow \Theta^{\triangledown}_{\mathcal
E}/\glt \xrightarrow{\text{Ker}}
\mathrm{TR}_\mathcal E \subset \mathrm P(\mathrm K_\mathbb R(\pp)).
\]
The map Ker is a local homeomorphism and the composition is an
isomorphism. Since $\Theta^{\triangledown}_{\mathcal E}$ is path
connected, the two maps are both isomorphisms. We get the first
statement of the proposition.
\end{proof}

\subsection{Remarks on the $\cccp$}
\label{sec1.5}
In this section, we want to summarize some properties of our $\cccp$ from previous sections. The aim is to help the readers gain a better understanding, especially those who are already familiar with the classical $(s,t)$-upper half plane model.

The set-up of the space of stability conditions in the paper is different from the classical $(s,t)$-upper half plane model. Recall that we visualize a geometric stability condition as the kernel of its central charge in $\mathrm{K}(\pp)\otimes \mathbb R$.  In particular, when the central charge is non-degenerate, which is always the case for geometric stability conditions, the kernel is a straight line in $\mathrm{K}(\pp)\otimes \mathbb R$. We further take the projectivization of $\mathrm{K}(\pp)\otimes \mathbb R$, the kernel of the central charge is a point on $\pkp$. For a geometric stability condition, to satisfies the Harder-Narasimhan condition, the kernel of the central charge has to separate away from all the slope stable characters and torsion characters. In particular, the kernel can only be in the area Geo$_{LP}$ bounded by the Le Potier curve. The $\glt$ action does not affect the kernel of the central charge, and the space of the geometric stability condition is realized as a $\glt$-principal bundle over Geo$_{LP}$.

For a point  in Geo$_{LP}$ with coordinate $(1,s,q)$, we may write down a stability condition $\sigma_{s,q}=(Z_{s,q},\mathcal P_s)$ with heart $\mathcal P_s((0,1]) = $ Coh$_{\sharp s}$ and central charge as that in Proposition and Definition \ref{defn:geostabsq}:
\[Z_{s,q} = -(\ch_2-q \cdot \ch_0) +i(\ch_1-s\cdot \ch_1).\]

In many of other papers, a family of geometric stability condition is parameterized by $(s,t)$ on the upper half plane $\mathbb H$ via $\sigma'_{s,t}=(Z'_{s,t},\mathcal P_s)$ with the same heart $\mathcal P_s((0,1]) = $ Coh$_{\sharp s}$ and a different central charge
\[Z'_{s,t} = -(\ch^s_2+\frac{t^2}{2} \cdot \ch_0) +i\,t\,\ch^s_1.\]
Up to the $\glt$ action, $\sigma'_{s,t}$ is the same as $\sigma_{s,\frac{s^2+t^2}{2}}$. Note that under this correspondence, the $(s,t)$-upper half plane $\mathbb H$ is mapped to $\left\{(1,\chy,\che)|\left(\frac{\ch_1}{\ch_0}\right)^2-2\left(\frac{\ch_2}{\ch_0}\right)<0\right\}$ in the $\cccp$ in $\pkp$.

Since this different convention may upset some readers, we want to briefly illustrate some advantages of our approach, which will become more clear later in the paper. 
One most important benefit is that the characters and the stability conditions are on a same space. As seen in Section 1.3, the potential wall of $w$ and another Chern character $v$ is the straight line across these two points on the $\cccp$, or strictly speaking, the line segment  in Geo$_{LP}$. On the usual $(s,t)$ upper half, the potential wall is the semicircle with two endpoints being $L_{vw}\bigcap \bd_0$. Let $\sigma_P$ be a stability condition and $w$ be a Chern character on the $\cccp$, the argument of $Z_P(w)$  is the angle bounded by $L_{P-}$ and $l_{Pw}$. We may compare the slopes of different Chern characters by their positions on the $\cccp$ and this reduces huge amount of computations. This allows us to deal with several Chern characters and stability conditions simultaneously.

Moreover, in our set-up, the divisor cone can be identified with the $\cccp$. For a Chern character $w$, one may draw its Pic$_{\mathbb R}(\mss(w))$ as a $HB$-coordinate ($H$ vertical-axis; $B$ with slope $\frac{\ch_1}{\ch_0}$) with origin at $w$ on the $\cccp$, the actual walls are the base locus decomposition walls. The Donaldson morphism identifies $w^\perp$ with the divisor cone of $\mgs(w)$.
Let $v$ belong to $w^\perp$, then the divisor given by $v$ via the Donaldson morphism corresponds to the wall $\chi(-,v)=0$ on the $\cccp$. The cartoon for the Chern character $(4,0,-15)$ in the introduction can now be interpreted from this new viewpoint.

Another advantage of the $\cccp$ picture is that the space Geo$_{LP}$ is larger than the usual upper half plane. As explained previously, up to the $\glt$ action, Geo$_{LP}$ is the whole space of geometric stability conditions. The algebraic stability conditions (quiver regions) for exceptional triples are also easier to understand on the $\cccp$ rather than on the upper half plane. The quiver region with heart $\langle E_1[2],E_2[1],E_3\rangle$ in the $\cccp$ is the area that is below $l_{e_1e_3}$ and above $l_{e_2e_1^+}$, $l_{e_2e^+_3}$. Since Chern characters of exceptional bundles are usually not on the parabola $\bd_0$ (this is the case only for line bundles), the end points of the semicircular potential walls of them involve complicated computation. On the $(s,t)$-upper half plane, only quiver regions for heart $\langle \mathcal O(k-1)[2],\mathcal O(k)[1],\mathcal O(k+1)\rangle$ can be neatly described.  In this paper ,we need the general quiver regions (e.g. for heart $\langle \mathcal O(1)[2],\mathcal T[1],\mathcal O(2)\rangle$), which are important to decide the stable area for exceptional characters, and are useful to understand the effective and movable cone boundary of the $\mss(w)$. So the $\cccp$ seems to be a suitable choice.


\subsection{First constraint on the last wall}\label{sec1.6}
For a character, it is important to study the set of stability conditions for which there exist stable objects of the given character. We call this set the \emph{stable area} of the character. In this section, we give a first constraint on the stable area.

\begin{prop}
Let $w$ be a Chern character such that $\ch_0(w) > 0$ and $\bd(w)>0$,
and $E$ be an exceptional bundle such that $\frac{\ch_1(E)}{\ch_0(E)} < \frac{\ch_1(w)}{\ch_0(w)}$. Suppose $w$ is above the
line $L_{e^le^+}$ in the $\cccp$, then for any point $P\in\Geo_{LP}$ below $L_{wE}$ and to the
left of $L_{E\pm}$, there is no $\sigma_P$-semistable object $F$ with Chern
character $w$. \label{prop:wisnotstabbelowexc}
\end{prop}
\begin{rem}
When $\ch_0(w)=0$, there is a similar statement. The conditions are replaced by `$\ch_1(w)>0$' and `$\frac{\ch_2}{\ch_1}(w)$ is greater than the slope of $L_{e^le^+}$'. The proof is similar and left to the readers.
\label{rem:11}
\end{rem}
\begin{proof}

By the assumptions and Corollary \ref{cor:stabonthesegment}, we may assume that $P$ is in MZ$_\mathcal E$ for  an exceptional triple $\mathcal E=\{E_1,E_2,E_3\}$ such that $\chy(E_3)\leq \chy(E)$ and $e_3$ is above $l_{Pw}$. By an easy geometric property of C$_{LP}$, the character $w$ is also above $L_{e^l_3e^+_3}$. $E_3$ satisfies the assumptions, without loss of generality, we may assume that $E_3=E$.

We argue by contradiction. Assume $F$ is a $\sigma_P$-stable object with Chern character $w$. As $\sigma_P$ is below $L_{wE}$, by Lemma \ref{lemma:slopecompare}, we
have
\[\phi_P(F)<\phi_P(E).\]
Since $E$ and $F$ are both $\sigma_P$-semistable, we have
\[\Hom(E,F)=0.\]
On the other hand, since $P$ is in MZ$_\mathcal E$, it is to the right of $L_{E(-3)\pm}$.
Therefore, $E(-3)[1]$ and $F$ are in a same heart, we have
\[ (\Hom(E,F[2]))^*=\Hom(F,E(-3))=\Hom(F,E(-3)[1][-1])=0.\]
The two Hom vanishings imply $\chi(E,F)\leq 0$. But by assumptions that $\ch_0(F)\geq 0$, and $w$ is above the line $L_{e^le^+}$, which is given by $\chi(E,-)=0$, we have $\chi(E,F)>0$. This leads to a contradiction.
\end{proof}

\begin{rem}
The symmetric statement for $w$ with $\ch_0(w)<0$ above $L_{e^+e^r}$
and for $E$ with larger $\frac{\ch_1}{\ch_0}$ can be proved in the same way. \label{rem:wisnotstabbelow}
\end{rem}

Now we have the following result on characters of Bridgeland stable objects. Note that this is a generalization of Theorem \ref{thm:lp} to Bridgeland stable objects.

\begin{cor}
Fix a character $w$. Suppose that there exist $\sigma_{s,q}$-semistable objects of character $w$ for
some geometric stability condition $\sigma_{s,q}$. Then $w$ either lies
not inside $\clp$ or is proportional to an exceptional character. \label{cor:nostabincone}
\end{cor}
\begin{proof}
Suppose $w$ is inside Cone$_{LP}$ and not proportional to any exceptional character, by Proposition
\ref{prop:epintersectsgeolp}, there is an exceptional character $e$ such that $e$ is above the line segment $l_{w\sigma}$ and between vertical walls $L_{w\pm}$ and $L_{\sigma\pm}$. We may
assume that $\frac{\ch_1(w)}{\ch_0(w)}> s$, then $w$ is
above $L_{e^le^+}$. Now by Proposition
\ref{prop:wisnotstabbelowexc}, since $\sigma$ is below $L_{we}$ and to
the left of $L_{e\pm}$, there is no $\sigma$-semistable object of character $w$, which is a contradiction.
\end{proof}

We also want to introduce the following important notion.

\begin{defn}
Let $L$ be a straight line in the $\cccp$.
Suppose $L$ intersects
with $\bd_{\leq 0}$ along a line segment with two endpoints: $(1,f_1,g_1)$ and $(1,f_2,g_2)$. The \emph{$\frac{\ch_1}{\ch_0}$-length} of $L\bigcap\bd_{\leq 0}$ is defined to be $|f_1-f_2|$.
\label{def:c1length}
\end{defn}

In the $(s,t)$-upper half plane model in \cite{ABCH}, the $\frac{\ch_1}{\ch_0}$-length of $L_{EF}\bigcap\bd_{\leq 0}$ is the diameter of the semicircular potential wall of $E$ and $F$. This is a measure of the size of the wall, and we have the following result, which says for walls of small length, there exists no stable object.

\begin{cor}
Let $w\in\mathrm{K}(\pp)$ be a non-zero character not inside the Le Potier cone $\Cone_{LP}$,
and $\sigma$ be a geometric stability condition inside the cone
$\bd_{<0}$. When the $\frac{\ch_1}{\ch_0}$-length of $L_{w\sigma}\bigcap\bd_{\leq 0}$
is less than or equal to $1$, there is no $\sigma$-stable object $F$
of character $w$. \label{cor:lastwallradiusgeq1}
\end{cor}
\begin{proof}
We show the case when $\ch_0(w) \geq 0$, the other case can be proved
similarly. Among all integers $k\leq \frac{\ch_1(w)}{\ch_0(w)}$, let $c$ be the largest one such that $w$ is strictly above the line $L_{\mathcal O(c-1)\mathcal O(c)}$. Note that $\mathcal O(c+1)^+$ is on the line $L_{\mathcal O(c-1)\mathcal O(c)}$, since $w$ is not inside the Le Potier cone, we have $\frac{\ch_1(w)}{\ch_0(w)} \geq c+1$. Now $w$
is not above the line $L_{\mathcal O(c)\mathcal O(c+1)}$, so the segment $L_{\mathcal O(c+1)w}\cap\bd_{\leq 0}$
has $\frac{\ch_1}{\ch_0}$-length greater than or equal to that of $L_{\mathcal O(c)\mathcal O(c+1)}$, which is $1$. By
assumption, $\sigma$ is inside the cone $\bd_{<0}$, it
must lie on or below the line $L_{\mathcal O(c+1)w}$ and to the left
of $L_{\mathcal O(c+1)\pm}$. Note that $L_{\mathcal O(c-1)\mathcal O(c)}$ is just $L_{\mathcal O(c+1)^+\mathcal O(c)^l}$, by Proposition
\ref{prop:wisnotstabbelowexc},there is no $\sigma$-stable object of character $w$.
\end{proof}
\begin{rem}
If $F$ is $\sigma$-stable, in the proof we can see that below $L_{\sigma F}$ there exist the characters of at least two line bundles $\mathcal O(c-1)$ and $\mathcal O(c)$.
\label{rmk:twobelow}
\end{rem}

\section{Wall-crossing and canonical line bundles}\label{sec2}

In this section, we prove our first main theorem: the wall crossing in stability condition space induces the MMP for moduli of sheaves on $\pp$. In Section \ref{sec2.1}, we review the construction of moduli space of semistable objects as moduli of quiver representations. In Section \ref{sec2.2}, we prove the main technical result on vanishing of certain Ext$^2$. In Section \ref{sec2.3}, the generic stability of extension objects is proved. This will be used in the proof of the irreducibility of moduli of stable objects, which occupies Section \ref{sec2.4}. We rephrase some results from variation of GIT in our situation in Section \ref{sec2.5}, and use this to prove our main theorem in Section \ref{sec2.6}.

\subsection{Construction of the moduli space}\label{sec2.1}
In this section, we review the construction of the moduli space of
$\sigma$-semistable objects on $\pp$ with a given character via the geometric
invariant theory. Let $w$ be a Chern character and $\sigma_{s,q}$ be a geometric stability condition, we write $\mathfrak M^{s(ss)}_{\sigma_{s,q}}(w)$ for the moduli space of $\sigma_{s,q}$-(semi)stable objects in $\Coh_{\# s}$ with character $w$. The line $L_{w\sigma_{s,q}}$ passes through MZ$_\mathcal E$ for some exceptional triple $\mathcal E$. We may choose a point $P$ in MZ$_\mathcal E$ for some $\mathcal E$ such that the line segment $l_{P\sigma_{s,q}}$ is contained in Geo$_{LP}$. By Corollary \ref{cor:stabonthesegment}, the moduli space $\mathfrak M^{ss}_{\sigma_{s,q}}(w)$ is the same as $\mathfrak M^{ss}_{\sigma_{P}}(w)$.

Let $\mathcal E$ be the exceptional triple consisting of $E_1$, $E_2$ and $E_3$, and let $\mathcal A_\mathcal E$ be the heart $\langle E_1[2]$, $E_2[1]$, $E_3\rangle$. We write the phase $\phi_P(E_i)$ of $E_i$ at $\sigma_P$ as $\phi_i$. By Proposition \ref{prop:commomareaofalggeo}, $\phi_1<\phi_2<\phi_3$ and $-1<\phi_1<\phi_3-1<0$. There is a real number $t$, $0<t<1$, such that $-2<\phi_1-t<-1<\phi_2-t<0<\phi_3-t<1$. Let the heart $\Coh_{P}[t]$ be generated by $\sigma_P$-stable objects with phase in $(t,t+1]$, then it contains $\sigma_P$-stable objects $E_1[2]$, $E_2[1]$ and  $E_3$. By Lemma 3.16 in \cite{Mac}, $\Coh_{P}[t]$ $=$ $\mathcal A_\mathcal E$.  For any $\sigma_P$-stable object $F$ in $\Coh_{P}$ of character $w$, the phase $\phi_P(F)$ only depends on $w$, and is denoted by $\phi_P(w)$. When $\phi_P(w)-t>0$, $F$ is an object in $\mathcal A_\mathcal E$. In particular, when $F$ is a coherent sheaf, there is a `resolution' for $F$ given as
\[0\rightarrow E_1^{\oplus n_1}\rightarrow E_2^{\oplus n_2}\rightarrow E_3^{\oplus n_3}\rightarrow F\rightarrow 0 .\]
The character $\vec {n} = (n_1,n_2,n_3)$ is the unique triple such that $n_1\tilde {v}(E_1) -n_2\tilde {v}(E_2)+n_3\tilde {v}(E_3) = w$. When $\phi_P(w)-t\leq 0$, $F[1]$ is an object in $\mathcal A_\mathcal E$. When $F$ is a coherent sheaf, it appears as the cohomological sheaf at the middle term of
\[E_1^{\oplus n_1}\rightarrow E_2^{\oplus n_2}\rightarrow E_3^{\oplus n_3}.\]
The character $\vec {n} = (n_1,n_2,n_3)$ is the unique triple such that $n_1\tilde {v}(E_1) -n_2\tilde {v}(E_2)+n_3\tilde {v}(E_3) = -w$. The following easy lemma is useful to determine whether $F$ or $F[1]$ is in $\mathcal A_\mathcal E$. 

\begin{lemma}
Let $P$ be a point in $\mathrm{MZ}_\mathcal E$, and $F$ be a $\sigma_P$-stable object in $\Coh_{\# P}$. If $L_{PF}$ is above $e_3$, then $F$ is in the heart $\mathcal A_\mathcal E$. If $L_{PF}$ is above $e_1$, then $F[1]$ is in the heart $\mathcal A_\mathcal E$.
\label{lemma:ForF1inAE}
\end{lemma}
\begin{proof}
By Lemma \ref{lemma:slopecompare}, when $L_{PF}$ is above $e_3$, we have the inequality $\phi_P(F)\geq \phi_P(E_3)$. Therefore, $\phi_P(F) - t>0$ and $F$ is in $\mathcal A_\mathcal E$. When $L_{PF}$ is above $e_1$, we have $\phi_P(F) < \phi_P(E_1[1])$. Therefore, $\phi_P(F) - t< \phi_P(E_1[1]) - t <0$ and $F$ is in $\mathcal A_\mathcal E[-1]$.
\end{proof}


\begin{center}
\begin{tikzpicture}[domain=1:5]

\tikzset{%
    add/.style args={#1 and #2}{
        to path={%
 ($(\tikztostart)!-#1!(\tikztotarget)$)--($(\tikztotarget)!-#2!(\tikztostart)$)%
  \tikztonodes},add/.default={.2 and .2}}
}

\newcommand\XA{0.02}

\draw [name path =C0, opacity=0.1](-3,4.5) parabola bend (0,0) (1,0.5)
 node[right, opacity =0.5] {$\bd_0$};

\coordinate (E1) at (-1,0.5);
\draw (E1) node {$\bullet$} node [above left] {$e_1$};

\coordinate (E2) at (-0.5,-0.25);
\draw (E2) node {$\bullet$} node [above] {$e_2$};

\coordinate (E3) at (0,0);
\draw (E3) node {$\bullet$} node [above] {$e_3$};

\coordinate (F1) at (-1,-0.5);
\draw (F1) node {$\bullet$} node [below left] {$e^+_1$};


\coordinate (F3) at (0,-1);
\draw (F3) node {$\bullet$}node [right] {$e^+_3$};

\draw (E1) -- (E3);
\draw (F1) -- (E3);
\draw (E1) -- (F3);
\draw (F3) -- (E3);
\draw (E1) -- (F1);

\coordinate (E1) at (-3,4.5);
\draw (E1) node {$\bullet$} node [above left] {$e'_1$};


\coordinate (E3) at (-2,2);
\draw (E3) node {$\bullet$} node [above] {$e'_3$};

\coordinate (F1) at (-3,3.5);
\draw (F1) node {$\bullet$} node [below left] {$e'^+_1$};

\coordinate (F2) at (-2.5,2.75);
\draw (F2) node {$\bullet$} node [below left] {$e'_2$};

\coordinate (F3) at (-2,1);
\draw (F3) node {$\bullet$}node [below left] {$e'^+_3$};


\draw[->,opacity =0.3] (-4,0) -- (1.5,0) node[above right] {$\frac{\ch_1}{\ch_0}$};

\draw[->,opacity=0.3] (0,-2)-- (0,0) node [above right] {O} --  (0,5) node[right] {$\frac{\ch_2}{\ch_0}$};

\draw (E1) -- (E3);
\draw (F1) -- (E3);
\draw (E1) -- (F3);
\draw (F3) -- (E3);
\draw (E1) -- (F1);

\coordinate (W) at (0,-1.2);
\draw (W) node {$\bullet$}node [below left] {$w$};
\coordinate (P) at (-1.5,1.5);
\draw (P) node {$\bullet$}node [above right] {$\sigma_{P'}$};

\draw [add =0 and 1.2,dashed] (W) to (P);

\end{tikzpicture}
Picture: $F[1]$ is in $\mathcal A_\mathcal E$ and $F$ is in $\mathcal A_{\mathcal E'}$
\end{center}


\begin{rem}
The case when $P$ is in $\mathrm{TR}_{\mathcal E}$ and $L_{PF}$ is below both $e_1$ and $e_3$ seems to be missing from the lemma. However, in this case, by Proposition \ref{prop:wisnotstabbelowexc}, $F$ is not $\sigma_P$-stable.
\label{rem:belowboth}
\end{rem}

We define $Q_\mathcal E$ $=(Q_0,Q_1)$ to be the quiver associated to the exceptional triple $\mathcal E$. The set $Q_0$ has three vertices $v_1$, $v_2$ and $v_3$. The arrow set $Q_1$ consists of hom($E_1,E_2$) arrows from $v_1$ to $v_2$ and hom($E_2,E_3$) arrows from $v_2$ to $v_3$. Let $\vec n$ $= (n_1,n_2,n_3)$ be a dimension character for $Q_\mathcal E$, and $H_k$ be a complex linear space of dimension $k$, then the representation space $\repq$ can be identified with
\[\{(I,J)| I\in \Hom (H_{n_1},H_{n_2})\otimes \Hom(E_1,E_2), J \in\Hom (H_{n_2},H_{n_3})\otimes \Hom(E_2,E_3)\}.\]
We denote the composition map between $E_i$'s by $\alpha_\mathcal E$:
\[\alpha_\mathcal E:\Hom(E_1,E_2) \otimes \Hom(E_2,E_3) \rightarrow \Hom(E_1,E_3).\]
This gives a relation of the quiver $Q_\mathcal E$ and we have the space of quiver representations with relation:
\[\repqa:=\{(I,J)\in \mathbf{Rep}(Q_\mathcal E,\vec  n)\;|\; J\circ I \in \Hom (H_{n_1},H_{n_3})\otimes \ker \alpha_\mathcal E\}.\]

As a subvariety of $\repq$, $\repqa$ is determined by $J I=0$, which contains $n_1n_3$hom$(E_1,E_3)$ equations.

The category $\mathcal A_\mathcal E$  is equivalent to the category of finite dimensional modules over the path algebra ($Q_\mathcal E,\alpha_\mathcal E$). Any object $F$ in $\mathcal A_\mathcal E$  with character $n_1\tilde v(E_1)-n_2\tilde v(E_2)+n_3\tilde v(E_3)$ can be written as a representation $\textbf K_F$ (unique up to the $G_{\vec n}$-action) in $\repqa$. 
\begin{defn}
Let $\mathbf K=(I,J)$ and $\mathbf K'=(I',J')$ be two objects in $\repqa$ and $\mathbf{Rep}(Q_\mathcal E,\vec n',\alpha_\mathcal E)$, respectively. We introduce notations for the following sets of homomorphisms.
\[\Hom^i(\mathbf K,\mathbf K'):=\bigoplus_j\Hom_{\mathcal O}(H_{n_j}\otimes E_j,H_{n'_{j+i}}\otimes E_{j+i}).\]
Here $H_{n_i}$ and $H_{n'_i}$ are defined to be the zero space when $i\neq 0,1,2$. The derivatives $d^0$ and $d^1$ are linear maps defined as follows:
\begin{align*}
d^0: \Hom^0(\mathbf K,\mathbf K') & \rightarrow \Hom^1(\mathbf K,\mathbf K') \\
(f_0,f_1,f_2) & \mapsto (I'\circ f_0-f_1\circ I,J'\circ f_1-f_2\circ J)\\
d^1: \Hom^1(\mathbf K,\mathbf K') & \rightarrow \Hom^2(\mathbf K,\mathbf K') \\
(g_1,g_2) & \mapsto (J'\circ g_1+g_2\circ I).
\end{align*}
\label{defn:homandd}
\end{defn}
Let $F$ and $G$ be two objects in $\mathcal A_\mathcal E$, $\mathbf K_F$ and $\mathbf K_G$ be their representations in $\Rep(Q_\mathcal E,\alpha_\mathcal E)$. The Ext$^i$ groups of $F$ and $G$ can be computed via $\mathbf K_F$ and $\mathbf K_G$.
\begin{lemma}
The $\Ext^*(F,G)$ groups are the cohomology of the complex
\[\Hom^0(\mathbf K_F,\mathbf K_G)\xrightarrow{d^0}\Hom^1(\mathbf K_F,\mathbf K_G)\xrightarrow{d^1} \Hom^2(\mathbf K_F,\mathbf K_G).\]
In particular,
\begin{align*}
\ker d^0\simeq \Hom(F,G)\\ \Hom^2(\mathbf K_F,\mathbf K_G)/\mathrm{im}\; d^1\simeq \Ext^2(F,G).
\end{align*}
\label{lemma:homandext}
\end{lemma}

Let $\vec \rho$ be a weight character for objects in $\repq$, in particular, $\vec n\cdot \vec \rho=0$. An object \textbf K in Rep$(Q,\vec n)$ is $\vec \rho$-(semi)stable if and only if for any non-zero proper sub-representation $\textbf K'$ of \textbf K with dimension character $\vec n'$ $<$ $\vec n$, we have $\vec n'\cdot \vec \rho $ $<(\leq) $ $0$.

Now we want to relate Bridgeland stability of objects to King stability of quiver representations. Let $L$ be a line on the $\cccp$ not at the infinity. Suppose $L$ intersects $l_{e_1e_3}$ for an exceptional triple $\mathcal E(=\{E_1,E_2,E_3\})$. Let $f$ be a linear function with variables $\ch_0$, $\ch_1$ and $\ch_2$ such that the zero locus of $f$ is $L$. Moreover we assume that $f(\tilde v(E_1))$ is positive. The weight character $\vec \rho_{L,\mathcal E}$ is given by
\[\left(f(\tilde v(E_1)), -f(\tilde v(E_2)),f(\tilde v(E_3))\right)\]
up to a positive scalar.


\begin{lemma}
Let $F$ be an object in $\mathcal A_{\mathcal E}$ and $P$ be a point in $\mathrm {MZ}_\mathcal E$ such that $L_{FP}$ intersects $l_{e_1e_3}$, then $F$ (or $F[-1]$) is $\sigma_P$-(semi)stable if and only if $\mathbf K_F$ is $\vec \rho_{L_{FP},\mathcal E}$-(semi)stable.
\label{lemma: rhocharformular}
\end{lemma}
\begin{proof}
First we want to modify the stability condition in a way that the central charge of the exceptional bundles are better behaved, and the weight character remains the same. Since $L_{FP}$ intersects $l_{e_1e_3}$, by Corollary \ref{cor:stabonthesegment}, we may assume $P$ is in the triangle area TR$_\mathcal E$.
The central charges of objects $E_1[2]$, $E_2[1]$ and $E_3$ are
\begin{align*}
Z_P(E_1[2])=-\ch_2(E_1)+q\,\ch_0(E_1)+(\ch_1(E_1)-s\,\ch_0(E_1))i;\\
Z_P(E_2[1])=\ch_2(E_2)-q\,\ch_0(E_2)-(\ch_1(E_2)-s\,\ch_0(E_2))i;\\
Z_P(E_3)=-\ch_2(E_3)+q\,\ch_0(E_3)+(\ch_1(E_3)-s\,\ch_0(E_3))i.
\end{align*}
There is a suitable real number $0<t<1$ such that the new central charge $Z^\bullet_P$ $:=$ $e^{i\pi t}Z_P$ maps $E_1[2]$, $E_2[1]$ and $E_3$ to the upper half plane in $\mathbb C$.

Now we can rewrite the stability condition in terms of the weight character. Write $ \vec {Z}^\bullet_P$ $:=$ $\left(Z^\bullet_P(E_1[2]),Z^\bullet_P(E_2[1]),Z^\bullet_P(E_3)\right)$  $=$ $\vec a^\bullet+\vec b^\bullet i$ for two real vectors $\vec a^\bullet$ and $\vec b^\bullet$. The object $F$ is $ Z^\bullet_P$-(semi)stable if and only if for any non-zero proper subobject $F'$ in $\mathcal A_\mathcal E$,
\[\Arg Z^\bullet_P(F') < (\leq) \Arg Z^\bullet_P(F).\]
In other words, suppose the dimension vector of $\textbf K_{F}$ is $\vec n=(n_1,n_2,n_3)$, then for any nonzero proper sub-representation $\textbf K_{F'}$ with dimension vector $\vec n'$,
\begin{equation}
\Arg\vec n'\cdot \vec {Z}^\bullet_P <(\leq) \Arg\vec n\cdot \vec {Z}^\bullet_P.
\label{eq01}
\end{equation}
Let $\vec \rho^\bullet$ be the vector
\[-(\vec b^\bullet\cdot \vec n)\vec a^\bullet +  (\vec a^\bullet\cdot \vec n)\vec b^\bullet.\]
Since each factor of $\vec b^\bullet$ is non-negative, the inequality (\ref{eq01}) holds if and only if 
\[\frac{\vec n\cdot \vec a^\bullet}{\vec n\cdot \vec b^\bullet}<(\leq)\frac{\vec n'\cdot \vec a^\bullet}{\vec n'\cdot \vec b^\bullet}\]
if and only if $\vec n'\cdot \vec \rho^\bullet< (\leq) 0$.

We can also write $ \vec {Z}_P$ $:=$ $\left(Z_P(E_1[2]),Z_P(E_2[1]),Z_P(E_3)\right)$  $=$ $\vec a+\vec b i$ for two real vectors $\vec a$ and $\vec b$. As $Z^\bullet_P$ $=$ $e^{i\pi t}Z_P$, the character $\vec \rho$ $:=$ $(\vec b\cdot \vec n)\vec a - (\vec a\cdot \vec n)\vec b$ is the same as $\vec \rho^\bullet$.

At last we need to show that $\vec \rho$ is $\vec \rho_{L_{FP},\mathcal E}$ up to a positive scalar. Let $f$ be the linear function:
\[f(\ch_0,\ch_1,\ch_2) := (\vec a \cdot \vec n)(\ch_1-s\,\ch_0)-(\vec b\cdot \vec n)(-\ch_2+q\,\ch_0).\]
The zero locus of $f$ contains $P$ because $f(1,s,q) = 0$. We also have
\[f(\tilde v(F)) = f(n_1\tilde v(E_1)-n_2\tilde v(E_2)+n_3\tilde v(E_3)) =  (\vec a \cdot \vec n)(\vec b\cdot \vec a)-(\vec b\cdot \vec n)(\vec a \cdot \vec b) = 0.\]
Therefore, the zero locus of $f$ also contains $v(F)$.

It is easy to check that
\[(f(v(E_1)), -f(v(E_2)),f(v(E_3)))\]
 is the vector $\vec \rho$. Since $P$ is above the line $L_{E_1E_2}$, $\phi_P(E_2[1])<\phi_P(E_1[2])$ and the determinant $\det\begin{bmatrix}a_1 & a_2\\ b_1 & b_2 \end{bmatrix}<0$. Similarly, $\det \begin{bmatrix}a_1 & a_3\\ b_1 & b_3 \end{bmatrix}<0$. Therefore, the first factor of $\vec \rho^\bullet$, which is $-\det \begin{bmatrix}a_1 & a_2\\ b_1 & b_2 \end{bmatrix}n_2-\det \begin{bmatrix}a_1 & a_3\\ b_1 & b_3 \end{bmatrix}n_3$, is always positive. So $f(v(E_1))>0$.
 
Now $f$ satisfies the desired properties, and induces the weight character $\vec \rho$. Any other $f$ satisfying the same properties induces the same character up to a positive scalar.
\end{proof}

\begin{rem}
By the construction, the character $\vec \rho_{L_{FP},\mathcal E}$ (up to a positive scalar) only depends on the wall $L$ but not the position of $P$.
\end{rem}

To conclude the construction of the moduli space of $\sigma$-stable objects via the geometric invariant theory, we summarize the previous notations as follows. Let $w$ be a character and $\sigma_P$ a geometric stability condition. Suppose $P$ is in MZ$_\mathcal E$ for some exceptional triple $\mathcal E$ $=$ $\langle E_1,E_2,E_3\rangle$ such that $L_{Pw}$ intersects $l_{E_1E_3}$. We assume that $w$ or $-w$  can be written as $n_1\tilde v(E_1)-n_2\tilde v(E_2) +n_3\tilde v(E_3)$ for a positive dimension character $\vec n_w$$=$ $(n_1,n_2,n_3)$.
Let $G_{\vec n}$ be the group GL$(H_{n_1})\times \mathrm{GL}(H_{n_2})\times \mathrm{GL}(H_{n_3})$ acting naturally on the space of $\mathbf{Rep}(Q_\mathcal E,\vec n_w, \alpha_\mathcal E)$ and $\mathbf{Rep}(Q_\mathcal E,\vec n_w)$ with stabilizer containing the scalar group $\mathbb C^\times$.
\begin{prop}
Adopt the notation as above, the moduli space  $\mathfrak M^{ss}_\sigma (w)$ or $\mathfrak M^{ss}_\sigma (-w)$ of $\sigma_P$-semistable objects in $\mathrm{Coh}_{\sharp P}$ can be constructed as the GIT quotient space
\begin{align*}
& \mathbf{Rep}(Q_\mathcal E,\vec n_w, \alpha_\mathcal E)\text{ //}_{\det^{\vec \rho_{L_{w\sigma},\mathcal E}}}\big(G_{\vec n_w}/\mathbb C^\times\big)\\
= & \mathbf{Proj} \bigoplus_{m\geq 0} \mathbb C[\mathbf{Rep}(Q_\mathcal E,\vec n_w, \alpha_\mathcal E)]^{G_{\vec n_w}/\mathbb C^\times,\det^{ m\vec \rho_{L_{w\sigma},\mathcal E}}}.
\end{align*}
\label{prof:gitconst}
\end{prop}
\begin{proof}
By the previous discussion and Lemma \ref{lemma: rhocharformular}, the moduli space $\mathfrak M^{ss}_\sigma (w)$  or $\mathfrak M^{ss}_\sigma (-w)$ parameterizes the $\vec \rho_{L_{w\sigma},\mathcal E}$-semistable objects in $\mathcal A_\mathcal E$ with dimension character $\vec n_w$. By King's criterion, Proposition 3.1 in \cite{Ki}, $\mathbf K$ in \textbf{Rep}$(Q_\mathcal E,\vec n_w, \alpha_w)$ is $\vec \rho_{L_{w\sigma}}$-semistable if the point of $\mathbf K$ in the space \textbf{Rep}$(Q_\mathcal E,\vec n_w)$ is $\det^{ \vec \rho_{L_{w\sigma}}}$-semistable with respect to the $G_{\vec n_w}/\mathbb C^\times$-action. The map
\begin{align*}
\mathbb C[\textbf{Rep}(Q_\mathcal E,\vec n_w)]^{G,\det^{ m\vec \rho}} & \rightarrow \Big(\mathbb C[\textbf{Rep}(Q_\mathcal E,\vec n_w)]/I_{\alpha_\mathcal E}\Big)^{G,\det^{ m\vec \rho}} \\ & = \mathbb C[\textbf{Rep}(Q_\mathcal E,\vec n_w, \alpha_w)]^{G,\det^{ m\vec \rho}}.
\end{align*}
is surjective because the group $G=G_{\vec n_w}/\mathbb C^\times$ is semisimple. Therefore, the relation $\alpha_\mathcal E$ does not affect the stability condition. In other words, a point $\mathbf K$  is $\det^{ \vec \rho_{L_{w\sigma}}}$-semistable with respect to the $G_{\vec n_w}/\mathbb C^\times$-action on  the space $\mathbf{Rep}(Q_\mathcal E,\vec n_w)$ if and only if on  the space $\repqa$. As explained in Chapter 2.2 in
\cite{Gi} by Ginzburg, the moduli space is
constructed as the GIT quotient in the proposition.
\end{proof}

Now we have the following consequence on the finiteness of actual walls.

\begin{prop}
1. Let $w$ be a character in $\mathrm{K}(\pp$), then there are only finitely many actual walls for $\mathfrak M^{ss}_\sigma(w)$.\\
2. Suppose $\ch_0(w)> 0$, then for any $s<\chy(w)$ and $q$ large enough (depending on $s$), the moduli space $\mathfrak M^{s(ss)}_{\sigma_{s,q}}(w)$ is the same as the moduli space $\mathfrak M^{s(ss)}_{GM}(w)$ of Gieseker (semi)stable coherent sheaves.
\label{finite many wall, t>0 is HilbS}
\end{prop}
\begin{proof} 
By Corollary \ref{cor:lastwallradiusgeq1}, we only need to consider the region from the vertical wall to the tangent line of $\bd_0$. We may choose finitely many quiver region MZ$_\mathcal E$ such that each ray from $w$ contained in this region passes through at least one MZ$_\mathcal E$. In each MZ$_{\mathcal E}$, there are finitely many walls, because there are only finitely many dimension vectors of possible destabilizing subobjects.

The second statement is a consequence of the first statement and the standard fact that $\sigma_{s,q}$ tends to Gieseker stable condition when $q$ tends to infinity.
\end{proof}

\subsection{The Ext$^2$ vanishing property}\label{sec2.2}

In this section we prove the most important technical lemma. It is about the
vanishing property of Ext$^2$ of $\sigma$-stable objects. This property is trivial in the slope stable situation by Serre duality. But it is more involved in the Bridgeland stability situation since the objects may not be in the same heart.

\begin{lemma}
Let $\sigma_P$ be a geometric stability condition with $P$ in $\bd_{<0}$, $E$ and $F$ be two $\sigma_P$-stable objects in $\Coh_{\sharp P}$. Suppose $P$, $ v(E)$ and $ v(F)$ are collinear, then
\[\Hom(E,F[2]) = \Hom(F,E[2]) = 0.\]
\label{lemma: ext2 vanishing for stable factors}
\end{lemma}
\begin{proof}
\textbf{Case 1}: At least one of $\tilde v(E)$ and $\tilde v(F)$ is an exceptional character. Assume that $\tilde v(E)$ is exceptional. Suppose the dyadic number corresponding to $\tilde v(E)$ is $\frac{p}{2^q}$, let $E_1$ and $E_3$ be exceptional bundles corresponding to dyadic numbers $\frac{p-1}{2^q}$ and $\frac{p+1}{2^q}$ respectively. Then $\Hom(E_1,E$) and $\Hom(E,E_3$) are both non-zero. As $E$ is $\sigma$-stable, $l_{E\sigma}$ does not intersect $l_{e_1e_1^+}$ nor $l_{e_3e_3^+}$, otherwise this contradicts to Lemma \ref{lemma:slopecompare}. We may assume that $\sigma$ is in MZ$_\mathcal E$, where $\mathcal E = \{ E_1,E,E_3\}$. As $F$ has the same phase with $E$ at $\sigma$, $F[1]$ is in $\mathcal A_\mathcal E$. Since
\begin{align*}
& \Hom(E[1],E_1[2+s])  = 0 \text { for all } s\in\mathbb Z;\\
& \Hom(E[1],E[1+s])  = 0 \text { for all } s\neq 0;\\
& \Hom(E[1],E_3[s])  = 0 \text { for all } s\neq 1,\\
\end{align*}
$\Hom(E[1],G[s])$ $=0$, for any object $G$ in $\mathcal A_\mathcal E$ when $s\neq 0$ or $1$. Therefore, $\Hom(E,F[2])$ $=$ $\Hom(E[1],F[1+2])$ $=$ $0$. Similarly, We have $\Hom(G,E[1+s])$ $=0$, for any object $G$ in $\mathcal A_\mathcal E$ when $s\neq 0$ or $1$. Therefore, $\Hom(F,E[2])$ $=$ $0$.\\

\textbf{Case 2:} Neither $\tilde v(E)$ nor $\tilde v(F)$ is exceptional. By Corollary \ref{cor:nostabincone}, their corresponding points are below the Le Potier curve $C_{LP}$.

Case 2.1: The $\frac{\ch_1}{\ch_0}$-length of $L_{EF}\bigcap \bd_{\leq  0}$ is greater than $3$. 

By the construction of Bridgeland stability conditions, it is easy to see that the objects $E(-3)$ and $F(-3)$ are also stable for any geometric stability conditions on the line $L_{E(-3)F(-3)}$. Since the $\frac{\ch_1}{\ch_0}$-length of $L_{EF}\bigcap \bd_{\leq  0}$ is greater than $3$, the intersection point $Q$ of $L_{EF}$ and $L_{E(-3)F(-3)}$ is in $\bd_{\leq 0}$. By Corollary \ref{cor:stabonthesegment}, the objects $E$, $F$, $E(-3)$ and $F(-3)$ are all $\sigma_Q$-stable. By Lemma \ref{lemma:slopecompare}, $\phi_Q(E(-3))$ $=$ $\phi_Q(F(-3))$ $<$ $\phi_Q(E)$ $=$ $\phi_Q(F)$.  We have
\[\Hom(E,F(-3)),\Hom(F,E(-3))=0.\]
The statement then holds by Serre duality.

Case 2.2: The $\frac{\ch_1}{\ch_0}$-length of $L_{EF}\bigcap \bd_{\leq  0}$ is not greater than $3$. 


\begin{center}
\begin{tikzpicture}[domain=1:5]

\tikzset{%
    add/.style args={#1 and #2}{
        to path={%
 ($(\tikztostart)!-#1!(\tikztotarget)$)--($(\tikztotarget)!-#2!(\tikztostart)$)%
  \tikztonodes},add/.default={.2 and .2}}
}

\newcommand\XA{0.02}

\draw [name path =C0, opacity=0.1](-3,4.5) parabola bend (0,0) (3,4.5)
 node[right, opacity =0.5] {$\bd_0$};

\coordinate (E1) at (0,0);
\draw (E1) node {$\bullet$} node [above left] {};

\coordinate (E2) at (1,0.5);
\draw (E2) node {$\bullet$} node [above] {$\mathcal O(k+3)$};

\coordinate (E3) at (2,2);
\draw (E3) node {$\bullet$} node [above] {$\mathcal O(k+4)$};

\coordinate (F1) at (0,-1);
\draw (F1) node {$\bullet$} node [below] {$\mathcal O(k+2)^+$};


\coordinate (F3) at (2,1);
\draw (F3) node {$\bullet$}node [right] {$\mathcal O(k+4)^+$};

\draw (E1) -- (E3);
\draw (F1) -- (E3);
\draw (E1) -- (F3);
\draw (F3) -- (E3);
\draw (E1) -- (F1);

\coordinate (E1) at (-3,4.5);
\draw (E1) node {$\bullet$} node [above left] {$\mathcal O(k-1)$};


\coordinate (E3) at (-1,0.5);
\draw (E3) node {$\bullet$} node [above] {$\mathcal O(k+1)$};

\coordinate (F1) at (-3,3.5);
\draw (F1) node {$\bullet$} node [left] {$\mathcal O(k-1)^+$};


\coordinate (F3) at (-1,-0.5);
\draw (F3) node {$\bullet$}node [left] {$\mathcal O(k+1)^+$};


\draw[->,opacity =0.3] (-4,0) -- (4,0) node[above right] {$\frac{\ch_1}{\ch_0}$};

\draw[->,opacity=0.3] (0,-2)-- (0,0) node [above right] {O} --  (0,5) node[right] {$\frac{\ch_2}{\ch_0}$};

\draw (E1) -- (E3);
\draw (F1) -- (E3);
\draw (E1) -- (F3);
\draw (F3) -- (E3);
\draw (E1) -- (F1);

\coordinate (E) at (2,-0.2);
\draw (E) node {$\bullet$}node [below left] {$E$};
\coordinate (F) at (-3.5,2.5);
\draw (F) node {$\bullet$}node [below] {$F$};

\draw [add =0.2 and 0.2,dashed] (E) to (F);
\draw [add =0.2 and -0.5,opacity=0] (E) to (F) node[opacity=1] {$\bullet$} node[above, opacity=1] {$\sigma_P$};

\end{tikzpicture}

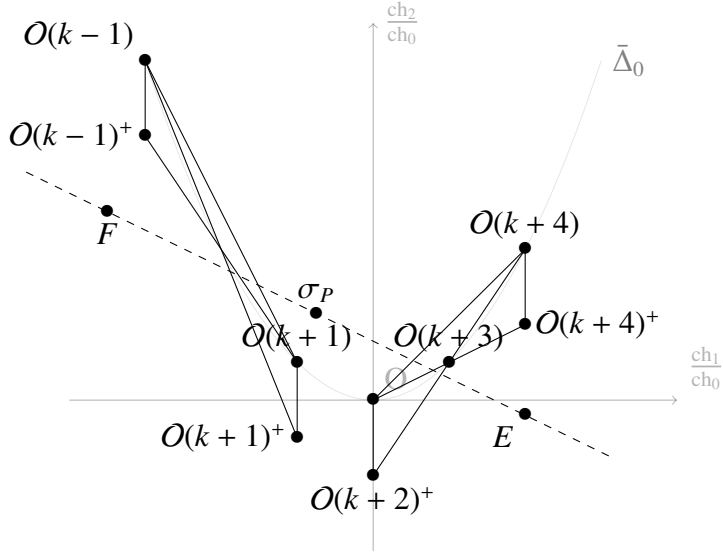
Picture: $L_{EF}\bigcap \bd_{\leq  0}$ is not greater than $3$.
\end{center}


Since $E$ is $\sigma_P$-stable, by Corollary \ref{cor:lastwallradiusgeq1} and its proof, there exists an integer $k$ such that the points $v(\mathcal O(k+1))$, which is $\left(1, k+1,\frac{(k+1)^2}{2}\right)$ and $v(\mathcal O(k+2))$  are below the segment $L_{EF}\bigcap \bd_{\leq  0}$ (see above picture, where $e_1$, $e_2$, $e_3$ correspond to $\mathcal O(k+2)$, $\mathcal O(k+3)$, $\mathcal O(k+4)$ respectively). Equivalently,  the points $v(\mathcal O(k-1))$ and $v(\mathcal O(k-2))$ are below the segments $L_{E(-3)F(-3)}\bigcap \bd_{\leq  0}$.

Let $\mathcal E_k$ be the exceptional triple $\langle \mathcal O(k-1),\mathcal O(k),\mathcal O(k+1)\rangle$, then by our assumption, $L_{\mathcal O(k-1)\mathcal O(k+1)}$ must intersect both $L_{EF}$ and $L_{E(-3)F(-3)}$. By Lemma \ref{lemma:ForF1inAE}, as the point $v(\mathcal O(k+1))$  is below the segment $L_{EF}\bigcap \bd_{\leq  0}$, $E$ and $F$ are both in $\mathcal A_{\mathcal E_k}$. As the point $v(\mathcal O(k-1))$  is below the segment $L_{E(-3)F(-3)}\bigcap \bd_{\leq  0}$, both $E(-3)[1]$ and $F(-3)[1]$ are in $\mathcal A_{\mathcal E_k}$. Therefore, we have
\[\Hom (E,F(-3)) = \Hom (E,(F(-3)[1])[-1]) = 0.\]
By Serre duality, the statement holds.
\end{proof}
In particular, if an object $E$ is $\sigma$-semistable for some geometric stability condition $\sigma$, we have $\Hom(E,E[2])$ $=$ $0$. To see this, $E$ admits $\sigma$-stable Jordan-Holder filtrations, and for any two stable factors we have the $\Hom(-,-[2])$ vanishing, hence $\Hom(E,E[2])$ $=$ $0$.

As an immediate application, \textbf{Rep}$(Q_\mathcal E,\vec n,\alpha_\mathcal E)^{\vec \rho-ss}$ is smooth.
\begin{cor}
Let $x$ be a point in $\mathbf{Rep}(Q_\mathcal E,\vec n_w, \alpha_\mathcal E)^{\vec \rho-ss}$, then as a closed subvariety of $\mathbf{Rep}(Q_\mathcal E,\vec n_w)$, $\mathbf{Rep}(Q_\mathcal E,\vec n_w, \alpha_\mathcal E)$ is smooth at the point $x$.
\label{cor:smoothness}
\end{cor}
\begin{proof}
Let \textbf K $= (I_0,J_0)$  be the quiver representation that $x$ stands for. The dimension of the Zariski tangent space at $x$ is the dimension of
\[\Hom_{\mathbb C} \left(\mathbb C[\textbf{Rep}(Q_\mathcal E,\vec n_w)]\middle/(J\circ I)\;,\; \mathbb C[t]/(t^2)\right)\]
at $(I_0,J_0)$. Each tangent direction can be written in the form $(I_0,J_0)$ $+$ $t(I_1,J_1)$. In order to satisfy the equation $J\circ I \in (t^2)$, we need \[J_0\circ I_1+J_1\circ  I_0 =0.\]
Hence the space of $(I_1,J_1)$ is just the kernel of $d^1:$
$\Hom^1(\textbf K,\textbf K)$ $\rightarrow$ $\Hom^2(\textbf K,\textbf
K)$. By Lemma \ref{lemma: ext2 vanishing for stable factors}, $d^1$ is surjective. The
Zariski tangent space has dimension hom$^1(\textbf K,\textbf K)$
$-$ hom$^2(\textbf K,\textbf K)$. On the other hand, $\mathbf{Rep}(Q_{\mathcal E},\vec{n}, \alpha_\mathcal E)$ is the zero locus of $n_1n_3\cdot \hom(E_1,E_3) = \hom^2(\textbf K,\textbf K)$ equations, hence each irreducible component is of dimension at least $\hom^1(\textbf K,\textbf K)$
$-$ hom$^2(\textbf K,\textbf K)$, which is not less than the dimension of the Zariski tangent space at $x$. Therefore, \textbf{Rep}$(Q_\mathcal E,\vec n_w, \alpha_\mathcal E)$ is smooth at the point $x$.
\end{proof}
\begin{rem}
When the dimension character $\vec n$ is primitive, $G(=G_{\vec n_w}/\mathbb C^\times)$ acts freely on the stable locus. By Luna's \'{e}tale slice theorem,
\[\Rep(Q_\mathcal E,\vec n_w, \alpha_\mathcal E)^{\vec \rho-s}
\rightarrow \left. \Rep(Q_\mathcal E,\vec n_w, \alpha_\mathcal E)^{\vec \rho-s}\middle/G\right.\]
is a principal $G$-bundle. Since
$\Rep(Q_\mathcal E,\vec n_w, \alpha_\mathcal E)^{\vec \rho-s}$ is smooth, by Proposition IV.17.7.7 in \cite{Gr}, the
base space is also smooth.
\label{rem:smoothness}
\end{rem}

\subsection{Generic stability}\label{sec2.3}
Based on Lemma \ref{lemma: ext2 vanishing for stable factors}, we establish some estimate on the dimension of strictly semistable objects in this section. The technical result Lemma \ref{lemma:repfgdim} is useful in the proof for the irreducibility of the moduli space. 

\begin{defn}
Suppose $\vec{n} = \vec{n}' + \vec{n}''$ such that $\vec n'\cdot \vec \rho = \vec n''\cdot \vec \rho = 0$.  Choose $\mathbf{F}\in \mathbf{Rep}(Q_{\mathcal E},\vec{n}', \alpha_\mathcal E)^{\vec\rho-ss}$ and $\mathbf{G} \in \mathbf{Rep}(Q_{\mathcal E},\vec{n}'',\alpha_\mathcal E)^{\vec\rho-ss}$. We write $\mathbf{Rep}(Q_{\mathcal E}, \mathbf{F}, \mathbf{G})$ as the \emph{subspace} in $\mathbf{Rep}(Q_{\mathcal E},\vec{n},\alpha_\mathcal E)^{\vec\rho-ss}$ consisting of representations $\mathbf{K}$ that can be written as an extension of $\mathbf{G}$ by $\mathbf{F}$:
\[0\rightarrow \mathbf{F}\rightarrow \mathbf{K}\rightarrow \mathbf{G}\rightarrow 0.\]
We also write $\Rep(Q_{\mathcal E}, \vec n', \vec n'')$ for the union of all $\Rep(Q_{\mathcal E}, \mathbf{F}, \mathbf{G})$ such that $\mathbf{F}\in \mathbf{Rep}(Q_{\mathcal E},\vec{n}', \alpha_\mathcal E)^{\vec\rho-ss}$ and $\mathbf{G} \in \mathbf{Rep}(Q_{\mathcal E},\vec{n}'',\alpha_\mathcal E)^{\vec\rho-ss}$.
\label{def:repef}
\end{defn}
We have the following dimension estimate for $\Rep(Q_\mathcal E,\mathbf F,\mathbf G)$:

\begin{lemma}
\[\dim \Rep(Q_{\mathcal E}, \mathbf{F}, \mathbf{G}) \leq -\chi(\mathbf{G},\mathbf{F}) + \dim G_{\vec{n}} - \hom(\mathbf{F},\mathbf{F}) -\hom(\mathbf{G},\mathbf{G}).\]
\label{lemma:repfgdim}
\end{lemma}
\begin{proof}
Let $X(\mathbf{F},\mathbf{G})$ be the subset of $\textbf{Rep}(Q_{\mathcal E}, \mathbf{F}, \mathbf{G})$ consisting of objects of the form:
\[I= \left( \begin{array}{cc}
I_\mathbf{F} & I(\mathbf{G},\mathbf{F}) \\
0 & I_\mathbf{G}\end{array} \right), J= \left( \begin{array}{cc}
J_\mathbf{F} & J(\mathbf{G},\mathbf{F}) \\
0 & J_\mathbf{G}\end{array} \right),\]
for a pair $(I(\mathbf{G},\mathbf{F}), J(\mathbf{G},\mathbf{F}))\in \Hom^1(\mathbf G,\mathbf F)$. The morphisms are shown in the following diagram:
\begin{displaymath}
    \xymatrix{
       \mathbf{F}: & \mathbb C^{n_1'} \ar[r]^{I_\mathbf{F}} & \mathbb C^{n_2'}  \ar[r]^{J_\mathbf{F}} & \mathbb C^{n_3'} \\
    \mathbf{G}: &     \mathbb C^{n_1''} \ar[r]^{I_\mathbf{G}} \ar[ur]^{I(\mathbf{G},\mathbf{F})}& \mathbb C^{n_2''}  \ar[r]^{J_\mathbf{G}} \ar[ur]^{J(\mathbf{G},\mathbf{F})}& \mathbb C^{n_3''}
        }
\end{displaymath}
Due to the condition that $J \circ I \in \ker \alpha_\mathcal E \otimes \Hom(\mathbb C^{n_1}, \mathbb C^{n_3})$, the pair $(I(\mathbf{G},\mathbf{F}), J(\mathbf{G},\mathbf{F}))$ is contained in the kernel of the morphism
\[d^1(\mathbf{G},\mathbf{F}): \Hom^1(\mathbf{G},\mathbf{F}) \to \Hom^2(\mathbf{G},\mathbf{F}).\]
By Lemma \ref{lemma: ext2 vanishing for stable factors}, $d^1(\mathbf{G},\mathbf{F})$ is surjective, hence
\[\dim X(\mathbf{F},\mathbf{G})\leq \hom^1(\mathbf{G},\mathbf{F}) - \hom^2(\mathbf{G},\mathbf{F}).\]

Each element $g \in GL_{\vec{n}}$ can be written as a block matrix $\left( \begin{array}{cc}
A & B \\
C & D\end{array} \right)$, where $A \in \Hom^0(\mathbf{F},\mathbf{F})$, $B \in \Hom^0(\mathbf{G},\mathbf{F})$, $C \in \Hom^0(\mathbf{F},\mathbf{G})$ and $D \in \Hom^0(\mathbf{G},\mathbf{G})$. Note that when $A \in \Hom(\mathbf{F},\mathbf{F})$, $D \in \Hom(\mathbf{G},\mathbf{G})$ and $C=0$, we have $g\cdot X(\mathbf{F},\mathbf{G})=X(\mathbf{F},\mathbf{G})$. Therefore,
\begin{align*}
\dim \textbf{Rep}(Q_{\mathcal E}, \mathbf{F}, \mathbf{G}) &= \dim G_{\vec{n}}\cdot X(\mathbf{F},\mathbf{G})\\
&\leq \dim G_{\vec{n}} + \dim X(\mathbf{F},\mathbf{G}) - \hom(\mathbf{F},\mathbf{F}) - \hom(\mathbf{G},\mathbf{G}) -\hom^0(\mathbf{G},\mathbf{F})\\
&\leq -\chi(\mathbf{G},\mathbf{F}) + \dim G_{\vec{n}} - \hom(\mathbf{F},\mathbf{F}) -\hom(\mathbf{G},\mathbf{G}).
\end{align*}
\end{proof}

\begin{defn}
$\Rep(Q_{\mathcal E},\vec{n},\alpha_\mathcal E)^{\vec\rho-ss}_c:= \{\mathbf{F}\in \Rep(Q_{\mathcal E},\vec{n},\alpha_\mathcal E)^{\vec\rho-ss}| \hom(\mathbf{F},\mathbf{F})=c\}$.\\
$\mathbf{Rep}(Q_{\mathcal E},\vec n', \vec n'')^{\vec\rho - ss}_{c,d}:=\{\mathbf{K}\in \Rep(Q_{\mathcal E},\mathbf{F},\mathbf {G})^{}| \hom(\mathbf{F},\mathbf{F})=c,\hom(\mathbf{G},\mathbf{G})=d\}$.
\end{defn}

The following proposition shows that given a Chern character $w$ not inside Cone$_{LP}$ and a generic stability condition $\sigma$, stable objects are dense in the moduli space $\mathfrak M^{ss}_{\sigma}(w)$. Note that this is a non-trivial statement only when $w$ is not primitive. 

\begin{prop}
Let $\vec n$ be a character for $\Rep(Q_\mathcal E,\alpha_\mathcal E)$ such that $\chi(\vec n, \vec n) \leq -1$. Let $\vec \rho$ be a generic weight with respect to $\vec n$, in other words, $\vec\rho \cdot \vec n' \neq 0$ for any $\vec n' < \vec n$ that is not proportional to $\vec n$. We have
\[\dim \repqa^{\vec\rho-ss} = -\chi(\vec n, \vec n) + \dim G_{\vec n},\]
for each irreducible component of $\repqa^{\vec\rho-ss}$. Moreover,
\[\dim \left(\repqa^{\vec\rho-ss} \setminus \repqa^{\vec\rho-s}\right) \leq -\chi(\vec n, \vec n) + \dim G_{\vec n}-1.\]
In particular, there is no component whose objects are all strictly semistable objects.
\label{prop:compcontainsstabobj}
\end{prop}
\begin{proof}
The first statement basically follows from the proof of Corollary \ref{cor:smoothness}. Just note that $hom^1(\mathbf{K},\mathbf{K})-hom^2(\mathbf{K},\mathbf{K})$ in that proof is exactly $-\chi(\vec n, \vec n) + \dim G_{\vec n}$ here. We will repeat the proof:

$\textbf{Rep}(Q_{\mathcal E},\vec{n},\alpha_\mathcal E)$ is the zero locus of $n_1n_3\cdot hom(E_1,E_3)$ equations, hence each irreducible component is of dimension at least $-\chi(\vec n, \vec n) + \dim G_{\vec n}$.

On the other hand, for any $\vec\rho$-semistable object $\mathbf{K} \in \repqa^{\vec \rho-ss}$, $d^1: \Hom^1(\mathbf{K},\mathbf{K}) \to \Hom^2(\mathbf{K},\mathbf{K})$ is surjective by Lemma \ref{lemma: ext2 vanishing for stable factors}, the Zariski tangent space is of dimension $-\chi(\vec n, \vec n) + \dim G_{\vec n}$. Since $\textbf{Rep}(Q_{\mathcal E},\vec{n},\alpha_\mathcal E)^{\vec\rho-ss}$ is open in $\textbf{Rep}(Q_{\mathcal E},\vec{n},\alpha_\mathcal E)$, for each irreducible component of $\repqa^{\vec \rho-ss}$, its dimension is $-\chi(\vec n, \vec n) + \dim G_{\vec n}$.

For the second statement, when $\vec n$ is primitive and $\vec \rho$ is generic, we have
\[\textbf{Rep}(Q_{\mathcal E},\vec{n},\alpha_\mathcal E)^{\vec\rho-ss} = \textbf{Rep}(Q_{\mathcal E},\vec{n},\alpha_\mathcal E)^{\vec\rho-s},\]
so the statement holds automatically in this case.

We may assume that $\vec n = m \vec n_0$, in which $\vec n_0$ is primitive. Since $\vec\rho$ is generic, any strictly semistable object must be destabilized by an object in $\textbf{Rep}(Q_{\mathcal E},a\vec n_0)^{\vec\rho-ss}$ for some $0 < a < m$. Hence
\[\left.\textbf{Rep}(Q_{\mathcal E},\vec{n},\alpha_\mathcal E)^{\vec\rho-ss} \text{{\Large$\backslash$}} \textbf{Rep}(Q_{\mathcal E},\vec{n},\alpha_\mathcal E)^{\vec\rho-s}\right. = \bigcup_{1\leq a\leq m-1} \textbf{Rep}(Q_{\mathcal E},a\vec n_0, (m-a)\vec n_0).\]

For each object $\mathbf{F}\in \textbf{Rep}(Q_{\mathcal E},a\vec n_0,\alpha_\mathcal E)_c^{\vec \rho-ss}$, the orbit $G_{a\vec n_0} \cdot \mathbf{F}$ in $\textbf{Rep}(Q_{\mathcal E},a\vec n_0,\alpha_\mathcal E)^{\vec \rho-ss}$ is of dimension $\dim G_{a\vec n_0} - c$. Therefore, by Lemma \ref{lemma:repfgdim}, we have
\begin{align*}
&\;\;\;\;\dim \textbf{Rep}(Q_{\mathcal E},a\vec n_0, (m-a)\vec n_0)^{\vec\rho - ss}_{c,d}\\
&\leq -\chi((m-a)\vec n_0, a\vec n_0) + \dim G_{\vec n} -c -d-(\dim G_{a\vec n_0} -c) -(\dim G_{(m-a)\vec n_0} -d) \\
&\;\;\;\;+ \dim \textbf{Rep}(Q_{\mathcal E},a\vec n_0,\alpha_\mathcal E)^{\vec\rho - ss}_{c} + \dim \textbf{Rep}(Q_{\mathcal E}, (m-a)\vec n_0,\alpha_\mathcal E)^{\vec\rho - ss}_d\\
&\leq -\chi((m-a)\vec n_0, a\vec n_0) + \dim G_{\vec n} - \chi((m-a)\vec n_0, (m-a)\vec n_0) - \chi(a\vec n_0, a\vec n_0)\\
&= -\chi(\vec n, \vec n) + \dim G_{\vec n} + \chi(a\vec n_0, (m-a)\vec n_0)\\
&\leq -\chi(\vec n, \vec n) + \dim G_{\vec n}-1.
\end{align*}
The last inequality holds since $\chi(\vec n,\vec n)\leq -1$. Therefore,
\begin{align*}
&\dim\left(\mathbf{Rep}(Q_{\mathcal E},\vec{n},\alpha_\mathcal E)^{\vec\rho-ss} \text{\Large$\backslash$} \mathbf{Rep}(Q_{\mathcal E},\vec{n},\alpha_\mathcal E)^{\vec\rho-s}\right)\\
\leq & \max_{c,d} \left\{\dim \textbf{Rep}(Q_{\mathcal E},a\vec n_0, (m-a)\vec n_0)^{\vec\rho - ss}_{c,d}\right\}\\
\leq & -\chi(\vec n, \vec n) + \dim G_{\vec n}-1.
\end{align*}
In particular, since each component is of dimension $-\chi(\vec n, \vec n) + \dim G_{\vec n}$, there is no component consisting of strictly semistable objects.
\end{proof}

\subsection{The irreducibility of the moduli space}\label{sec2.4}
Based on the results and methods in the previous sections, we are able to estimate the dimension of the space of new stable objects after a wall-crossing. When the wall is to the left of the  vertical wall, we show that the new stable objects in the next chamber has codimension at least $3$. Together with Proposition \ref{prop:compcontainsstabobj}, this will imply the irreducibility of the moduli space.

Let $w$ be a character in K(\textbf P$^2$) with $\ch_0(w)\geq 0$ and $\sigma_{s,q}$ a stability condition with $s<\frac{\ch_1(w)}{\ch_0(w)}$ such that  $(1,s,q)$ is contained in MZ$_\mathcal E$. Let $\vec n$ be the dimension character for $w$ in $Q_\mathcal E$, and $\vec\rho$ be the weight character corresponding to $L_{w\sigma}$. Let $\vec \rho_-$ be the character in the chamber below $L_{w\sigma}$ and $\vec \rho_+$ in the chamber above $L_{w\sigma}$. The following two lemmas will be used in the proof of Proposition \ref{contracting is more than producing lemma}.
\begin{lemma}
Suppose $\mathbf{K}$ is an object in $\mathbf{Rep}(Q_\mathcal E,\vec n_w, \alpha_\mathcal E)^{\vec \rho_- -s}\setminus \mathbf{Rep}(Q_\mathcal E,\vec n_w, \alpha_\mathcal E)^{\vec \rho_+ -s}$, then it can be written as a non-trivial extension
\[0\rightarrow \mathbf{K}'\rightarrow \mathbf{K}\rightarrow \mathbf{K}''\rightarrow 0\]
of objects in $\Rep(Q_\mathcal E,\alpha_\mathcal E)$ such that the dimension character $\vec n'$ of $\mathbf{K}'$ satisfies $\vec \rho_-\cdot \vec n'< 0 = \vec \rho\cdot \vec n'$, and $\Hom(\mathbf{K}'',\mathbf{K}')$ $=$ $0$.
\label{lemma:Fdecompastwofactor}
\end{lemma}
\begin{proof}
By the assumption on $\mathbf{K}$, it is a strictly $\vec \rho$-semistable object, and is destabilized by a non-zero $\vec \rho$-stable proper subobject $\mathbf{K'}$ with $\vec \rho\cdot \vec n'=0$. As $\mathbf{K}$ is $\vec \rho_-$-stable, we have  $\vec \rho_-\cdot \vec n'<0$. Let the quotient be $\mathbf{K''}$, then $\mathbf{K'}$ and $\mathbf{K''}$ are the objects we want.

In order to see that $\Hom(\mathbf{K}'',\mathbf{K}')$ $=$ $0$, suppose there is a non-zero map in $\Hom(\mathbf{K''},\mathbf{K'})$, then its image $\tilde{\mathbf{K}}$ in $\mathbf{K'}$ is both a sub-representation and quotient representation of $\mathbf{K}$. Let $\vec{\tilde {n}}$ be the dimension vector of $\tilde{\mathbf{K}}$. As $\mathbf{K}$ is $\vec\rho_-$-stable, we get $\vec\rho_-\cdot \vec{\tilde n}<0<\vec\rho_-\cdot \vec{\tilde n}$, which leads to a contradiction.
\end{proof}
For a dimension vector $\vec n$ of $Q_\mathcal E$, we write $\ch_i(\vec n)$ for $n_1\ch_i(E_1)-n_2\ch_i(E_2)+n_3\ch_i(E_3)$, $i=0,1,2$.
\begin{lemma}
Let $\vec n$ and $\vec m$ be two dimension vectors of $Q_\mathcal E$. \\
1. The Euler character $\chi(\vec n,\vec m)$ can be computed as
\begin{align*}\ch_2(\vec n)\ch_0(\vec m)+\ch_2(\vec m)\ch_0(\vec n)-\ch_1(\vec n)\ch_1(\vec m)\\+\frac{3}{2}\left(\ch_1(\vec m)\ch_0(\vec n)-\ch_0(\vec m)\ch_1(\vec n)\right)+\ch_0(\vec n)\ch_0(\vec m).\end{align*}
2. Suppose $\ch_0(\vec n)\leq 0$, let $w$ be of character $-(\ch_0(\vec n),\ch_1(\vec n),\ch_2(\vec n))$ and $P$ be a point in $\bd_{<0}$ to the left of the vertical wall $L_{w\pm}$ such that $L_{Pw}$ intersects $l_{E_1E_3}$. Let $\vec \rho$ be $\vec \rho_{L_{Pw}}$, and $\vec\rho_-$ be the character in the chamber below $L_{Pw}$. Suppose $\vec m$ satisfies $\vec \rho_-\cdot \vec m <0 = \vec \rho\cdot \vec m$ and $\vec n$ satisfies $\vec \rho_-\cdot \vec n =0 = \vec \rho\cdot \vec n$, then \[\ch_0(\vec n)\ch_1(\vec m)-\ch_0 (\vec m) \ch_1(\vec n)>0.\]
\label{lemma:chiformrho}
\end{lemma}
\begin{proof}
The first statement follows from the Hirzebruch-Riemann-Roch formula for \textbf P$^2$:
\begin{align*}
\chi(F,G) &= \ch_2(F)\ch_0(G)+\ch_2(G)\ch_0(F)-\ch_1(F)\ch_1(G) \\ & +\frac{3}{2}\big(\ch_1(G)\ch_0(F)-\ch_0(G)\ch_1(F)\big)+\ch_0(F)\ch_0(G).
\end{align*}
For the second statement, by definition of $\vec \rho$, $\vec \rho_-$ is in the same chamber as $\vec \rho+\epsilon(0,n_3,-n_2)$  for small enough $\epsilon>0$. We have
\[\ch_0(\vec n)\ch_1(\vec m)-\ch_0 (\vec m) \ch_1(\vec n)=\vecm\cdot \vec \Upsilon,\]
where $\vec \Upsilon$ is the vector $\left(\left|\begin{matrix}\ch_0(\vec n) & \ch_1(\vec n)\\ \ch_0(E_1) & \ch_1(E_1)\end{matrix}\right|,-\left|\begin{matrix}\ch_0(\vec n) & \ch_1(\vec n)\\ \ch_0(E_2) & \ch_1(E_2)\end{matrix}\right|,\left|\begin{matrix}\ch_0(\vec n) & \ch_1(\vec n)\\ \ch_0(E_3) & \ch_1(E_3)\end{matrix}\right|\right)$. The vector $\vec \Upsilon$ is a weight character for $\vec n$ since $\vec n\cdot\vec \Upsilon$ $= \ch_0(\vec n)\ch_1(\vec n)- \ch_1(\vec n)\ch_0(\vec n)$ $= 0$.

When MZ$_\mathcal E$ intersects the vertical wall $L_{w\pm}$, by formula in Lemma \ref{lemma: rhocharformular}, $\vec \Upsilon$ is proportional (up to a positive scalar) to the character on the vertical wall. As $\vec\rho$ is to the left of the vertical wall, $\vec \Upsilon $ can be written as $ a\vec\rho-b\vec \rho_-$ for some positive numbers $a$ and $b$. Therefore, $\vec m \cdot \vec \Upsilon = -b \vec m\cdot \vec \rho_->0$.

When MZ$_\mathcal E$ is to the left of the vertical wall $L_{w\pm}$, we have $\frac{\ch_1(E_i)}{\ch_0(E_i)}\leq \frac{\ch_1(\vec n)}{\ch_0(\vec n)}$ for $i = 1,2,3$. As $\ch_0(\vec n)\leq 0$, we have $\left|\begin{matrix}\ch_0(\vec n) & \ch_1(\vec n)\\ \ch_0(E_i) & \ch_1(E_i)\end{matrix}\right|>0$. Since the third term of $\vec \rho$ is negative and the character space of $\vec n$ is spanned by $\vec \rho$ and $(0,n_3,-n_2)$, the character $\vec \Upsilon$ can be written as $a\vec \rho$ $-$ $b(0,n_3,-n_2)$ for some positive number $a$ and $b$. As $\vec \rho_-$ is in the same chamber as $\vec \rho+\epsilon(0,n_3,-n_2)$ and $\vec m\cdot \vec \rho_-<0$, we get $\vec m\cdot \vec \Upsilon >0$.
\end{proof}

Now we can give an estimate of the dimension of new stable objects after wall crossing.

\begin{prop}
The dimension of the space $\Rep(Q_\mathcal E,\vec n_w, \alpha_\mathcal E)^{\vec \rho_- -s}\lbs$ $\Rep(Q_\mathcal E,\vec n_w, \alpha_\mathcal E)^{\vec \rho_+ -s}$ is less than $-\chi(w,w)+\dim G_{\vec n_w} -2$.
\label{contracting is more than producing lemma}
\end{prop}
\begin{proof}
By Lemma \ref{lemma:Fdecompastwofactor}, the space \textbf{Rep}$(Q_\mathcal E,\vec n_w, \alpha_\mathcal E)^{\vec \rho_- -s}\lbs$ \textbf{Rep}$(Q_\mathcal E,\vec n_w, \alpha_\mathcal E)^{\vec \rho_+ -s}$ can be covered by the following pieces:
\[\textbf{Rep}^{\vec \rho_- -s}\lbs \textbf{Rep}^{\vec \rho_+ -s} = \bigcup_{\vec m} \textbf{Rep} (Q_\mathcal E,\vec m,(\vec n_w-\vec m))^{\vec \rho_- -s},\]
where $\vec m$ satisfies:
\begin{itemize}
\item $\vec \rho_-\cdot \vec m <0 = \vec \rho\cdot \vec m$;
\item $\chi(\vec n_w-\vec m,\vec m)\leq 0$.
\end{itemize}
The second condition is due to Lemma \ref{lemma: ext2 vanishing for stable factors} and Lemma \ref{lemma:Fdecompastwofactor}. Now similar to the proof of Proposition \ref{prop:compcontainsstabobj}, we have
\begin{align*}
&\;\;\;\;\dim \textbf{Rep}(Q_{\mathcal E},\vec m, \vec n_w-\vec m)_{c,d}^{\vec \rho_- -s}\\
&\leq -\chi(\vec n_w-\vec m, \vec m) + \dim G_{\vec n_w} -c -d-(\dim G_{(\vec n_w-\vec m)} -c) -(\dim G_{\vec m} -d) \\
&\;\;\;\;+ \dim \textbf{Rep}(Q_{\mathcal E},\vec n_w-\vec m,\alpha_\mathcal E)^{\vec\rho_- - ss}_{c} + \dim \textbf{Rep}(Q_{\mathcal E}, \vec m)^{\vec\rho_- - ss}_d\\
&\leq -\chi(\vec n_w-\vec m, \vec m) + \dim G_{\vec n_w} - \chi(\vec n_w-\vec m, \vec n_w-\vec m) - \chi(\vec m, \vec m)\\
&= -\chi(\vec n_w, \vec n_w) + \dim G_{\vec n_w} + \chi(\vec m, \vec n_w-\vec m)
\end{align*}
By  Lemma \ref{lemma:chiformrho},
\begin{align*}
& - \chi (\vec m,\vec n_w-\vec m)\\
\geq & \chi(\vec n_w-\vec m,\vec m) - \chi (\vec m,\vec n_w-\vec m)\\
= & \chi(\vec n_w,\vec m) - \chi (\vec m,\vec n_w)\\
= & 3 \left(\ch_0(\vec n_w)\ch_1(\vec m) - \ch_0(\vec m)\ch_1(\vec n_w)\right)\\
\geq & 3
\end{align*}
The last inequality is due to the second statement of Lemma \ref{lemma:chiformrho}.
\end{proof}

Now we prove the irreducibility of the moduli space of stable objects. This is well known to hold for moduli of Gieseker stable sheaves. The moduli spaces are given by moduli of quiver representations, so the dimension of each component has a lower bound. The point is, by the previous results, the dimension of new stable objects is smaller than this lower bound, so irreducible component cannot be produced after wall crossing.

\begin{theorem}
Let $w$ be a primitive character in $\mathrm{K}(\pp)$ such that $\ch_0(w)>0$. For a generic geometric stability condition $\sigma=\sigma_{s,q}$ with $s<\chy(w)$  not on any actual wall of $w$, the moduli space $\mathfrak M^{ss}_\sigma(w)$ is irreducible and smooth.
\label{smoothness and irred property}
\end{theorem}

\begin{proof}
The smoothness is proved in Corollary \ref{cor:smoothness}. We only need to show the irreducibility.

For any $\sigma$, the line $L_{w\sigma}$ intersects some MZ$_\mathcal E$. In fact, we may always choose $\mathcal E$ to be $\{ \mathcal O(k-1),\mathcal O(k),\mathcal O(k+1)\}$. By Proposition \ref{prof:gitconst}, $\mathfrak M^{ss}_{\sigma}(w)$ can always be constructed as
\[\repqwa\text{ //}_{\det^{\vec \rho_{L_{w\sigma}\mathcal E}}}\left(G_{\vec n_w}/ \mathbb C^\times\right).\]
In the chamber near the vertical wall, the component that
contains \textbf{Rep}$(Q_\mathcal E,\vec n_w, \alpha_\mathcal E)^{\vec \rho-s}$ is irreducible
since the quotient space $\repqwa^{\vec\rho-s}/G$ is $\mathfrak M_{\text{GM}}^{s}(w)$, which is smooth and connected.

By Proposition \ref{contracting is more than producing lemma}, while
crossing an actual wall, the new stable locus \textbf{Rep}$(Q_\mathcal E,\vec n_w, \alpha_\mathcal E)^{\vec \rho_- -s}\lbs$ \textbf{Rep}$(Q_\mathcal E,\vec n_w, \alpha_\mathcal E)^{\vec \rho_+ -s}$
in \textbf{Rep}$(Q_\mathcal E,\vec n_w, \alpha_\mathcal E)^{\vec \rho_- -s}$ has codimension greater than
$2$. On the other hand, since the $\repqwa$ is a subspace in $\repqw$ determined by $n_1n_2\hom(E_1,E_3)$ equations, each irreducible component has dimension at least
\[n_1n_2\hom(E_1,E_2)+n_2n_3\hom(E_2,E_3)-n_1n_2\hom(E_1,E_3),\]
which is the same as the dimension of \textbf{Rep}$(Q_\mathcal E,\vec n_w, \alpha_\mathcal E)^{\vec \rho_+ -s}$ and is greater than the dimension of the new stable locus \textbf{Rep}$(Q_\mathcal E,\vec n_w, \alpha_\mathcal E)^{\vec \rho_- -s}\setminus$ \textbf{Rep}$(Q_\mathcal E,\vec n_w, \alpha_\mathcal E)^{\vec \rho_+ -s}$. Since the stable locus is open in $\repqwa$, the new stable locus is contained in the same irreducible component of $\repqwa^{\vec \rho_+-s}$. \textbf{Rep}$(Q_\mathcal E,\vec n_w, \alpha_\mathcal E)^{\vec \rho_- -s}$ is still irreducible. Hence the moduli space of Bridgeland stable objects, given as the GIT quotient, is also irreducible.
\end{proof}

\begin{rem}
There is natural isomorphism 
\[\mathfrak M^{ss}_{\sigma_{s,q}}(w)\simeq \mathfrak M^{ss}_{\sigma_{-s,q}}(-\ch_0(w),\ch_1(w),-\ch_2(w))\]
induced by the map $\iota: F\mapsto \mathcal {RH}om(F,\mathcal O)[1]$. In terms of the quiver representation, an object $\mathbf {K}_F\in\repqa^{\vec \rho-ss}$:
\[\mathbf {K}_F: E_1\otimes H_{n_1}\xrightarrow {I_{\mathbf F}}E_2\otimes H_{n_2}\xrightarrow {J_{\mathbf F}} E_3\otimes H_{n_3}\]
is mapped to $\mathbf K_{\iota(F)}\in \mathbf{Rep}(Q_{\mathcal E^\vee},(n_3,n_2,n_1),\alpha_{\mathcal E^\vee})^{(-\rho_3,-\rho_2,-\rho_1)-ss}$:
\[\mathbf {K}_{\iota(F)}: E^\vee_3\otimes H^*_{n_3}\xrightarrow {J^T_{\mathbf F}}E_2^\vee\otimes H^*_{n_2}\xrightarrow {I^T_{\mathbf F}} E^\vee_1\otimes H^*_{n_1}.\]
The statement in Theorem \ref{smoothness and irred property} holds for  $\mss (-w)$ when  $\sigma=\sigma_{s,q}$ with $s>\chy(w)$. 
\label{rem:rightside}
\end{rem}

\subsection{Properties of GIT}\label{sec2.5}
Birational geometry via GIT has been studied in \cite{DoHu} by
Dolgachev and Hu, \cite{Th} by Thaddeus. Since the theorems in \cite{DoHu,Th} are stated based on a slightly different set-up, in this section, we
recollect some properties from these papers in the
language of affine GIT.

Let $X$ be an affine algebraic $G$-variety , where $G$ is a
reductive group and acts on $X$ via a linear representation. Given a
character $\rho$: $G\rightarrow \mathbb C^{\times}$, the
(semi)stable locus is written as $X^{\vec \rho-s}$ ($X^{\vec \rho-ss}$). We
write $\mathbb C[X]^{G,\chi}$ for the $\chi$-semi-invariant functions on
$X$, in other words, one has
\[f(g^{-1}(x))=\chi(g)\cdot f(x),\text{ for } \forall g\in G, x\in X.\]
Denote the GIT quotient by $X$//$_{\vec \rho} G$ $:=$ \textbf{Proj}
$\bigoplus_{n\geq 0} \mathbb C[X]^{G,\vec \rho^n}$ and the map from
$X^{\vec \rho-ss}$ to $X$//$_{\vec \rho} G$ by $F_{\vec \rho}$.\\

In additions, we need the following assumptions on $X$ and $G$:
\begin{enumerate}
\item There are only finite many walls in the space of characters on
which there are strictly semistable points, in the chamber we
have $X^{\vec \rho-s}$ $=$ $X^{\vec \rho-ss}$.
\item $X^{\vec \rho-s}$ is smooth
and the action of $G$ on $X^{\vec \rho-s}$ is free.
\item  $X$//$_{\vec \rho} G$ is projective and
irreducible.
\item The closure of any $X^{\vec \rho-s}$ (if non-empty) for any
$\vec \rho$ is a same irreducible component.
\item Given any point $x$
$\in$ $X$, the set of characters $\{\rho|$ $x$ $\in$ $X^{\rho-ss}\}$
is closed.
\end{enumerate}

Let $\vec \rho$ be a generic character (i.e. not on any walls) such
that $X^{\vec \rho-s}$ is non-empty. By assumptions 2 and 3, we have a
$G$-principal bundle $X^{\vec \rho-s}$ $\rightarrow $ $X$//$_{\vec \rho} G$ $=$
$X^{\vec \rho-s}$/$G$.
\begin{defn}
Let $\vec \rho_0$ be a character of $G$, we
denote $\mathcal L_{\vec \rho,\vec \rho_0}$ to be the \emph{line bundle} over
$X$//$_{\vec \rho} G$ by composing the transition functions of the
$G$-principal bundles with $\vec \rho_0$.
\label{notation: line bundle on GIT}
\end{defn}
In other words, viewing
$X^{\vec \rho-s}$/$G$ as a complex manifold, it has an open cover with
trivialization of $G$-fibers. The line bundle $\mathcal
L_{\vec \rho,\vec \rho_0}$ is by composing each transition function on
the overlap of charts by $\vec \rho_0$.

Now we are ready to list some properties from the variation
geometric invariant theory.
\begin{prop}
Let $X$ be an affine algebraic $G$-variety that satisfies the
assumptions 1 to 5,  and $\vec \rho$ be a generic character. The following properties hold:
\begin{enumerate}
\item $\Gamma$ $(X$//$_{\vec \rho} G$, $\mathcal L_{\vec \rho,\vec \rho_1}^{\otimes n})$
$\simeq$ $\mathbb
C[X^{\vec \rho-s}]^{G,\vec \rho_1^n}$.
\item Let $\vec \rho_+$ be a character of $G$ in the same chamber of $\vec \rho$, then $\mathbb
C[X^{\vec \rho-s}]^{G,\vec \rho_+^n}$ $=$ $\mathbb C[X]^{G,\vec \rho_+^n}$ for
$n\gg 1$ and $\mathcal L_{\vec \rho,\vec \rho_+}$ is ample. Let $\vec \rho_0$ be a
generic character on the wall of the $\vec \rho$-chamber,
then $\mathcal L_{\vec \rho,\vec \rho_0}$ is nef and semi-ample.
\item There is an inclusion $X^{\vec \rho_+-ss}$
$\subset$ $X^{\vec \rho_0-ss}$ inducing a canonical projective morphism
$\mathrm{pr}_+$: $X$//$_{\vec \rho_+}G$
$\rightarrow$ $X$//$_{\vec \rho_0}G$.
\item A curve $C$ (projective, smooth, connected) in $X$//$_{\vec \rho_+}G$
is contracted by $\mathrm{pr}_+$ if and only if it is contracted by
$X$//$_{\vec \rho_+}G$ $\rightarrow$ $\mathbf{Proj}\oplus_{n\geq0}
\Gamma(X$//$_{\vec \rho_+}G,\mathcal L_{\vec \rho_+,\vec \rho_0}^{\otimes n})$.
\item Let $\vec \rho_+$ and $\vec \rho_-$ be in two chambers on different sides
of the wall. Assume
that $X^{\vec \rho_{+}-s}$ and $X^{\vec \rho_{-}-s}$ are both non-empty, then the morphisms
$X$//$_{\vec \rho_{\pm}}G$ $\rightarrow$ $X$//$_{\vec \rho_0}G$ are proper and
birational. If they are both small, then the rational map
$X$//$_{\vec \rho_{-}}G$ $\dashrightarrow$ $X$//$_{\vec \rho_{+}}G$ is a flip
with respect to $\mathcal L_{\vec \rho_+,\vec \rho_0}$.
\end{enumerate}
\label{recollections
from VGIT}
\end{prop}
\begin{proof} 1. This is true for a general $G$-principal bundle by flat
descent theorem, see \cite{SGA} Expos\'{e} I, Th\'{e}or\`{e}me 4.5.

2 and 3. By the assumption 5,  $X^{\vec \rho-s}$ $\subset$ $X^{\vec \rho_*-ss}$ for $*=0$ or $+$.
By the  assumption 4, the natural map: $\mathbb C[X]^{G,\vec \rho_*^n}$
$\rightarrow$ $\mathbb C[X^{\vec \rho-s}]^{G,\vec \rho_*^n}$ $\simeq$ $\Gamma$
$(X$//$_{\vec \rho} G$, $\mathcal L_{\vec \rho,\vec \rho_*}^{\otimes n})$ is
injective for  $n\in \mathbb Z_{\geq 0}$. Hence the base
locus of $\mathcal L_{\vec \rho,\vec \rho_*}$ is empty. $\mathcal
R(X$//$_{\vec \rho}G$, $\mathcal L_{\vec \rho,\vec \rho_*})$ $\simeq$
$\bigoplus_{n\geq 0} \mathbb C[X^{\vec \rho-s}]^{G,\vec \rho_*^n}$ is finitely
generated over $\mathbb C$. The canonical morphism $X$//$_{\rho}G$
$\rightarrow$ \textbf{Proj}$\bigoplus_{n\geq 0} \mathbb
C[X^{\vec{\rho}-s}]^{G,\vec{\rho}_*^n}$ is birational and projective when
$X^{\vec{\rho}_*-s}$ is non-empty. Now we have series of morphisms:
\begin{center}
pr$_+$: $X$//$_{\vec{\rho}}G$ $\rightarrow$ \textbf{Proj}$\bigoplus_{n\geq
0} \mathbb C[X^{\vec{\rho}-s}]^{G,\vec{\rho}_*^n}$ $\rightarrow$
\textbf{Proj}$\bigoplus_{n\geq 0} \mathbb C[X]^{G,\vec{\rho}_*^n}$ $=$
$X$//$_{\vec{\rho}_*}G$.
\end{center}
The morphism pr$_+$ maps each $\vec{\rho}_*$ S-equivariant class to itself
set-theoretically. When $\vec{\rho}_+$ is in the same chamber of $\vec{\rho}$,
by the assumption 2, this is an isomorphism, implying that $\mathcal
L_{\vec{\rho},\vec{\rho}_+}$ must be ample and $\mathbb
C[X^{\vec{\rho}-s}]^{G,\vec{\rho}_*^n}$ $=$ $\mathbb C[X]^{G,\vec{\rho}_*^n}$ for $n$
large enough. By the definition of $\mathcal L_{\vec{\rho},\vec{\rho}_+}$, it extends
linearly  to a map from the space of $\mathbb R$-characters
of $G$ to NS$_{\mathbb R}$($X$//$_{\vec{\rho}}G$). Since all elements in
the $\vec{\rho}$ chamber are mapped
into the ample cone, $\vec{\rho}_0$ must be nef.\\

4. `$\Leftarrow$': The morphism 
\[X\text{//}_{\vec \rho_+}G \rightarrow X\text{//}_{\vec \rho_0}G=\mathbf{Proj}\bigoplus_{n\geq0} \mathbb C[X]^{G,\vec{\rho}_0^n}\]
factors via the morphism
\textbf{Proj}$\bigoplus_{n\geq0} \mathbb C[X^{\vec{\rho}_+-s}]^{G,\vec{\rho}_0^n}$
$\rightarrow$ \textbf{Proj}$\bigoplus_{n\geq0} \mathbb
C[X]^{G,\vec{\rho}_0^n}$. If $C$ is contracted at
\textbf{Proj}$\bigoplus_{n\geq0} \mathbb C[X^{\vec{\rho}_+-s}]^{G,\vec{\rho}_0^n}$,
then it is also contracted at \textbf{Proj}$\bigoplus_{n\geq0} \mathbb
C[X]^{G,\vec{\rho}_0^n}$.

`$\Rightarrow$': Suppose $C$ is contracted to a point by pr$_+$. Let $G'$ be the kernel of $\vec{\rho}_0$, we show that
there is a
subvariety $P$ in $X^{\vec{\rho}_+-s}$ such that
\begin{itemize}
\item[I.] $P$ is a $G'$-principal bundle, and the base space is projective, connected;
\item[II.] $F_{\vec{\rho}_+} (P)$ $=$ $C$.
\end{itemize}
Suppose we find such $P$, then any function $f$ in $\mathbb
C[X^{\vec{\rho}_+-s}]^{G,\vec{\rho}_0^n}$ is constant on each $G'$ fiber. Since
the base space is projective and connected, it must be a constant on
$P$. Since $F_{\vec{\rho}_+} (P)$ $=$ $C$, the value of
$f$ on $F_{\vec{\rho}_+}^{-1}(C)$ is determine by this constant. Hence the
canonical morphism contracts $C$ to a point.

We may assume $G'$ $\neq$ $G$, choose $N$ large enough
and finitely many $f_i$'s in $\mathbb C[X]^{G,\vec{\rho}_0^N}$ such that
$\bigcap_i \left(V(f_i) \cap F_{\vec{\rho}_0}^{-1}(\mathrm{pr}_+(C))\right)$ is empty. Since
all points in $F_{\vec{\rho}_0}^{-1}(\mathrm{pr}_+(C))$ are S-equivariant in
$X^{\vec{\rho}_0-ss}$, for each point $x$ in $F_{\vec{\rho}_+}^{-1}(C)$, $\overline{Gx}$ contains all minimum orbits
$Gy$ in $F_{\vec{\rho}_0}^{-1}(\mathrm{pr}_+(C))$. Choose $y$ in $F_{\vec{\rho}_0}^{-1}(\mathrm{pr}_+(C))$ such that $Gy$ is
closed in $X^{\vec{\rho}_0-ss}$, let $P_y$ be
\[\bigcap_{i}\{x\in F_{\vec{\rho}_+}^{-1}(C)|f_i(x) =f_i(y)\}.\] For any
$p$ $\in$ $C$, since $G$ is reductive and the $G$-orbit $\overline{F_{\vec{\rho}_+}^{-1}(p)}$
contains $y$, there is a subgroup $\beta$:
$\mathbb C^\times$ $\rightarrow$ $G$ and $x_p$ $\in$
$F^{-1}_{\vec{\rho}_+}(p)$ such that $y$ $\in$
$\overline{\beta(\mathbb C^\times)\cdot\{x_p\}}$. Since $y$ $\in$
$X^{\vec{\rho}_0-ss}$, there is a $\vec{\rho}_0^N$-semi-invariant $f_i$ such
that $f_i(y)\neq 0$. Therefore $\vec{\rho}_0\circ \beta$ $\neq$ $0$, and for any $\vec{\rho}_0$-semi-invariant function $f$, $f(x_p)$
$=$ $f(y)$. The point $x_p$ is in $P_y$ and therefore $F_{\vec\rho}(P_y) = C$. 

Let $G''$ be the kernel of
$\vec{\rho}_0^N$.  By the choices of $f_i$'s, another point $x_q$ on $Gx_p$
is in $P_y$ if and only if they are on the same $G''$-orbit. Since
$G$ acts freely on all stable points, $P_y$ becomes a $G''$
principal bundle over base $C$. As $[G'':G']$ is finite, we may
choose a connected component of $P_y$ such that viewing as a $G'$-principal
bundle, the induced morphism from the base
space to $C$ is finite. This component of $P_y$ then satisfies both condition I and II at the beginning.\\

5. This is due to Theorem 3.3 in \cite{Th}.
\end{proof}
\begin{rem}
When the difference between $X^{\vec{\rho}_+-s}$ and $X^{\vec{\rho}_--s}$ is of
codimension two in $X^{\vec{\rho}_+-s}\cup X^{\vec{\rho}_--s}$, since
$X^{\vec{\rho}_+-s}\cup X^{\vec{\rho}_--s}$ is smooth, irreducible and
quasi-affine by the second assumption, we have:
\begin{center}
$\mathbb C[X^{\vec{\rho}_+-s}]^{G,\vec{\rho}_-^n}$ $=$ $\mathbb
C[X^{\vec{\rho}_+-s}\cup X^{\vec{\rho}_--s}]^{G,\vec{\rho}_-^n}$ $=$ $\mathbb
C[X^{\vec{\rho}_--s}]^{G,\vec{\rho}_-^n}$ $=$ $\mathbb C[X]^{G,\vec{\rho}_-^n}$ for
$n\gg 0$.
\end{center}
In this case,  the birational morphism between $X^{s,\vec{\rho}_+}$ and
$X^{s,\vec{\rho}_-}$ identifies NS$_{\mathbb R} (X$//$_{\vec{\rho}_+}G$) and
NS$_{\mathbb R} (X$//$_{\vec{\rho}_-}G$). It maps $[\mathcal
L_{\vec{\rho}_+,\vec{\rho}_*}]$ to $[\mathcal L_{\vec{\rho}_-,\vec{\rho}_*}]$ for all
$\vec{\rho}_*$ in either $\vec{\rho}_+$ and $\vec{\rho}_-$ chamber. \label{remark on the divisor
glueing in flip case}
\label{rem:lbndlglue}
\end{rem}

\subsection{Walls-crossing as minimal model program}\label{sec2.6}
Let $w$ be a primitive character in $\mathrm{K}(\pp)$ such that $\ch_0(w)>0$. We can run the minimal model program for $\mathfrak M^s_{GM}(w)$ via wall crossing on the space of stability conditions.
\begin{theorem}
Adopt the notations as above, the actual walls $L_{w\sigma}$ (chambers) to the left of the vertical wall $L_{w\pm}$ in
the $\cccp$ is one-to-one corresponding to the
stable base locus decomposition walls (chambers) on one side (primitive side) of the divisor
cone of $\mathfrak M^s_{GM}(w)$.
\label{left half upper plane's main theorem in
the body}
\end{theorem}

\begin{proof}
Suppose $L=L_{w\sigma}$ passes through $\mathrm{MZ}_\mathcal E$ for an exceptional triple $\mathcal E$. By Lemma \ref{lemma: rhocharformular}, $L$ associates a character (up to a positive scalar) $\vec \rho_{L}$ to the group $G_{\vec n_w}/\mathbb C^\times$. By
Proposition \ref{prof:gitconst}, the moduli space
 $\mathfrak M^{ss}_{\sigma} (w)$ is constructed as the quotient space 
$\repqwa\lfss_{\det^{\vec\rho_{L}}}\left(G_{\vec n_w}/\mathbb C^{\times}\right)$. We first check that the $G$-variety $\repqwa$ satisfies the
assumptions of Proposition \ref{recollections from VGIT}. Assumption 1 is due to Proposition \ref{finite many
wall, t>0 is HilbS}. Assumption 2 is due to Corollary \ref{cor:smoothness} and Remark \ref{rem:smoothness}. Assumption 3 and 4 are due
to Theorem \ref{smoothness and irred property} and its proof. Assumption 5 is automatically satisfied in our case.

By Definition \ref{notation: line bundle on GIT}, the character $\vec \rho_{L}$ induces a divisor (up to a positive scalar) $[\mathcal L_{\vec \rho_L,\vec \rho_{L}}]$ on $\repqwa\lfss_{\det^{\vec\rho_L}}\left(G_{\vec n_w}/\mathbb C^{\times}\right)$. We start from the chamber on the left of the vertical wall, where
$\repqwa\lfss_{\det^{\vec\rho_L}}\left(G_{\vec n_w}/\mathbb C^{\times}\right)$ is isomorphic to $\mathfrak M^s_{GM}(w)$, and vary the stability to the wall near the tangent line of $\td_0$ across $w$. At an actual destabilizing wall $L$, let pr$_+$ be
the morphism 
\[\repqwa\lfss_{\det^{\vec\rho_{L+}}}\left(G_{\vec n_w}/\mathbb C^{\times}\right)\rightarrow \repqwa\lfss_{\det^{\vec\rho_{L}}}\left(G_{\vec n_w}/\mathbb C^{\times}\right)\]
as that in Proposition \ref{recollections from
VGIT}.
One of three different cases may happen:
\begin{enumerate}
\item pr$_+$ is a small contraction;
\item pr$_+$ is birational and has an exceptional divisor;
\item all objects in $\mathfrak M^s_{L}(w)$ becomes strictly semistable.
\end{enumerate}

By Proposition \ref{contracting is more than producing lemma}, in
Case 1, we get small contractions on both sides. By Property 5 in Proposition \ref{recollections from VGIT}, this is
the flip with respect to the divisor $[\mathcal
L_{\vec\rho_{L+},\vec\rho_L}]$. Since
the different locus between  $\repqwa^{\vec\rho_{L+}}$ and  $\repqwa^{\vec\rho_{L-}}$ is of codimension at least $2$, their divisor cones are identified with each other as
explained in Remark \ref{remark on the divisor glueing in flip
case}. In particular, before encountering any wall of Case 2 or 3, the divisor $[\mathcal L_{\vec\rho_{L+},\vec\rho_L}]$ is identified to a divisor $[\mathcal L_{\vec\rho_L}]$ on $\mathfrak M^s_{GM}(w)$. The flip $\mathfrak M^s_{L+}(w)\dashrightarrow\mathfrak M^s_{L-}(w)$ is with respect to this divisor.

In Case 2, by Proposition \ref{contracting is
more than producing lemma}, the morphism pr$_-$ on the left side  
\[\repqwa\lfss_{\det^{\vec\rho_{L-}}}\left(G_{\vec n_w}/\mathbb C^{\times}\right)\rightarrow \repqwa\lfss_{\det^{\vec\rho_{L}}}\left(G_{\vec n_w}/\mathbb C^{\times}\right)\]
does not contract any divisors. Hence the Picard number of
$\repqwa/_{\det^{\vec\rho_{L-}}}\left(G_{\vec n_w}/\mathbb C^{\times}\right)$  is $1$. By
Property 4 in Proposition \ref{recollections from VGIT}, Case 2 only
happens when the canonical model associated to $\mathcal
L_{\vec\rho_L}$
contracts a divisor, in other words, the divisor of $\mathcal L_{\vec\rho_L}$ on
$\mathfrak M^s_{GM}(w)$ is on the boundary of the movable cone. The next
destabilizing wall on the left corresponds to the zero divisor, it
must be Case 3. On the other hand, by Corollary \ref{cor:lastwallradiusgeq1}, Case 3 must happen at a wall before reaching the tangent line. This terminates the whole minimal model program.

In general, if the boundary of the Movable cone is
not the same as that of the Nef cone, then Case 2 happens.
Otherwise, Case 2 does not happen and the procedure ends up with a
Mori fibration of Case 3.
\end{proof}

\begin{rem}
On the vertical wall, the morphism $\mathrm{pr}_+$ is the Donaldson-Uhlenbeck morphism. If it contracts a divisor, the vertical wall corresponds to the movable boundary and the minimal model program stops. If $\mathrm{pr}_+$ is a small contraction, the wall-crossing behavior on the other side of the nef cone is the same as the wall-crossing behavior of $\mathfrak M^s_{GM}(\ch_0(w),-\ch_1(w),\ch_2(w))$ on the primitive side.
\label{rem:verticalwallmmp}
\end{rem}

\section{The last wall and criteria for actual walls}\label{sec3}

In this section, we give a description of the last wall (Section \ref{sec3.1}) and a numerical criteria of actual walls (Section \ref{sec3.2}). Section \ref{sec3.0} consists of several useful lemmas.

\subsection{Stable objects by extensions}\label{sec3.0}

The following lemma is useful to construct new stable objects after wall-crossing.

\begin{lemma}
Let $G$ and $F$ be two $\sigma_{s,q}$-stable objects of the same
phase, in particular, $\sigma_{s,q}$ is on the line $L_{GF}$.
Suppose we have \[\phi_{\sigma_{s,q+}}(G) >
\phi_{\sigma_{s,q+}}(F)\text{,  and } \Hom(G,F[1])\neq 0.\] Let $f$
be a non-zero element in $\Hom(G,F[1])$ and $C$ be the corresponding extension of $G$ by $F$, then $C$ is $\sigma_{s,q+}$-stable.
\label{lemma:extoftwostabobjsarestab}
\end{lemma}
\begin{proof}
By Corollary \ref{cor:stabonthesegment} and Proposition \ref{prop:commomareaofalggeo}, we may assume that $\sigma_{s,q}$ is
in a quiver region MZ$_\mathcal E$ so that $C$, $F$ and $G$ are
in the same heart $\mathcal A_{\mathcal E}[t]$ for a homological shift $t=0$ or $1$. We write $\sigma$ for
$\sigma_{s,q}$, and $\sigma_+$ for $\sigma_{s,q+}$.

We prove the lemma by contradiction, suppose $D$ is a $\sigma_+$-stable
sub-complex destabilizing $C$ in $\mathcal A_\mathcal E[t]$. We have the following diagram:
\[
\xymatrix{
  0 \ar[r]^{} & K \ar[d]_{} \ar[r]^{} & D \ar[d]^{} \ar[r]^{} & I \ar[d]_{} \ar[r]^{} & 0\\
  0 \ar[r]^{} & F \ar[r]^{} & C  \ar[r]^{} &  G \ar[r]^{} & 0 }
\]
such that the vertical maps are all injective in $\mathcal
A_{\mathcal E}[t]$. Three different cases may happen.

If $I=0$, then $\phi_{\sigma_+}(K) = \phi_{\sigma_+}(D)\geq
\phi_{\sigma_+}(C) >\phi_{\sigma_+}(F)$. But $F$ is also
$\sigma_+$-stable, this leads to a contradiction.

If $K=0$, then either
$\phi_{\sigma}(I)<\phi_\sigma(G)=\phi_\sigma(C)$ or $I=G$. The second case
that $I=G$ is impossible since the extension is non-splitting. In the
first case, as the phase function is continuous (by the support
property), we have $\phi_{\sigma_+}(I)<\phi_{\sigma_+}(C)$. Therefore the object $D$ does not destabilize $C$ at $\sigma_+$, which is a contradiction.

If both $K$ and $I$ are non-zero, then since $F$ and $G$ are $\sigma$-stable, 
$\phi_{\sigma}(I)\leq\phi_\sigma(G)=\phi_\sigma(C)$ and
$\phi_{\sigma}(K)\leq\phi_\sigma(F)=\phi_\sigma(C)$. When both equalities hold, we have $I=G$ and $K=F$, and in this case, $D=C$. If at least one of the equalities does not hold, then $\phi_{\sigma}(D)<\phi_\sigma(C)$. Again by the continuity of the phase function, we see that $\phi_{\sigma_+}(I)<\phi_{\sigma_+}(C)$ and get the contradiction.
\end{proof}

In general, we also need the following direct sum version, which can be proved in a similar way.
\begin{cor}
Let $G$ and $F$ be two $\sigma_{s,q}$-stable objects of the same
phase.
Suppose we have \[\phi_{\sigma_{s,q+}}(G) >
\phi_{\sigma_{s,q+}}(F)\text{,  and } \Hom(G,F[1])=n> 0.\] Let $f$
be a rank $m$ map in $\Hom(G,F^{\oplus m}[1])$ and $C$ be the object
extended by $G$ and $F^{\oplus m}$ via $f$, then $C$ is $\sigma_{s,q+}$-stable.
\label{cor:extoftwostabobjsarestab}
\end{cor}

Now we collect some geometric properties of the Le Potier curve. For an exceptional character $e$, by the Hirzebruch-Riemann-Roch formula, the equation for $L_{e^+e^r}$, i.e. $\chi(-,e')=0$, in the $\cccp$ is 
\[\ch_0(e)\che-\left(\ch_1(e)+\frac{3}{2}\ch_0(e)\right)\chy+\ch_2(e)+\frac{3}{2}\ch_1(e)+\ch_0(e)=0.\]
In particular, the slope of $L_{e^+e^r}$ is $\chy(e)+\frac{3}{2}$. The line $L_{e^+e^r}$ is parallel to $L_{ee(3)}$ and $L_{e^le(3)^r}$. A similar computation shows that the slope of $L_{e^+e^l}$ is $\chy(e)-\frac{3}{2}$.

We first want to prove the following result, which will be used to prove Lemma \ref{lemma:ext1nonvanishing}.

\begin{lemma}
1. Let $e$ be an exceptional character, then for any point $p$ on the line segment $l_{e^+e^r}$ (not on the boundary), the line $L_{ep}$ intersects the Le Potier curve $C_{LP}$ at two points. In addition, the $\frac{\ch_1}{\ch_0}$-length of these two points is greater than $3$.

2. Let $u$ and $v$ be two Chern characters with $\ch_0(u)$, $\ch_0(v)>0$ on the $C_{LP}$ such that their $\frac{\ch_1}{\ch_0}$-length is greater than $3$, then $\chi(u,v)>0$, $\chi(v,u)>0$.
\label{lem:l3}
\end{lemma}


\begin{center}
\begin{tikzpicture}[domain=1:5]

\tikzset{%
    add/.style args={#1 and #2}{
        to path={%
 ($(\tikztostart)!-#1!(\tikztotarget)$)--($(\tikztotarget)!-#2!(\tikztostart)$)%
  \tikztonodes},add/.default={.2 and .2}}
}

\newcommand\XA{0}
\newcommand\XP{2.3}

\draw [name path =C0, opacity=0.1](-3,4.5) parabola bend (0,0) (3,4.5)
 node[right, opacity =0.5] {$\bd_0$};

\draw [name path = C1, opacity=0.5](-3,4) parabola bend (0,-0.5) (3,4)
 node[right] {$\bd_{\frac{1}{2}}$};

\draw [name path =C2, opacity=0.5](-3,3.5) parabola bend (0,-1) (3,3.5)
 node[right] {$\bd_1$};


\coordinate (B3) at (-3,4.5);
\coordinate (B2) at (-2,2);

\coordinate (B1) at (-1,0.5);
\coordinate (A0) at (0,0);
\coordinate (A1) at (1,0.5);
\coordinate (A2) at (2,2);
\coordinate (A3) at (3,4.5);

\coordinate (E0) at (0,-1);
\draw (E0) node [below right] {};

\coordinate (E1) at (1,-0.5);
\draw (E1) node [below right] {};

\coordinate (E2) at (2,1);
\draw (E2) node [below right] {};

\coordinate (E3) at (3,3.5);
\draw (E3) node [below right] {};

\coordinate (F1) at (-1,-0.5);
\draw (F1) node [below left] {};

\coordinate (F2) at (-2,1);
\draw (F2) node [below left] {};

\coordinate (F3) at (-3,3.5);
\draw (F3) node [below left] {};

\draw [name path =L0,opacity =\XA] (E0) -- (B2);
\draw [name intersections={of=C1 and L0},  thick] (E0) -- (intersection-1);
\draw [name path =R13,opacity =\XA] (F3) -- (B2);
\draw [name intersections={of=C1 and R13},  thick] (F3) -- (intersection-1);

\draw [name path =L1,opacity =\XA] (E1) -- (B1);
\draw [name intersections={of=C1 and L1},  thick] (E1) -- (intersection-1);
\draw [name path =R12,opacity =\XA] (F2) -- (B1);
\draw [name intersections={of=C1 and R12},  thick] (F2) -- (intersection-1);

\draw [name path =L2,opacity =\XA] (E2) -- (A0);
\draw [name intersections={of=C1 and L2},  thick] (E2) -- (intersection-1);
\draw [name path =R11,opacity =\XA] (F1) -- (A0);
\draw [name intersections={of=C1 and R11},  thick] (F1) -- (intersection-1);

\draw [name path =L3,opacity =\XA] (E3) -- (A1);
\draw [name intersections={of=C1 and L3},  thick] (E3) -- (intersection-1);
\draw [name path =R0,opacity =\XA] (E0) -- (A1);
\draw [name intersections={of=C1 and R0},  thick] (E0) -- (intersection-1);

\draw [name path =L12,opacity =\XA] (F2) -- (B3);
\draw [name intersections={of=C1 and L12},  thick] (F2) -- (intersection-1);
\draw [name path =L11,opacity =\XA] (F1) -- (B2);
\draw [name intersections={of=C1 and L11},  thick] (F1) -- (intersection-1);

\draw [name path =R1,opacity =\XA] (E1) -- (A2);
\draw [name intersections={of=C1 and R1},  thick] (E1) -- (intersection-1);
\draw [name path =R2,opacity =\XA] (E2) -- (A3);
\draw [name intersections={of=C1 and R2},  thick] (E2) -- (intersection-1);

\coordinate (S3) at (-2.5,2.5);
\draw (S3) node [below left] {};

\coordinate (S2) at (-1.5,0.5);

\coordinate (S1) at (-.5,-0.5);

\coordinate (T1) at (.5,-0.5);

\coordinate (T2) at (1.5,0.5);

\coordinate (T3) at (2.5,2.5);
\draw (T3) node [below right] {};

\draw [name path =RS3,opacity =\XA] (S3) -- (B3);
\draw [name intersections={of=C1 and RS3},  thick] (S3) -- (intersection-1);
\draw [name path =LS3,opacity =\XA] (S3) -- (A0);
\draw [name intersections={of=C1 and LS3},  thick] (S3) -- (intersection-1) node [right] {$e'^l$};

\draw [name path =RS2,opacity =\XA] (S2) -- (B2);
\draw [name intersections={of=C1 and RS2},  thick] (S2) -- (intersection-1)node [right] {$e'^r$};
\draw [name path =LS2,opacity =\XA] (S2) -- (A1);
\draw [name intersections={of=C1 and LS2},  thick] (S2) -- (intersection-1);

\draw [name path =RS1,opacity =\XA] (S1) -- (B3);
\draw [name intersections={of=C1 and RS1},  thick] (S1) -- (intersection-1);
\draw [name path =LS1,opacity =\XA] (S1) -- (A2);
\draw [name intersections={of=C1 and LS1},  thick] (S1) -- (intersection-1);

\draw [name path =RT1,opacity =\XA] (T1) -- (B2);
\draw [name intersections={of=C1 and RT1},  thick] (T1) -- (intersection-1);
\draw [name path =LT1,opacity =\XA] (T1) -- (A3);
\draw [name intersections={of=C1 and LT1},  thick] (T1) -- (intersection-1);

\draw [name path =RT2,opacity =\XA] (T2) -- (B1);
\draw [name intersections={of=C1 and RT2},  thick] (T2) -- (intersection-1);
\draw [name path =LT2,opacity =\XA] (T2) -- (A2);
\draw [name intersections={of=C1 and LT2},  thick] (T2) -- (intersection-1);

\draw [name path =RT3,opacity =\XA] (T3) -- (A0);
\draw [name intersections={of=C1 and RT3},  thick] (T3) -- (intersection-1);
\draw [name path =LT3,opacity =\XA] (T3) -- (A3);
\draw [name intersections={of=C1 and LT3},  thick] (T3) -- (intersection-1);

\coordinate (E) at (2,2);
\draw (E) node {$\bullet$} node [above left]{$e$};

\coordinate (P) at (\XP,3.5*\XP-6);
\draw (P) node {$\bullet$} node [below right]{$p$};

\draw [add =20 and 6,name path = IT, dashed] (E) to (P);
\draw [name intersections={of=L12 and IT}] (P) -- (intersection-1) node [below left] {};

\node at (-1.9,1.305) {$\bullet$};
\node[above right] at (-1.9,1.305) {$q$};


\draw[->,opacity =0.3] (-4,0) -- (4,0) node[above right] {$\frac{\ch_1}{\ch_0}$};

\draw[->,opacity=0.3] (0,-2)-- (0,0) node [above right] {O} --  (0,4.5) node[right] {$\frac{\ch_2}{\ch_0}$};

\end{tikzpicture}

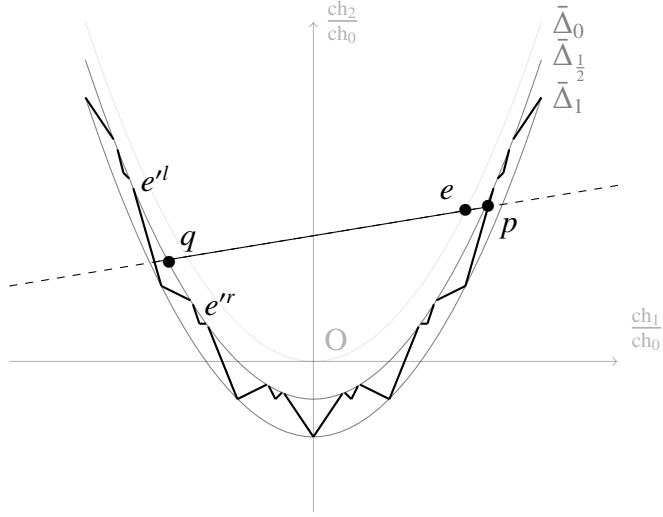
Figure: The intersection of $L_{ep}$ with $C_{LP}$.
\end{center}

\begin{proof}
1. We first show that $L_{ep}$ only intersects $C_{LP}$ at two points. Since any point on $C_{LP}$ to the right of $e^r$ is above the line $L_{ee^r}$, we only need to consider points to the left of $e$.

For any $e'^r$ to the left of $e$ that is above $L_{ep}$, it is also strictly above $L_{ee^r}$. Since $e$, $e^r$ and $e(-3)^r$ are collinear, $e'$ is to the left of $e(-3)$. In other words, $e'$ satisfies that $\chy(e')<\chy(e)-3$. The slope of $L_{ep}$ $>$ the slope of $L_{ee'^l}$ $>$ the slope of  $L_{e'(3)^re'^l}$ $=$ the slope of $L_{e'^+e'^r}$. Therefore, $L_{ep}$ does not intersect $l_{e'^le'^+}$ or $l_{e'^+e'^r}$.

For any $e'^l$ below $L_{ep}$ to the left of $e$, the line segments $l_{e'^le'^+}$ and $l_{e'^+e'^r}$ are below $\bd_{\frac{1}{2}}$. The segment of $\bd_{\frac{1}{2}}$ between $e'^l$ and $e'^r$ is below $L_{ep}$, hence $l_{e'^le'^+}$ and $l_{e'^+e'^r}$ are below $L_{ep}$, and they do not intersect $L_{ep}$. Let $q$ be the intersection point of $L_{ep}$ and $\bd_{\frac{1}{2}}$ ( there are two such points and we consider the one to the left of $e$). When $q$ is not on any segment of $\bd_{\frac{1}{2}}$ between $e'^l$ and $e'^r$, the intersection points of $L_{ep}$ and $C_{LP}$ are $q$ and $p$. When $q$ is on the segment between $e'^l$ and $e'^r$ for an exceptional character $e'$, the second intersection point is either on $l_{e'^le'^+}$ or $l_{e'^+e'^r}$. So there is only one intersection point other than $p$.

The points $e$, $e^r$ and $e(-3)^r$ are collinear, and the $\chy$-length of $e^r$ and $e(-3)^r$ is $3$. Since the $\chy$-length of $L_{ep}\cap \bd_{\frac{1}{2}}$ is increasing when $p$ is moving from $e^r$ to $e^+$, the $\frac{\ch_1}{\ch_0}$-length of $L_{ep}\cap \bd_{\frac{1}{2}}$ is greater than $3$. Therefore, the $\frac{\ch_1}{\ch_0}$-length of $L_{ep}\cap C_{LP}$ is greater than $3$.

2. Suppose $u$ is on $\bd_{\frac{1}{2}}$, then the line $\chi(u,-)=0$ in the $\cccp$ is $L_{uu(-3)}$. Hence the point $v$ in the $\cccp$ is above $L_{uu(-3)}$. As $\ch_0(u)$ and $\ch_0(v)$ are positive, $\chi(u,v)>0$. $\chi(v,u)>$ can be proved similarly.

Suppose $u$ is on $l_{e^+e^r}$ for an exceptional $e$, we first show that $\chi(u,v)>0$. The line $\chi(u,-)=0$ passes through $e$, and both $e^r$ and $e^l$ are below the line $\chi(u,-)=0$. By the $\chy$-length assumption, $v$ is above both $L_{ee^r}$ and $L_{ee^l}$. Therefore, $v$ is above the line $\chi(u,-)=0$. Since $\ch_0(u)$ and $\ch_0(v)$ are positive, $\chi(u,v)>0$.

The line $\chi(-,u)=0$ in the $\cccp$ passes through $e(3)$, and intersects $l_{e(3)e(3)^r}$. If $v$ is on the line segment $l_{e(3)e(3)^r}$, by the case that $u$ is on $l_{e^+e^r}$, we get $\chi(v,u)>0$. If $v$ is not on the line segment $l_{e(3)e(3)^r}$, then by the assumption on the $\chy$-length, $v$ is above the curve $\chi(-,u)=0$, we also get $\chi(v,u)>0$.

The case that $u$ is on $l_{e^+e^l}$ can be proved in the same way.
\end{proof}

Now we can state an important lemma. A similar definition also appears in \cite{CH2}.

\begin{lemma}
Let $u$ and $v$ be Chern characters such that:
\begin{enumerate}
\item $u$ and $v$ are not inside the Le Potier cone.
\item $\Delta(v,u)\geq 0 $.
\item In the $\cccp$,
$L_{uv}$ intersects $C_{LP}$ at two points
and the  $\frac{\ch_1}{\ch_0}$-length between them is greater than
$3$.
\end{enumerate}
Then we have
\[\chi(u,v),\ \chi(v,u)<0.\]
\label{lemma:ext1nonvanishing}
\end{lemma}
\begin{rem}
When both $\ch_0(u)$ and $\ch_0(v)$ are $0$, the third condition does
not make sense. But the statement still holds if the first two
conditions hold. To see this, note that by the second condition,
\[\chi(v,w)=\chi(w,v)=-2\Delta(v,w)=-\ch_1(v)\ch_1(w)\leq 0.\]
Now the first condition implies that $\ch_1(w)$ and
$\ch_1(v)$ are both non-zero, so
\[-\ch_1(v)\ch_1(w)< 0.\]
\label{remark:ext1nonvanishing}
\end{rem}
\begin{proof}
By the first condition, $u$ and $v$ are below $\bd_{\frac{1}{2}}$.
By the third condition, let $f_1$ and $f_2$ be two characters corresponding to the intersection points of $L_{uv}$ and
$C_{LP}$ such that $\ch_0(f_1)>0$, $\ch_0(f_2)>0$ and $\chy(f_1)>\chy(f_2)$.

We may assume that $v=a_1f_1-a_2f_2$ and $u= b_1f_1-b_2f_2$ for some
non-zero real numbers $a_1$, $a_2$, $b_1$, $b_2$, since $u$ and $v$ are not inside Cone$_{LP}$, we see $a_1$, $a_2$ have the same sign and $b_1$, $b_2$ have the same sign. Moreover, by the second condition, we have
\begin{equation}
\Delta(v+au,v+au)\geq \Delta(v-au,v-au) \label{eq1}
\end{equation}
for any positive number $a$. Hence $a_1$, $a_2$, $b_1$, $b_2$ all have the same sign. Without loss of generality, we may assume they are all positive.

As $f_i$ is on
$C_{LP}$, we have
\[\chi(f_1,f_1),\; \chi(f_2,f_2)\leq 0.\]
By the third condition, the $\frac{\ch_1}{\ch_0}$-distance of $f_1$ and $f_2$ is greater than $3$. By Lemma \ref{lem:l3},
\[\chi(f_1,f_2)>0,\chi(f_2,f_1)>0.\]
Combining these results, we have
\[\chi(u,v) \leq -b_1a_2\chi(f_1,f_2) -b_2a_1\chi(f_2,f_1) <0,\]
and
\[\chi(v,u) \leq -b_2a_1\chi(f_1,f_2) -b_1a_2\chi(f_2,f_1) <0.\]
\end{proof}

Note that if we have stable objects $A$ and $B$ of characters $u$ and $v$ respectively, satisfying the conditions in the lemma, then the lemma implies that Ext$^1(A,B)>0$ and Ext$^1(B,A)>0$. By Lemma \ref{lemma:extoftwostabobjsarestab}, this implies the existence of stable objects as extensions on both sides. This observation will be used in the proof of the last wall to show the non-emptiness of the moduli, and in the proof of the actual walls to show the existence of objects destabilized on each side of the wall.

\subsection{The Last Wall} \label{sec3.1}
In this section, we describe the last wall for a given character $w$ that is not inside the Le Potier cone Cone$_{LP}$. By the {\em last wall} of $w$, we mean that for $P\in\bd_{<0}$, there is $\sigma_P$-stable objets of character $w$ or $w[1]$ if and only if $P$ is above the last wall. By result from Section 3, this wall corresponds to the boundary of the effective cone when running MMP. The last wall is first computed in \cite{CHW} and \cite{Woolf} by Coskun, Huizenga and Woolf. We would like to state the result based on our set-up and give a different proof.
To describe the last wall for character $w$,  we first define the exceptional bundle associated to $w$.

\begin{defn}
Let $E$ be an exceptional bundle, we define $\mathfrak R_E$ to be
\emph{the closure of the region} bounded by $L_{e(-3)^ree^r}$,
$l_{e^re^+}$, $l_{e^+e^l}$ and $L_{e^le(-3)e(-3)^l}$ in the $\cccp$ (see the picture below). Symmetrically, we define $\mathfrak L_E$ to be the closure of the region bounded by $L_{E(3)^lee^l}$,
$l_{e^le^+}$, $l_{e^+e^r}$ and $L_{e^rE(3)E(3)^r}$ in the $\cccp$.
\label{def:regionofexcdesarea}
\end{defn}

\begin{center}
\begin{tikzpicture}[domain=1:5]

\tikzset{%
    add/.style args={#1 and #2}{
        to path={%
 ($(\tikztostart)!-#1!(\tikztotarget)$)--($(\tikztotarget)!-#2!(\tikztostart)$)%
  \tikztonodes},add/.default={.2 and .2}}
}

\newcommand\XA{0}
\newcommand\XP{2.3}

\draw [name path =C0](-2.5,3.125) parabola bend (0,0) (2,2)
 node[right] {$\bd_0$};

\draw [name path = C1, opacity=1](-2.5,2.625) parabola bend (0,-0.5) (2,1.5)
 node[right] {$\bd_{\frac{1}{2}}$};


\coordinate (B3) at (-3,4.5);
\coordinate (B2) at (-2,2);

\coordinate (B1) at (-1,0.5);
\coordinate (A0) at (0,0);
\coordinate (A1) at (1,0.5);
\coordinate (A2) at (2,2);
\coordinate (A3) at (3,4.5);

\coordinate (E0) at (0,-1);
\draw (E0) node [below right] {};

\coordinate (E1) at (1,-0.5);
\draw (E1) node [below] {$e^+$};

\coordinate (EE1) at (1,0.5);
\draw (EE1) node [above] {$e$} node {$\bullet$};

\coordinate (E2) at (2,1);
\draw (E2) node [below right] {};

\coordinate (E3) at (3,3.5);
\draw (E3) node [below right] {};

\coordinate (F1) at (-1,-0.5);
\draw (F1) node [below left] {};

\coordinate (F2) at (-2,1);
\draw (F2) node [below left] {};

\coordinate (FF2) at (-2,2);
\draw (FF2) node [above right] {$e(-3)$} node {$\bullet$};

\draw [name path =L1,opacity =\XA] (E1) -- (B1);
\draw [name intersections={of=C1 and L1},  thick] (E1) -- (intersection-1) node [name = AA1] {$\bullet$} node [below] {$e^l$};

\draw [name path =R1,opacity =\XA] (E1) -- (A2);
\draw [name intersections={of=C1 and R1},  thick] (E1) -- (intersection-1) node [name = AA2] {$\bullet$} node [below right] {$e^r$};

\draw [name path =R12,opacity =\XA] (F2) -- (B1);
\draw [name intersections={of=C1 and R12},  thick] (F2) -- (intersection-1) node [name = BB2] {$\bullet$} node [below left] {$e(-3)^r$};

\draw [name path =L12,opacity =\XA] (F2) -- (B3);
\draw [name intersections={of=C1 and L12},  thick] (F2) -- (intersection-1) node [name = BB1] {$\bullet$} node [left] {$e(-3)^l$};

\draw [add =0 and 0.5, dashed] (BB2) to (AA2);
\draw [add =-1 and 1] (BB2) to (AA2) node [right] {$L_{e(-3)^ree^r}$};
\draw [add =0 and 0.3, dashed] (BB1) to (AA1);
\draw [add =-1 and 1] (BB1) to (AA1) node [right] {$L_{e^le(-3)e(-3)^l}$};

\draw node at (3,-1) {$\mathfrak R_E$};


\draw[->,opacity =0.3] (-3,0) -- (4,0) node[above right] {$\frac{\ch_1}{\ch_0}$};

\draw[->,opacity=0.3] (0,-2)-- (0,0) node [above right] {O} --  (0,3.5) node[right] {$\frac{\ch_2}{\ch_0}$};

\end{tikzpicture}

Figure: The region of $\mathfrak R_E$.
\end{center}


The following property translates an important technical result in \cite{CHW} into our set-up.
\begin{prop}[Theorem 4.1 in \cite{CHW}]
The regions associated to the exceptional bundles cover all rational points not above the Le Potier curve.
\[\coprod_{E \text{ exc}}\mathfrak R_E\ \supset \ \mathrm P(\mathrm K(\pp))\setminus \tilde C_{LP}.\]
A similar statement holds for $\mathfrak L_E$.
\end{prop}
\begin{proof}
Let $w$ be a reduced character in P$(\mathrm K(\pp))$
not above $C_{LP}$. There is a unique line
$L^{\bd_{\frac{1}{2}}}_w$ through $w$ on the $\cccp$ such that it intersects with $\bd_{\frac{1}{2}}$ at two points $f_1$ and $f_2$ ,
both of which are to the left of $w$ and their
$\frac{\ch_1}{\ch_0}$-length is $3$. Let $f_1$ be the points with larger
$\frac{\ch_1}{\ch_0}$. By Theorem 4.1 in \cite{CHW},  there is a unique exceptional bundle $E$
such that on the curve $\bd_{\frac{1}{2}}$, $f_1$ is on the
segment between $e^l$ and $e^r$. For any character $u$ on the line
$L_{f_1f_2}$, we have $\chi(f_1,u) = 0$, hence $\chi(f_1,w)=0$. The points $e^r$ and $e^l$ are on the different sides of the line $\chi(-,w)=0$, therefore
\[\chi(e^l,w)\cdot\chi(e^r,w)<0.\]
Note that the boundary $L_{E(-3)^rEE^r}$ is the line:
$\chi(e^r,-)=0$, and the boundary $L_{E^lE(-3)E(-3)^l}$ is the line:
 $\chi(e^l,-)=0$. Hence, $w$ is in $\mathfrak R_E$.
\end{proof}

\begin{rem}
It is possible to show that $L^{\bd_{\frac{1}{2}}}_w$ must intersect a line segment $l_{e^le^r}$ without using Theorem 4.1 in \cite{CHW}, but the argument is rather involved. The sketch of argument is as follows: 1. If $L^{\bd_{\frac{1}{2}}}_w$ does not intersect any line segment $l_{e^le^r}$, then for any exceptional bundle $E$ with character below $L^{\bd_{\frac{1}{2}}}_w$, by Proposition \ref{prop:wisnotstabbelowexc}, $\mathfrak M^s_\sigma(w)$ is empty for $\sigma$ below $L_{wE}$. Therefore, $\mathfrak M^s_\sigma(w)$ is empty for $\sigma$ below $L^{\bd_{\frac{1}{2}}}_w$. 2. By the same argument for the last wall and Lemma \ref{lemma:ext1nonvanishing}, $\mathfrak M^s_\sigma(w)$ is non-empty for $\sigma$ on $L^{\bd_{\frac{1}{2}}}_w$. This leads to the contradiction.
\end{rem}

Thanks to this result, we can introduce the following definition, which will be related to the last wall.

\begin{defn}
Let $w$ be a character not inside  $\Cone_{LP}$ (see Definition \ref{defn:lpcurve}), we define the
exceptional bundle $E_w$ associated to $w$ to be the unique one such
that $\mathfrak R_{E_w}$ contains $w$. Similarly we have the definition
of $E^{\text{(rhs)}}_w$ according to $\mathfrak L_E$.
\label{defn:Ew}
\end{defn}
\begin{rem}[Torsion Case]
In the case that $\ch_0(w)=0$ and $\ch_1(w)>0$, $E_w$ is the unique exceptional bundle such that 
\begin{center}
the slope of $L_{e(-3)e^l}$ $<$ $\frac{\ch_2}{\ch_1}(w)$ $<$ the slope of $L_{ee^r}$.
\end{center}
The bundle $E^{\text{(rhs)}}_w$ is not defined in the torsion case.
\label{rem:torcase}

\end{rem}

Now we can state the location of the last wall.

\begin{defn}
Let $w$ be a character (not necessarily primitive) not inside  $\Cone_{LP}$ (may be on the boundary but not at the origin, see Definition \ref{defn:lpcurve})
and $E=E_w$ be its associated exceptional vector bundle. We
define the last wall $L_w^{\text{last}}$ of $w$ according to three
different cases:
\begin{enumerate}
\item If $w$ is above $L_{e^+e(-3)^+}$, then
$L_w^{\text{last}}$ $:=$ $L_{we}$.
\item If $w$ is below $L_{e^+e(-3)^+}$, then
$L_w^{\text{last}}$ $:=$ $L_{we(-3)}$.
\item If $w$ is on $L_{e^+e(-3)^+}$, then
$L_w^{\text{last}}$ $:=$ $L_{e^+e(-3)^+}$.
\end{enumerate}
The last wall $\lrw_w$ on the right side to the vertical wall is defined similarly by using $E^{\text{(rhs)}}$. The torsion character does not have $E^{\text{(rhs)}}$ or $\lrw$.
\label{def:lastwall}
\end{defn}
In the cartoon below, $F_i$ is of Case $i$ in the definition respectively.


\begin{center}
\begin{tikzpicture}[domain=1:5]

\tikzset{%
    add/.style args={#1 and #2}{
        to path={%
 ($(\tikztostart)!-#1!(\tikztotarget)$)--($(\tikztotarget)!-#2!(\tikztostart)$)%
  \tikztonodes},add/.default={.2 and .2}}
}

\newcommand\XA{0}
\newcommand\XP{2.3}

\draw [name path =C0,opacity =0.4](-2.5,3.125) parabola bend (0,0) (2,2)
 node[right] {$\bd_0$};

\draw [name path = C1, opacity=0.4](-2.5,2.625) parabola bend (0,-0.5) (2,1.5)
 node[right] {$\bd_{\frac{1}{2}}$};


\coordinate (B3) at (-3,4.5);
\coordinate (B2) at (-2,2);

\coordinate (B1) at (-1,0.5);
\coordinate (A0) at (0,0);
\coordinate (A1) at (1,0.5);
\coordinate (A2) at (2,2);
\coordinate (A3) at (3,4.5);

\coordinate (E0) at (0,-1);
\draw (E0) node [below right] {};

\coordinate (E1) at (1,-0.5);
\draw (E1) node [below] {$e^+$};

\coordinate (EE1) at (1,0.5);
\draw (EE1) node [above] {$e$} node {$\bullet$};

\coordinate (E2) at (2,1);
\draw (E2) node [below right] {};

\coordinate (E3) at (3,3.5);
\draw (E3) node [below right] {};

\coordinate (F1) at (-1,-0.5);
\draw (F1) node [below left] {};

\coordinate (F2) at (-2,1);
\draw (F2) node [below left] {};

\coordinate (FF2) at (-2,2);
\draw (FF2) node [right] {$e(-3)$} node {$\bullet$};

\draw [name path =L1,opacity =\XA] (E1) -- (B1);
\draw [name intersections={of=C1 and L1},  thick] (E1) -- (intersection-1) node [name = AA1] {$\bullet$} node [below] {$e^l$};

\draw [name path =R1,opacity =\XA] (E1) -- (A2);
\draw [name intersections={of=C1 and R1},  thick] (E1) -- (intersection-1) node [name = AA2] {$\bullet$} node [below right] {$e^r$};

\draw [name path =R12,opacity =\XA] (F2) -- (B1);
\draw [name intersections={of=C1 and R12},  thick] (F2) -- (intersection-1) node [name = BB2] {$\bullet$} node [below left] {$e(-3)^r$};

\draw [name path =L12,opacity =\XA] (F2) -- (B3);
\draw [name intersections={of=C1 and L12},  thick] (F2) -- (intersection-1) node [name = BB1] {$\bullet$} node [left] {$e(-3)^l$};

\draw [add =0 and 0.5, dashed] (BB2) to (AA2);
\draw [add =-1 and 1] (BB2) to (AA2);
\draw [add =0 and 0.3, dashed] (BB1) to (AA1);
\draw [add =-1 and 1] (BB1) to (AA1);

\draw [add =0 and 1, dashed] (BB2) to (AA1) node {$\bullet$} node[right]{$F_3$};
\draw [add =0 and 1.5, dashed] (BB2) to (AA1);
\draw [add =-1 and 0.5, dashed] (AA1) to (BB2) node [left] {$\lw_{F_3}$};

\coordinate (G1)  at (3,-1);
\draw (G1) node {$\bullet$} node [above right] {$F_1$};
\draw [add =0  and 1.9] (G1) to (EE1) node [above] {$\lw_{F_1}$};

\coordinate (G2)  at (3,-2);
\draw (G2) node {$\bullet$} node [below right] {$F_2$};
\draw [add =0  and 0.3] (G2) to (FF2) node [above] {$\lw_{F_2}$};





\end{tikzpicture}

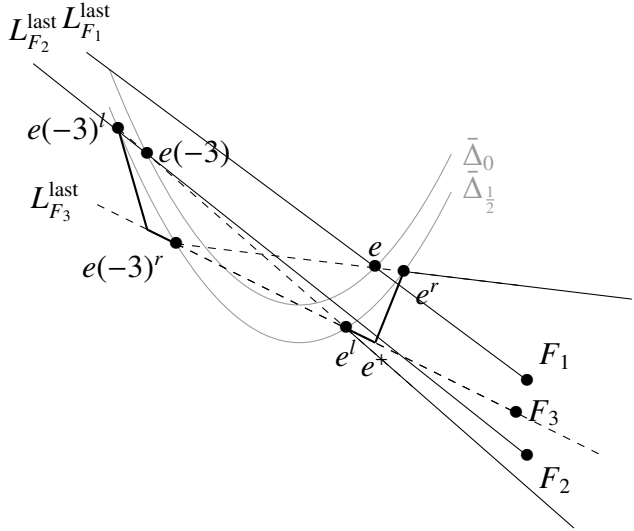
Figure: Three different cases of the last wall.
\end{center}


The following lemma shows that for stability conditions below the wall $\lw_w$ ($\lrw_w$), there is no stable object with character $w$.

\begin{lemma}
Let $w$ be a character in $\mathrm K(\pp)$ not inside  $\Cone_{LP}$; $\sigma$ be a geometric stability condition in $\bd_{<0}$ below $\lw_w$ or $\lrw_w$. Then $\mathfrak M^s_\sigma(w)$ and $\mss (-w)$ are both empty.
\label{lemma:lw}
\end{lemma}
\begin{proof}
We prove the lemma in the case for $\lw_w$. The $\lrw_w$ case can be proved similarly. We may assume $\ch_0(w)\geq 0$, since otherwise $\mathfrak M^s_\sigma(w)$ is empty when $\sigma$ is to the left of the vertical wall $L_{w\pm}$. When $w$ is of Case 1 or 3 in the Definition \ref{def:lastwall}, the statement follows from Proposition \ref{prop:wisnotstabbelowexc} directly.

When $w$ is of Case 2 in Definition \ref{def:lastwall}, we have $\chi(w,E_w(-3))<0$. For any $\sigma$-stable $F$ with character $w$,  Hom$(F,E_w(-3)[t])$ may be nonzero only when $0\leq t\leq 3$. Since $F$ is in Coh$_{\# s_\sigma}$,  $\Hom(E_w,\HH^{-1}(F))=0$. By Serre duality,
\[\hom(F,E_w(-3)[3]) = \hom(E_w,F[-1]) =  \hom(E_w,\HH^{-1}(F))=0.\]

On the other hand, when $\sigma$ is below $\lw_w$ and inside $\bd_{<0}$, by Corollary \ref{cor:regionofEstab}, $E_w(-3)[1]$ is $\sigma$-stable. By Lemma \ref{lemma:slopecompare}, $\phi_\sigma(E_w(-3)[1])<\phi_\sigma(F)$. Therefore, Hom$(F,E_w(-3)[1])= 0$. This leads to a contradiction to the inequality that $\chi(w,E_w(-3))<0$.
\end{proof}

The existence of stable objects before the last wall is more complicated. This is first proved by Coskun, Huizenga and Woolf. The authors write down the generic slope stable coherent sheaves build by exceptional bundles and show that these objects do not get destabilized before the last wall. Our approach is more close to the idea of Bayer and Macr{\`{i}} for K3 surfaces. We aim to show that for each wall-crossing before the last wall, new stable objects (extended by two objects) are generated on both sides, hence the moduli space is non-empty. We may benefit from this approach since the similar techniques can be applied in the criteria for actual walls.

\begin{lemma}
Let $w$ be a character $\mathrm K(\pp)$ with $\ch_0(w)> 0$ and not inside $\Cone_{LP}$. Let $\sigma$ be a geometric stability condition. Assume that the wall $L_{w\sigma}$ is between the vertical wall $L_{w\pm}$ and
\begin{itemize}
\item $\lw_w$, when $w$ is of Case 1 or 2 in Definition \ref{def:lastwall};
\item $L_{wE_w(-3)}$, when $w$ is of Case  3 in Definition \ref{def:lastwall}.
\end{itemize}
Let $v\in\mathrm K(\pp)$ be a character on $L_{w\sigma}$ such that $\ch_0(v)\geq 0$ and $\frac{\ch_1(v)}{\ch_0(v)}>\frac{\ch_1(w)}{\ch_0(w)}$, then the wall $L_{v\sigma}$ is between $L_{v\pm}$ and $\lw_v$.
\label{lemma:largerslopeisstable}
\end{lemma}


\begin{center}
\begin{tikzpicture}[domain=1:5]

\tikzset{%
    add/.style args={#1 and #2}{
        to path={%
 ($(\tikztostart)!-#1!(\tikztotarget)$)--($(\tikztotarget)!-#2!(\tikztostart)$)%
  \tikztonodes},add/.default={.2 and .2}}
}

\newcommand\XA{0}
\newcommand\XP{2.3}

\draw [name path =C0,opacity =0.4](-2.5,3.125) parabola bend (0,0) (2,2)
 node[right] {$\bd_0$};

\draw [name path = C1, opacity=0.4](-2.5,2.625) parabola bend (0,-0.5) (2,1.5)
 node[right] {$\bd_{\frac{1}{2}}$};


\coordinate (B3) at (-3,4.5);
\coordinate (B2) at (-2,2);

\coordinate (B1) at (-1,0.5);
\coordinate (A0) at (0,0);
\coordinate (A1) at (1,0.5);
\coordinate (A2) at (2,2);
\coordinate (A3) at (3,4.5);

\coordinate (E0) at (0,-1);
\draw (E0) node [below right] {};

\coordinate (E1) at (1,-0.5);

\coordinate (EE1) at (1,0.5);
\draw (EE1) node [above] {$e$} node {$\bullet$};

\coordinate (E2) at (2,1);
\draw (E2) node [below right] {};

\coordinate (E3) at (3,3.5);
\draw (E3) node [below right] {};

\coordinate (F1) at (-1,-0.5);
\draw (F1) node [below left] {};

\coordinate (F2) at (-2,1);
\draw (F2) node [below left] {};

\coordinate (FF2) at (-2,2);
\draw (FF2) node [right] {$e(-3)$} node {$\bullet$};

\draw [name path =L1,opacity =\XA] (E1) -- (B1);
\draw [name intersections={of=C1 and L1},  thick] (E1) -- (intersection-1) node [name = AA1]{};

\draw [name path =R1,opacity =\XA] (E1) -- (A2);
\draw [name intersections={of=C1 and R1},  thick] (E1) -- (intersection-1) node [name = AA2]{};

\draw [name path =R12,opacity =\XA] (F2) -- (B1);
\draw [name intersections={of=C1 and R12},  opacity=0] (F2) -- (intersection-1) node [name = BB2]{};

\draw [name path =L12,opacity =\XA] (F2) -- (B3);
\draw [name intersections={of=C1 and L12},  opacity=0] (F2) -- (intersection-1) node [name = BB1]{};

\draw [add =-1 and 1] (BB2) to (AA2);
\draw [add =-1 and 1] (BB1) to (AA1);

\draw [add =0.5 and 1.5, dashed] (BB2) to (AA1);

\coordinate (W) at (3,-2);
\draw (W) node {$\bullet$} node[below] {$w$};
\coordinate (P) at (0,1);
\draw (P) node {$\bullet$} node [above] {$\sigma$};
\draw [add =0.8 and 0.3] (P) to (W) node [name = V] {$\bullet$} node [ right] {$v$};
\draw [add =0 and 0.1] (V) to (FF2) node [left] {$\lw_v$};





\end{tikzpicture}

Figure: $L_{v\sigma}$ is between $L_{v\pm}$ and $\lw_v$.
\end{center}


\begin{proof}
By the definition of $\mathfrak R_E$ and the assumptions on $L_{w\sigma}$, the slope of $L_{w\sigma}$ is less than the slope of $L_{e_we^r_w}$.  As $\chy(v)>\chy(w)$ and $\ch_0(v)\geq 0$, $v$ is to the right of $w$ in the $\cccp$. Therefore, either $v$ is in $\mathfrak R_{E_w}$, or $\chy(E_v)<\chy(E_w)$. $L_{vE_w}$ is either $\lw_v$ or between $\lw_v$ and $L_{v\pm}$.

When $w$ is of Case 1 in the Definition \ref{def:lastwall}, $E_w$ is below $L_{vw\sigma}$, therefore $L_{vw\sigma}$ is between the wall $L_{vE_w}$ and $L_{v\pm}$, and the conclusion follows.

When $w$ is of Case 2 and 3 in the Definition \ref{def:lastwall}, $v$ is in $\mathfrak R_{E_w}$ of Case 3 or $E_v$ has slope less than $E_w$. In either case, $L_{vE_w(-3)}$ is either $\lw_v$ or between $\lw_v$ and $L_{v\pm}$. $E_w(-3)$ is below $L_{vw\sigma}$, therefore $L_{vw\sigma}$ is between the wall $L_{vE_w(-3)}$ and $L_{v\pm}$, hence between the wall $\lw_v$ and $L_{v\pm}$.
\end{proof}

\begin{theorem}
Let $w$ be a character in $\mathrm K(\pp)$ not inside the Le Potier cone $\Cone_{LP}$; $\sigma$ be a geometric stability condition in $\bd_{< 0}$ between $\lw_w$ and $\lrw_w$. When $\sigma$ is not on the vertical wall $L_{w\pm}$, either $\mathfrak M^s_\sigma(w)$ or $\mathfrak M^s_\sigma(-w)$ is non-empty.
\label{prop:lastwall}
\end{theorem}

The proof of the theorem is rather involved, so we want to sketch the idea here. The aim is to show the existence of new stable objects given by extensions after wall crossing. First, on any wall before the last wall, the pair of destabilizing Chern characters $w'$ and $w-w'$  are between their own last walls and the vertical walls. By induction on the discriminant, there exists stable objects with  characters $w'$ and $w-w'$. Then by  Lemma \ref{lemma:extoftwostabobjsarestab} and \ref{lemma:ext1nonvanishing}, we show that these objects have non-trivial extensions, and will extend to stable objects after each wall-crossing. However, several different cases may happen so that the idea cannot work directly. When one of the destabilizing characters is proportional to an exceptional character, Condition 1 in Lemma \ref{lemma:ext1nonvanishing} fails and we need other ways to show $\chi(w',w-w')<0$. The most complicated case is when $w'$ is of higher rank and $\lrw_{w-w'}$ is $L_{ww'}$ (Case 3.II.2 in the proof). In this case, $\mathfrak M^{ss}_\sigma\big((w'-w)[1]\big)$ may not contain any stable objects. To deal with that, we adjust $w'-w$  to another character $\tilde w$ on $L_{ww'}$ so that $\tilde w$ is of positive rank and $\lrw_{\tilde w}$ is not $L_{ww'}$. The details of the argument are as follows.

\begin{proof}
Assume the proposition does not hold, among all the characters $w$ not inside the Le Potier cone, such that $\mathfrak M^s_{\sigma'}(w)$ and $\mathfrak M^s_{\sigma'}(-w)$ are both empty for some $\sigma'$ in $\bd_{< 0}$ between $\lw_w$ and $\lrw_w$, we may choose $w$ with the minimum discriminant $\Delta$.  We may assume that $\ch_0(w)\geq 0$. When $\sigma'$ is to the left of $L_{w\pm}$, $\mathfrak M^s_{\sigma_{s,q}}(w)$  contains Gieseker-Mumford stable objects for $q \gg\frac{s^2}{2}$ and $s<\frac{\ch_1(w)}{\ch_0(w)}$. $\mssq(w)$ is not empty by Theorem \ref{thm:lp}. There is a `last wall' $L_{\sigma w}$ prior to $L^{\text{last}}_w$ such that $\mathfrak M^{s}_{\sigma+}(w)$ is non-empty, on the wall all objects in $\mathfrak M^{ss}_{\sigma}(w)$ are strictly semistable, and $M^{ss}_{\sigma-}(w)$ is empty. There are three main different cases according to the number of exceptional characters on $L_{\sigma w}$.\\

\textbf{Case 1}. There is no exceptional character on $L_{\sigma w}$. Let $F$ be a $\sigma_+$-stable object of character $w$, then $F$ is destabilized by a $\sigma$-stable object $G$ with $\tilde v(G) = w'$ on the line segment $l_{\sigma w}$. $\mathfrak M^{ss}_{\sigma}(w-w')$ is not empty since it contains $F/G$. Since there is no exceptional character on $L_{\sigma w}$, the wall $L_{\sigma w}$ is not the last wall $\lw_{w-w'}$ or $\lrw_{w-w'}$ for $w-w'$. By Corollary \ref{cor:nostabincone}, $w-w'$ is not inside Cone$_{LP}$. By Lemma \ref{lemma:lw}, $L_{\sigma w}$ is between $\lw_{w-w'}$ and $\lrw_{w-w'}$. Corollary 3.10 in \cite{BMS} implies $\Delta(w')<\Delta(w)$. By induction on $\Delta$ and the fact that $L_{\sigma (w-w')}$ is not the vertical wall, we can assume that $\mathfrak M^{s}_{\sigma}(w-w')$ is non-empty.

We check that the pair $w'$ and $w-w'$ satisfies the conditions in Lemma \ref{lemma:ext1nonvanishing}:

1.  Note that $\mathfrak M^{s}_{\sigma}(w-w')$ and $\mathfrak M^{s}_{\sigma}(w')$ are non-empty, $w'$ and $w-w'$ are not exceptional. By Lemma \ref{cor:nostabincone}, both $w'$ and $w-w'$ are not inside Cone$_{LP}$.

2. $w'+a(w-w')$ is outside the cone $\Delta_{\leq 0}$ for any $a\geq 0$. Since $L_{w\sigma}$ intersects $\bd_{<0}$, $w'-a(w-w')$ belongs to $\bd_{<0}$ for some $a>0$.  \[\Delta(w'-a(w-w'))=\Delta(w')+\Delta(w-w')-2a\Delta(w',w-w')<0\] 
implies $\Delta(w',w-w')\geq 0$.

3. When $w$ is not right orthogonal to $E_w$, by Lemma \ref{lem:l3}, the $\frac{\ch_1}{\ch_0}$-length of $L_{wE_w}\cap C_{LP}$ is greater than $3$. Hence, the $\frac{\ch_1}{\ch_0}$-length of $L_{w\sigma}\cap C_{LP}$ is greater than $3$. When $w$ is right orthogonal to $E_w$, note that $w'$ is not in the triangle area TR$_{we_w e^+_w}$, since otherwise, $E_{w'} = E_w$ and $L_{w'\sigma}$ is to the left of $L_{w'E_w}$, by Proposition \ref{prop:wisnotstabbelowexc}, $\mathfrak M^{s}_{\sigma}(w')$ is empty, there is no $\sigma$-stable object $G$ to destabilize $F$. Now since the $\frac{\ch_1}{\ch_0}$-length of $L_{wE_w}\cap C_{LP}$ is greater than $3$, the $\frac{\ch_1}{\ch_0}$-length of $L_{w\sigma} \cap C_{LP}$ is greater than $3$.

Now by Lemma \ref{lemma:ext1nonvanishing}, we have $\chi(w',w-w')<0$. For $\sigma$-stable objects $F'$ and $F''$ with characters $w'$ and $w-w'$ respectively, and $i\neq 0,1,2$,  Hom$(F',F''[i]) = 0$ since $F'$ and $F''$ are in a same heart and in addition by Serre duality. These imply Hom$(F',F''[1]) \neq 0$. Now by Lemma \ref{lemma:extoftwostabobjsarestab}, the non-trivial extension of $F'$ by $F''$ is $\sigma_-$-stable, therefore $\mathfrak M^s_{\sigma_-}(w)$ is non-empty, which contradicts to the assumption that $L_{\sigma w}$ is the last wall.\\

\textbf{Case 2:} There are more than two exceptional characters on $L_{\sigma w}$. This can only happens when $L_{w\sigma}$ is the line $\chi(E,-)=0$ for exceptional bundle $E = E_w$. In this case, $w$ is of Case 3 in Definition \ref{def:lastwall}, $L_{\sigma w}$ is $\lw_w$.\\

\textbf{Case 3:} There are one or two exceptional characters on $L_{\sigma w}$.

Similar to Case 1, we consider the character $w'$. We first prove the `lower rank wall' case, i.e. ch$_0 (w')$ $\leq$ ch$_0 (w)$. In this case, since $\phi_{\sigma+}(w)<\phi_{\sigma+}(w-w')$, the character $w-w'$ satisfies the condition in Lemma \ref{lemma:largerslopeisstable}, therefore $\mathfrak M^s_{\sigma}(w-w')$ is non-empty by induction on $\Delta$. We only need to show $\chi(w',w-w')<0$ so that by the same argument of the last paragraph in Case 1, $\mathfrak M^s_{\sigma - }(w)$ is non-empty. If $w'$ is not proportional to any exceptional character, then the proof in Case 1 works, and the pair $w'$ and $w-w'$ still satisfies the conditions in Lemma \ref{lemma:ext1nonvanishing}. If $w'$ is proportional to an exceptional character $E$,  since $L_{E (w-w')}$ is not $\lw_{(w-w')}$, $\chi(E,w-w')<0$. Therefore, $\chi(E,w-E)<0$ and $\mathfrak M^s_{\sigma-}(E,w-E)$ is non-empty. This completes the argument for the lower rank case.

Now we may assume ch$_0 (w')$ $>$ ch$_0 (w)$ and let $w''= w'-w$, then ch$_0 (w'')>0$. On the $\cccp$, $w'$ and $w''$ are in different components of $L_{w\sigma}\cap \bd_{\geq 0}$. If $\mathfrak M^s_{\sigma}(w''[1])$ is non-empty, then the argument for the lower rank case still works and implies that $\mathfrak M^s_{\sigma - }(w)$ is non-empty. On the other hand, by induction on $\Delta$, Proposition \ref{prop:compcontainsstabobj} and Proposition \ref{contracting is more than producing lemma}, the semistable locus $\mathfrak M^{ss}_{\sigma}(w''[1])$ is non-empty. So the only remaining case to consider is that $L_{w\sigma}$ is the right last wall $\lrw_{w''}$ for $w''$. 

Case 3.I: $w''$ is proportional to an exceptional character $E$: $w'' = a\tilde v(E)$. Since $E$ is to the left of $E_w$, we have $\chi(w,E)>0$, this implies
\[\chi(w',E)>\chi(w'',E) = a\chi(E,E) = a.\]
By Corollary \ref{cor:regionofEstab}, both $G$ and $E[1]$ are $\sigma$-stable in a same heart, this implies Hom$(G,E)=$ Hom$(G,(E[1])[-1])=$ $0$. Therefore,
\[\ext^1(G,E[1])=\hom(G,E[2])\geq \chi(w',E)>a.\]
By Corollary \ref{cor:extoftwostabobjsarestab}, there exists $\sigma_-$-stable object extended by $G$ and $E^{\oplus a}[1]$.

Case 3.II: $w''$ is not proportional to any exceptional character. As $\lrw_{w''} = L_{w''\sigma}$, and there are at most two exceptional characters on $L_{w''\sigma}$ by assumption, $w''$ is not of Case $3$ in Definition \ref{def:lastwall}, either $E^{\text{(rhs)}}_{w''}$ or $E^{\text{(rhs)}}_{w''}(3)$ is on the line segment $l_{ww''}$.

Case 3.II.1: $w''$ is of (right side) Case 2 in Definition \ref{def:lastwall}  and $\tilde v(E^{\text{(rhs)}}_{w''}(3))$ is on $l_{ww''}$.  The character $w$ can be written as
\[a\tilde v(E^{\text{(rhs)}}_{w''}(3))-bw''\]
for some positive numbers $a$ and $b$. Since $\chi(E^{\text{(rhs)}}_{w''}(3),w'')<0$, we have
\[\chi(E^{\text{(rhs)}}_{w''}(3),w)>0.\]
This implies $\chy\left(E^{\text{(rhs)}}_{w''}(3)\right)\leq \chy (E_w)$. As $E^{\text{(rhs)}}_{w''}(3)$ is above $\lw_w$, it must be $E_w$. On the other hand, as $\chi(E_w,w) = \chi(E^{\text{(rhs)}}_{w''}(3),w)>0$, $w$ is of Case $1$ in Definition \ref{def:lastwall}. The wall $L_{wE_ww''}$ is just the last wall $\lw_w$ of $w$.

Case 3.II.2: $w''$  is of (right side) Case 1 in Definition \ref{def:lastwall}  and $E=\tilde v(E^{\text{(rhs)}}_{w''})$ is on $l_{ww''}$. By Definition \ref{def:lastwall}, $\chi(w'',E)>0$. Consider the character $\tilde w := w'' -\chi(w'',E)\tilde v(E)$, we have $\chi(\tilde w,E) = 0$, therefore $\tilde w$ is on the line $L_{e^r(e(3)^l)}$.

The character $w$ must be above $L_{EE(3)}$, otherwise $L_{wE}$ is the last wall $\lw_w$. The intersection of $L_{wE}\cap L_{e^r(e(3)^l)}$ is outside the cone $\bd_{<0}$ and on the different side of $w$ in the $\cccp$. As $w''$, $E$ and $\tilde w$ are on the same component of $L_{wE}\cap\bd_{\geq 0}$, ch$_0(\tilde w)$ is greater than $0$. The character $w+\tilde w = w'-\chi(w'',E)\tilde v(E)$ is on the line segment $l_{ww'}$.


\begin{center}
\begin{tikzpicture}[domain=1:5]

\tikzset{%
    add/.style args={#1 and #2}{
        to path={%
 ($(\tikztostart)!-#1!(\tikztotarget)$)--($(\tikztotarget)!-#2!(\tikztostart)$)%
  \tikztonodes},add/.default={.2 and .2}}
}

\newcommand\XA{0}
\newcommand\XP{2.3}

\draw [name path =C0](-2.5,3.125) parabola bend (0,0) (3,4.5)
 node[right] {$\bd_0$};

\draw [name path = C1, opacity=0](-2.5,2.625) parabola bend (0,-0.5) (2,1.5) node[right] {$\bd_{\frac{1}{2}}$};


\coordinate (B3) at (-3,4.5);
\coordinate (B2) at (-2,2);

\coordinate (B1) at (-1,0.5);
\coordinate (A0) at (0,0);
\coordinate (A1) at (1,0.5);
\coordinate (A2) at (2,2);
\coordinate (A3) at (3,4.5);

\coordinate (E0) at (0,-1);
\draw (E0) node [below right] {};

\coordinate (E1) at (1,-0.5);
\draw (E1) node [below] {};

\coordinate (EE1) at (1,0.5);
?\draw (EE1) node [above] {$e(3)$} node {$\bullet$};

\coordinate (E2) at (2,1);
\draw (E2) node [below right] {};

\coordinate (E3) at (3,3.5);
\draw (E3) node [below right] {};

\coordinate (F1) at (-1,-0.5);
\draw (F1) node [below left] {};

\coordinate (F2) at (-2,1);
\draw (F2) node [below left] {};

\coordinate (FF2) at (-2,2);
\draw (FF2) node [above right] {$e$} node {$\bullet$};

\draw [name path =L1,opacity =\XA] (E1) -- (B1);
\draw [name intersections={of=C1 and L1},  opacity =0] (E1) -- (intersection-1) node [name = AA1, opacity =1] {$\bullet$} node [below, opacity =1] {$e(3)^l$};

\draw [name path =R1,opacity =\XA] (E1) -- (A2);
\draw [name intersections={of=C1 and R1}, opacity =0] (E1) -- (intersection-1) node [name = AA2] {$\bullet$} node [below right] {};

\draw [name path =R12,opacity =\XA] (F2) -- (B1);
\draw [name intersections={of=C1 and R12},  thick] (F2) -- (intersection-1) node [name = BB2] {$\bullet$} node [below left] {$e^r$};

\draw [name path =L12,opacity =\XA] (F2) -- (B3);
\draw [name intersections={of=C1 and L12},  thick] (F2) -- (intersection-1) node [name = BB1] {} node [left] {};

\draw [add= 0 and 1, dashed] (AA1) to (BB2) node [above] {$L_{e^re(3)^l}$};
\draw [add= 0 and 0.7, dashed] (AA1) to (BB2) node [name = WT] {$\bullet$} node [below] {$\tilde w$};

\draw [add= 0.5 and 6] (WT) to (FF2) node [name = W] {$\bullet$} node [below] {$w$};
\draw [add= 0.5 and 5] (WT) to (FF2) node  {$\bullet$} node [below] {$w'$};
\draw [add= 0.5 and -0.5] (WT) to (FF2) node {$\bullet$} node [below] {$w''$};


\draw[->,opacity =0.3] (-3,0) -- (4,0) node[above right] {$\frac{\ch_1}{\ch_0}$};

\draw[->,opacity=0.3] (0,-2)-- (0,0) node [above right] {O} --  (0,3.5) node[right] {$\frac{\ch_2}{\ch_0}$};

\end{tikzpicture}

Figure: Definition of $\tilde w$.
\end{center}


When $w'$ is not proportional to any exceptional character, $w+\tilde w$ is outside  Cone$_{LP}$ and on the same component of $L_{wE}\cap\bd_{\geq 0}$ as $w$. Since the line segment $l_{\tilde w(w+\tilde w)}$ intersects $\bd_{<0}$, $\Delta(\tilde w,w+\tilde w)<0$.  This implies  $\Delta (\tilde w)<\Delta (w)$ and $\Delta (w+\tilde w)$ $<$ $\Delta (w)$. As $\mathfrak M^s_\sigma(w')$ is non-empty, by Lemma \ref{lemma:largerslopeisstable} and induction on $\Delta$,  $\mathfrak M^s_\sigma(w+\tilde w)$ is non-empty. As $\chi(\tilde w,E) = 0$, $L_{\tilde w E}$ is not the last wall $\lrw_{\tilde w}$ for $\tilde w$. By induction on $\Delta$, $\mathfrak M^s_\sigma(-\tilde w)$ is non-empty. The character pair $w+\tilde w$ and $-\tilde w$ satisfy the conditions in Lemma \ref{lemma:ext1nonvanishing}, hence $\chi(w+\tilde w,\tilde w[1])<0$. $\mathfrak M^s_{\sigma-}(w)$ is non-empty by Lemma \ref{lemma:extoftwostabobjsarestab}.

When $w'$ is $\tilde v(E')$ for an exceptional bundle $E'$, as $E'$ is to the right of $E(3)$, we have $\chi(E,E')>0$. Hence,
\[\chi(w+\tilde w,E') = \chi\big(E'-\chi(w'',E)\tilde v(E),E'\big) = 1-\chi(w'',E)\chi(E,E')\leq 0.\]
This implies the characters $w+\tilde w$ and $\tilde v(E')$ are on two different sides of $L_{e'^+e'^l}$. Therefore, $w+\tilde w$ is not inside Cone$_{LP}$. The rest of the argument is the same as the case when $w'$ is not proportional to exceptional character.\\

Up to now, we finish the argument for the case that $\sigma$ is on the left side of $L_{w\pm}$. When $\sigma$ is on the right side of $L_{w\pm}$, the statement follows from the symmetric property (ch$_0(w)> 0$):
\[\mss(w)\simeq \mathfrak M^s_{\sigma'}(w'[1]), \;\; F\mapsto \mathcal {RH}om(F,\mathcal O)[1],\]
where $\sigma'$ is with parameter $(-s_\sigma,q_\sigma)$ and $w'=\big( \ch_0(w),-\ch_1(w),\ch_2(w)\big)$.
\end{proof}

\subsection{The criteria for actual walls}\label{sec3.2}
In this section we give a numerical criteria for actual walls of a given Chern character. In the $\cccp$, the actual wall for $w$ is the potential wall $L_{w\sigma}$ where new stable objects are produced on both sides and curves are contracted on at least one side. When $\sigma$ is to the left of the vertical wall $L_{w\pm}$, one can always choose a destabilizing factor $v$ with positive rank and smaller slope. As $\Delta(v)$ is less than $\Delta(w)$, there are finitely many candidates $v$. By checking the positions of $v$ and $v-w$ on the $\cccp$, which are purely numerical data, Theorem \ref{thm: actual wall} determines whether $L_{w\sigma}$ is an actual wall induced by this pair. The idea of the proof is very similar to that of the last wall, we first show there are stable objects on the wall with characters $v$ and $w-v$ by Theorem \ref{prop:lastwall}. We then argue that the ext$^1$ of the stable objects is greater than $0$ by Lemma \ref{lemma:ext1nonvanishing}, and finally claim that curves must be contracted from the $\sigma_+$-side wall-crossing.

To state the criteria for actual walls, we first need to introduce the following definition.

\begin{defn}
For a Chern character $w$ with $\ch_0(w)\geq0$ and an exceptional character $e$, we define the \emph{triangle $\mathrm{TR}_{we}$} to be the triangle region bounded by lines $L_{we}$, $L_{e^le^+}$ and $L_{e^+e^r}$ in the $\cccp$.
\label{def:trew}
\end{defn}


\begin{center}
\begin{tikzpicture}[domain=1:5]

\tikzset{%
    add/.style args={#1 and #2}{
        to path={%
 ($(\tikztostart)!-#1!(\tikztotarget)$)--($(\tikztotarget)!-#2!(\tikztostart)$)%
  \tikztonodes},add/.default={.2 and .2}}
}

\newcommand\XA{0}
\newcommand\XP{2.3}

\draw [name path =C0, opacity=0.1](-2,2) parabola bend (0,0) (2,2)
 node[right, opacity =0.5] {$\bd_0$};

\draw [name path = C1, opacity=0.5](-2,1.5) parabola bend (0,-0.5) (2,1.5)
 node[right] {$\bd_{\frac{1}{2}}$};


\coordinate (B3) at (-3,4.5);
\coordinate (B2) at (-2,2);

\coordinate (B1) at (-1,0.5);
\coordinate (A0) at (0,0);
\coordinate (A1) at (1,0.5);
\coordinate (A2) at (2,2);
\coordinate (A3) at (3,4.5);

\coordinate (E0) at (0,-1);
\draw (E0) node [below right] {};

\coordinate (E1) at (1,-0.5);
\draw (E1) node [below right] {};

\coordinate (E2) at (2,1);
\draw (E2) node [below right] {};

\coordinate (E3) at (3,3.5);
\draw (E3) node [below right] {};

\coordinate (F1) at (-1,-0.5);
\draw (F1) node [below left] {};

\coordinate (F2) at (-2,1);
\draw (F2) node [below left] {};

\coordinate (F3) at (-3,3.5);
\draw (F3) node [below left] {};

\draw [name path =L0,opacity =\XA] (E0) -- (B2);
\draw [name intersections={of=C1 and L0},  thick] (E0) -- (intersection-1) node [name=EE0]{};
\draw [add= 0 and 1] (EE0) to (E0);

\draw [name path =L1,opacity =\XA] (E1) -- (B1);
\draw [name intersections={of=C1 and L1},  thick] (E1) -- (intersection-1);

\draw [name path =R0,opacity =\XA] (E0) -- (A1);
\draw [name intersections={of=C1 and R0},  thick] (E0) -- (intersection-1);

\draw [name path =R1,opacity =\XA] (E1) -- (A2);
\draw [name intersections={of=C1 and R1},  thick] (E1) -- (intersection-1)node [name=EE1]{};
\draw [add= 0 and 1] (EE1) to (E1);

\coordinate (S3) at (-2.5,2.5);
\draw (S3) node [below left] {};

\coordinate (S2) at (-1.5,0.5);

\coordinate (S1) at (-.5,-0.5);

\coordinate (T1) at (.5,-0.5);

\coordinate (T2) at (1.5,0.5);

\coordinate (T3) at (2.5,2.5);
\draw (T3) node [below right] {};

\draw [name path =RT1,opacity =\XA] (T1) -- (B2);
\draw [name intersections={of=C1 and RT1},  thick] (T1) -- (intersection-1);
\draw [name path =LT1,opacity =\XA] (T1) -- (A3);
\draw [name intersections={of=C1 and LT1},  thick] (T1) -- (intersection-1)node [name=EE2]{};
\draw [add= 0 and 1] (EE2) to (T1);

\coordinate (W) at (0.4,-3);
\draw (W) node {$\bullet$} node [below] {$w$};
\coordinate (O) at (0,0);
\draw (O) node {$\bullet$} node [above] {$E$};
\draw [name path=L1] (W) -- (O);
\coordinate (O1) at (0.5,-0.25);
\draw (O1) node {$\bullet$} node [above] {$E'$};
\draw [name path=L1] (W) -- (O1);
\coordinate (O2) at (1,0.5);
\draw (O2) node {$\bullet$} node [above] {$E''$};
\draw [name path=L1] (W) -- (O2);


\draw[->,opacity =0.3] (-2,0) -- (2,0) node[above right] {$\frac{\ch_1}{\ch_0}$};

\draw[->,opacity=0.3] (0,-3.5)-- (0,0) node [above right] {O} --  (0,2.5) node[right] {$\frac{\ch_2}{\ch_0}$};

\coordinate (H1) at (0.1,-1.1);
\coordinate (H2) at (-1,-2);
\draw[->] (H2) node [left] {TR$_{wE}$} to (H1);

\coordinate (H1) at (0.9,-0.6);
\coordinate (H2) at (1.5,-1.3);
\draw[->] (H2) node [below] {TR$_{wE''}$} to (H1);

\end{tikzpicture}

Figure: Definition of TR$_{wE}$.
\end{center}


Now we can state the main theorem on actual walls. The regions TR$_{wE}$ will be used to detect the non-emptiness of moduli spaces of stable object of any `sub-character', as will be explained in the proof of the theorem.

\begin{theorem}
Let $w\in\mathrm K(\pp)$ be a Chern character outside the Le Potier cone with $\ch_0(w)\geq 0$. For any stability condition $\sigma_{s,q}$ in $\bd_{< 0}$ between the wall $\lw_w$ and the vertical ray $L_{w+}$, the wall $L_{\sigma w}$ is an actual wall for $w$ if and only if there exists a Chern character $v\in \mathrm K(\pp)$ on the line segment $l_{\sigma w}$ such that:
\begin{itemize}
\item $\ch_0(v)>0$ and $\frac{\ch_1(v)}{\ch_0(v)}<\frac{\ch_1(w)}{\ch_0(w)}$;
\item the characters $v$ and $w-v$ are either exceptional or not inside the Le Potier cone and both of them are not in $\mathrm{TR}_{wE}$ for any exceptional bundle $E$.
\end{itemize} \label{thm: actual wall}
\end{theorem}

\begin{rem}
1. For given characters $w$ and $v$, one only needs to check whether $v$ or $w-v$ are in $\mathrm{TR}_{wE}$ for at most two particular exceptional bundles. Suppose the intersection points $L_{\sigma w}\cap \bd_{\frac{1}{2}}$ fall between the segment between $e_i^r$ and $e_i^l$ for some exceptional character $e_1$ and $e_2$, then one only needs to check the triangles $\mathrm{TR}_{wE_i}$.

2. By the term `in $\mathrm{TR}_{wE}$', strictly speaking, we mean that `in the closure of  $\mathrm{TR}_{wE}$ but not on the line $L_{e^+e^l}$ when $E$ is not to the right of  $E_w$ or not on the line $L_{e^+e^r}$ when $E$ is to the left of $E_w$'.
\label{rem:actwall}
\end{rem}

\begin{proof}
The first step is to show that when $v$ (or $w-v$) is not inside the Le Potier cone, the condition  `$v$ (or $w-v$) is not in TR$_{wE}$ for any exceptional $E$' is equivalent to `$\mathfrak M^s_\sigma(v)$ (or $\mathfrak M^s_\sigma(w-v)$) is non-empty for $\sigma$ in $\bd_{< 0}$ on the line $L_{wv}$'.

The `$\Leftarrow$' direction is easy to check: Suppose $v$ is in $TR_{wE}$ for some $E$, then $E$ must be $E_w$ or to the right of $E_w$. This implies $l_{\sigma w}$ intersects $l_{e^+e^r}$. The character $v$  is in $\mathfrak R_E$ and of Case 1 in Definition \ref{def:lastwall}. By Proposition \ref{prop:wisnotstabbelowexc},  $\mathfrak M^s_\sigma(v)$ is empty. The $w-v$ part is proved in a similar way.

For the `$\Rightarrow$' direction, let $f_1$ and $f_2$ be the intersection points of the line $L_{vw}$ and the parabola $\bd_{\frac{1}{2}}$. Suppose that $f_1$ has larger $\chy$, and $f_1$ lies on the segment between $e^l$ and $e^r$ for some exceptional bundle $E$ by Theorem 4.1 in \cite{CHW}. Since $v$ is below $L_{ee^r}$, $E_v$ is either $E$ or to the left of $E$.

Three different cases may happen:
\begin{enumerate}
\item If $v$ is above $L_{e^+e^l}$, then $E_v=E$. Since $v$ is not in TR$_{wE}$, $v$ is above $L_{Ew}$. This implies $E$ is below $L_{vw}$, and $L_{v\sigma}$ is between $\lw_v$ and $L_{v\pm}$.
 \item If $E\neq E_w$ and $v$ is not above $L_{e^+e^l}$, then $w$ is below the line $L_{e(-3)e^l}$ and $E(-3)$ is below $L_{wv}$. Hence, $L_{v\sigma}$ is between $L_{vE(-3)}$ and $L_{v\pm}$;
 \item If $E= E_w$ and $v$ is not above the line $L_{e^le^+}$, then by Remark \ref{rem:actwall}, $v$ is above $\lw_w=L_{E(-3)w}$. $w$ is below $\lw_v=L_{E(-3)v}$, therefore, $L_{wv\sigma}$ is between $\lw_v$ and $L_{v\pm}$.
\end{enumerate}
In either case, $L_{v\sigma}$ is between $\lw_v$ and $L_{v\pm}$. It follows from Theorem \ref{prop:lastwall} that $\mathfrak M_\sigma^{s}(v)$ is non-empty for any $\sigma\in L_{v\sigma}\cap \bd_{\leq 0}$.

Write $u$ for $w-v$. When $\ch_0(u)\geq 0$, by Lemma \ref{lemma:largerslopeisstable} and the similar argument as for $v$, $\mathfrak M_\sigma^{ss}(u)$ is non-empty for any $\sigma\in L_{vw}\cap \bd_{\leq 0}$. If $\ch_0(u) < 0$, let $E$ be the exceptional bundle such that $f_2$ lies on the segment of $\bd_{\frac{1}{2}}$ between $x=e^l$ and $x=e^r$. By a similar argument as for $v$, when $u$ is above $L_{e^+e^l}$, $L_{uw}$ is between  $\lrw_u=L_{uE}$ and $L_{i\pm}$. When $u$ is not above $L_{e^+e^l}$, $L_{uw}$ is between $L_{uE(3)}$ and $L_{u\pm}$. By Theorem \ref{prop:lastwall}, $\mathfrak M_\sigma^{ss}(u)$ is non-empty for any $\sigma\in L_{vw}\cap \bd_{\leq 0}$ .\\

The second step is to prove the `only if' direction in the statement, which follows from the claim in the first step. If $L_{\sigma w}$ is an actual wall, an object $F$ with character $w$ is destabilized by a stable object with character $v$ such that $\ch_0(v)>0$ and $\chy (v)<\chy (w)$. By Corollary \ref{cor:nostabincone}, $v$ is exceptional or outside Cone$_{LP}$. By the previous discussion, since $\mss(v)$ is not empty, $v$ is not in any TR$_{wE}$.  For the character $w-v$, since $\mathfrak M^{ss}_{\sigma}(w-v)$ is non-empty, we only need to consider the case when all semistable objects are strictly semistable. Since $\mathfrak M^{ss}_{\sigma}(w-v)$ is non-empty, by  Proposition \ref{prop:compcontainsstabobj},
Proposition \ref{contracting is more than producing lemma} and Theorem \ref{prop:lastwall}, $\mathfrak M^{s}_{\sigma}(w-v)=\phi$ if and only if $L_{w\sigma}$ is $\lrw_{w-v}$ or $\lw_{w-v}$. The second case is not possible by Lemma \ref{lemma:largerslopeisstable}. We may assume $\ch_0(v)>\ch_0(w)$. $v-w$ is of Case 1 in Definition \ref{def:lastwall} since otherwise $v-w$ is not in TR$_{E^{\mathrm{(rhs)}}_{v-w}w}$ and $\mathfrak M^{s}_{\sigma}(w-v)$ is not empty. Write $E^{(\mathrm{rhs})}_{v-w}$ as $E$, we let
\begin{equation*}
v':= v-\chi (v-w,E)\cdot e,
\end{equation*}
\text{then } $w-v'= w-v+\chi(v-w,E)\cdot e$. Since $\chi(w-v',E)=0$, $w-v'$ is the intersection point of $L_{w\sigma}$ and $L_{e^+e^r}$ and is not in TR$_{wE}$. By the same argument as in Case 3.II.2 of the proof of Theorem \ref{prop:lastwall}, $\ch_0(v')>0$, $\chy(v')<\chy(w)$ and $\mathfrak M^s_{\sigma}(v')$ is non-empty. Therefore, the pair $v'$ and $w-v'$ satisfies the requirements in the statement.\\

The last step is to prove the  `if' direction in the statement. Similar to the proof for the last wall, objects with characters $v$ and $u=w-v$ do not always have non-trivial extensions. We need to build Chern characters $u'$ and $v'$ on the line $L_{vw}$, such that:\\

1. $w=v'+u'$;

2. $\mathfrak M_\sigma^{s}(v')$ and $\mathfrak M_\sigma^{s}(u')$ are non-empty for $\sigma\in L_{vw}\cap \bd_{<0}$.

3. $\mathfrak M_{\sigma_+}^s(v',u')\rightarrow \mathfrak M_\sigma^{ss}(w)$ contracts curves.\\

Four cases may happen for $u$ and $v$:

i) $v$ and $u$ are not proportional to any exceptional characters, in other words, they are outside Cone$_{LP}$. Since they are outside the triangles TR$_{wE}$, $\mathfrak M_\sigma^{s}(v)$ and $\mathfrak M_\sigma^{s}(u)$ are non-empty. The characters $v$ and $u$ satisfy the conditions in Lemma \ref{lemma:ext1nonvanishing} due to the same argument as Case 1 in Theorem \ref{prop:lastwall}. This implies $\chi(v,u)< 0$. By the first property of Lemma \ref{lemma:chiformrho} and the same computation as in Proposition \ref{contracting is more than producing lemma}, $\chi(u,v)-\chi(v,u) \leq -3$. Therefore, by Lemma \ref{lemma: ext2 vanishing for stable factors}, for any $\sigma$-stable objects $F$ and $G$ in $\mathfrak M_\sigma^{s}(v)$ and $\mathfrak M_\sigma^{s}(u)$, ext$^1(G,F)\geq 3$. By Lemma \ref{lemma:extoftwostabobjsarestab}, $\mathfrak M_{\sigma_+}^s(F,G)\rightarrow \mathfrak M_\sigma^{ss}(w)$ contracts curves.\\

ii) $v$ is proportional to the character $e$ of some exceptional bundle $E$, but $u$ is not proportional to any exceptional characters. Write $v=ne$ for some integer $n\geq 1$. When $\ch_0(w)\geq \ch_0(e)$, the character $u'=u+(n-1)e$ is to the right of $w$. Therefore $u'$ is not in TR$_{wE}$ and $\mathfrak M_\sigma^{s}(u')$ is non-empty. $v'=w-u'=e$ and $\mss (e)$ is non-empty.

In the case $\ch_0(e) > \ch_0(w)$, we have:
\[\frac{\ch_1(u)}{\ch_0(u)}=\frac{\ch_1(w-ne)}{\ch_0(w-ne)} > \frac{\ch_1(w-e)}{\ch_0(w-e)}=\frac{\ch_1(u')}{\ch_0(u')}.\]
As $u$ is outside Cone$_{LP}$ and on the different component of $L_{w\sigma}\cap\bd_{>0}$ than that of $w$,  $u'$ is also outside Cone$_{LP}$ and not in any TR$_{wE}$. We may still let $v'$ be $e$.

As $L_{ew}$ is not the last wall $\lw_w$, $\chi(e,w)\leq 0$. We have $\chi(e,u')\leq -1$. By the same argument as in i), we have ext$^1(G,E)\geq 3$ for any object $G$ in $\mathfrak M_\sigma^{s}(u')$. Therefore, $\mathfrak M_{\sigma_+}^{s}(E,G)\rightarrow \mathfrak M_\sigma^{ss}(w)$ contracts curves.\\

iii) $u$ is proportional to the character $e$ of some exceptional bundle $E$, but $v$ is not proportional to any exceptional characters. As $u$ is not on the line segment $l_{\sigma w}$, it has negative $\ch_0$. Suppose $u=-ne$ and we may let $v'= w+e$ and $u'=-e$ in the similar way as in ii). By the same argument on the slope of $v$ and $v'$, $v'$ is outside any triangle are TR$_{wE}$. As $L_{we}$ is not the last wall $\lw_w$, we have $\chi(w,e)\geq 0$. Therefore $\chi(v',u')\leq -1$. By the same argument as in i), we have ext$^1(E,F)\geq 3$ for any object $F$ in $\mathfrak M_\sigma^{s}(v')$. Therefore, $\mathfrak M_{\sigma_+}^{s}(E,G)\rightarrow \mathfrak M_\sigma^{ss}(w)$ contracts curves.\\

iv) The Chern characters $u$ and $v$ are proportional to the characters $e_1$ and $e_2$ of exceptional bundles $E_1$ and $E_2$ respectively. Write $u=-n_1e_1$ and $v=n_2e_2$. Since $L_{vw}$ is not $\lw_w$, $\chi(e_2,w) \leq 0$. Therefore,
\[n_2 \leq n_1\chi(e_2, e_1) = n_1 \ext^2(E_2, E_1)=n_1 \hom(E_1, E_2(-3)) < n_1 \hom(E_1, E_2).\]
As a consequence, we see that
\[\dim \mathfrak M^s_{\sigma-}(E_1^{\oplus n_1}[1], E_2^{\oplus n_2}) = \dim \mathrm{Kr}_{\hom(E_1, E_2)}(n_1, n_2) = n_1n_2\hom(E_1, E_2)-n_1^2-n_2^2+1 \geq 2.\]
Here Kr$_{\hom(E_1, E_2)}(n_1, n_2)$ is the Kronecker model, in other words, it is the representations space of Hom$(\mathbb C^{n_2},\mathbb C^{n_1})^{\oplus \hom(E_1,E_2)}$ quotient by the natural group action of $GL(n_1)\times GL(n_2)/\mathbb C^*$.

In all cases, $L_{vw}$ is an actual wall for $w$.
\end{proof}

Now the following corollary follows easily:

\begin{cor}[Lower rank walls]
Let $w$ be a character with $\ch_0(w)\geq 0$. For any character $v$ with $0< \ch_0(v)\leq \ch_0(w)$, suppose that $v$ is between the wall $L_w^{last}$ and the vertical ray $L_{w+}$, outside the Le Potier cone $\mathrm{Cone}_{LP}$, and not in $\mathrm{TR}_{wE}$ for any exceptional bundle $E$. Then $L_{vw}$ is an actual wall.
 \label{cor: actual wall}
\end{cor}

\section{Applications: the ample cone and the movable cone}\label{sec4}

In this section, we work out several applications of our criteria on actual walls. We compute the boundary of the movable cone in Section \ref{sec4.1} and the boundary of the nef cone in Section \ref{sec4.2}. In Section \ref{sec4.4}, we compute all the actual walls of moduli space of stable sheaves of character $(4,0,-15)$, as an example of how to apply the machinery in this paper to a concrete situation.

\subsection{Movable cone}\label{sec4.1}
Let $w\in\mathrm K(\pp)$ be a character with $\ch_0(w)\geq 0$ not inside Cone$_{LP}$. It has been revealed in \cite{CHW} that when $\sigma$ is in the `last' chamber above $\lw_w$, the birational model $\mss(w)$ is of Picard number $1$ if and only if $w$ is right orthogonal to  $E_w$. In other words, the movable cone boundary on the primary side is not the same as the effective cone boundary if and only if $\chi(E_w,w)=0$. In this section, we determine the boundary of the movable cone in this case.

Let $(E_\alpha, E_\beta, E_\gamma)$ be a triple of exceptional bundles corresponding to dyadic numbers $\frac{p-1}{2^n}$, $\frac{p+1}{2^n}$, $\frac{p}{2^n}$, respectively. The following property is well-known, the reader is referred to \cite{GorRu}.

\begin{lemma}
For the triple $(E_\alpha, E_\beta, E_\gamma)$, we have
\[\chi(E_\alpha, E_\gamma)=\hom(E_\alpha, E_\gamma)=3\ch_0(E_\beta),\]
\[\chi(E_\gamma, E_\beta)=\hom(E_\gamma, E_\beta)=3\ch_0(E_\alpha),\]
\[\hom(E_\alpha, E_\gamma)\cdot \hom(E_\gamma, E_\beta)-\hom(E_\alpha, E_\beta)=3\ch_0(E_\gamma).\]
\[\hom(E_\beta(-3),E_\alpha)\hom(E_\alpha,E_\gamma)-\hom(E_\beta(-3),E_\gamma)=3\ch_0(E_\alpha)\]
For any exceptional $E_{(\frac{t}{2^q})}$ such that $\frac{p-1}{2^n}<\frac{t}{2^q}<\frac{p}{2^n}$, we have $\ch_0(E_{(\frac{t}{2^q})})<\ch_0(E_\gamma)$.
\label{lemma: rk of exc}
\end{lemma}

The following observation is from the proof for Proposition \ref{prop:lastwall}. It will be used in the next theorem.
\begin{lemma}
Let $w\in \mathrm K(\pp)$ be a Chern character with outside $\mathrm{Cone}_{LP}$. Let $e$ be an exceptional character such that in the $\cccp$, $w$ is in the area between two parallel lines $L_{ee(3)}$ and $L_{e^re^+}$. Then we have $|\ch_0(w)| > \ch_0(E)$.
\label{lem:highrkinstrip}
\end{lemma}

Let $w$ be a primitive character outside Cone$_{LP}$ with $\ch_0(w)\geq 0$. Assume that $w$ is right orthogonal to the exceptional bundle $E_w = E_\gamma$, and consider the triple $(E_\alpha, E_\beta, E_\gamma)$ corresponding to dyadic numbers $\frac{p-1}{2^n}$, $\frac{p+1}{2^n}$, $\frac{p}{2^n}$. The character $w$ can be uniquely written as $n_2e_\alpha-n_1e_{\beta-3}$ for positive numbers $n_1$, $n_2$.
\begin{theorem}
Adopt the notations as above, we may define a character $P$ based on two different cases:
\begin{itemize}
\item[i)] $P:=e_\gamma - (3\ch_0(E_\beta)-n_2)e_\alpha$, if $1 \leq n_2 < 3 \ch_0(E_\beta)$;
\item[ii)] $P:=e_\gamma$, if $n_2 \geq 3 \ch_0(E_\beta)$.
\end{itemize}
On the wall $L_{Pw}$, a divisor of $\mathfrak M^s_{L_{Pw_+}}(w)$ is contracted.
\label{thm:movcone}
\end{theorem}

\begin{proof}
In order to apply Theorem \ref{thm: actual wall}, we first show that in case i), $P$ is not inside Cone$_{LP}$, and is outside TR$_{wE}$ for any exceptional bundle $E$. Since $\chi(e_\beta,P)=0$, $P$ is on the line $L_{e_\beta^+e_\beta^l}$. By Lemma \ref{lemma: rk of exc},
\begin{align*}
\chi(P,P)= & 1+ (3\ch_0(E_\beta)-n_2)^2-(3\ch_0(E_\beta)-n_2)\chi(E_\alpha,E_\gamma)\\
= &1-n_2(3\ch_0(E_\beta)-n_2) \leq 0
\end{align*}
Since $P$ is on the line $L_{e_\beta^+e_\beta^l}$, it is  not inside Cone$_{LP}$ and is outside TR$_{wE}$ for any exceptional bundle $E$.\\

We next show that $w-P$ is outside TR$_{wE}$ for any exceptional bundle $E$. By Lemma \ref{lemma:largerslopeisstable}, we only need to treat the case when $\ch_0(w-P)<0$. We are going to  prove that for any exceptional bundle $E$ to the left of $E_\gamma(-3)$, if $E$ is above $L_{Pw}$, then $\chi(P-w, E)\leq 0$. This will imply that $w-P$ is not in TR$_{wE}$.

In case \emph{i}),
\[P-w=n_1e_\beta(-3)-(3\ch_0(E_\beta)\cdot e_\alpha-e_\gamma).\]
By Lemma \ref{lemma: rk of exc} and Serre duality,
\[\chi(3\ch_0(E_\beta)e_\alpha-e_\gamma, e_\alpha(-3))=3\ch_0(E_\beta)-\hom(e_\alpha,e_\gamma)=0.\]
Since $\chi(e_\beta(-3),e_\alpha(-3))$ is also $0$, the point $P-w$ is on the line $L_{e_\alpha(-3)^+e_\alpha(-3)^r}$.
$\ch_0(w)\geq 0$ implies  $n_2\ch_0(e_\alpha)\geq n_1\ch_0(e_\beta)$. Hence
\[n_1 \leq \frac{\ch_0(e_\alpha)}{\ch_0(e_\beta)}\cdot n_2 < 3\ch_0(e_\alpha).\]
By the fourth equation in Lemma \ref{lemma: rk of exc},
\begin{align*}
\chi(P-w, P-w) &= n_1^2 +1 -n_1\chi(e_\beta(-3), 3\ch_0(E_\beta) \cdot e_\alpha - e_\gamma) \\
&= n_1^2 +1 -3\ch_0(e_\alpha) \cdot n_1 <0.
\end{align*}
Combining with the result that $P-w$ is on the line $L_{e_\alpha(-3)^+e_\alpha(-3)^r}$, we know that $P-w$ is not above the curve C$_{LP}$, and for any exceptional bundle $E$ to the left of $e_\alpha(-3)$,
\[\chi(P-w, E) \leq 0.\]

On the other hand, the line segment $l_{(P-w)P}$ is above the line  $L_{(P-w)e_\gamma}$, hence above the line segment $l_{e_\alpha(-3)^re_\gamma}$. Since $l_{e_\alpha(-3)^re_\gamma}$ is above any exceptional characters between the vertical rays $L_{e_\alpha(-3)\pm}$ and $L_{e_\gamma\pm}$, the character $P-w$ is not in the triangle TR$_{wE}$ for any such exceptional bundle $E$.\\

In case \emph{ii}), the character $P-w$ can be rewritten as follows:
\begin{align*}
& P-w = e_\gamma-w\\ =\;& n_1e_{\beta-3} - (n_2-3\ch_0(E_\beta))e_\alpha - (3\ch_0(E_\beta)\cdot e_\alpha-e_\gamma) \\
=\;&  (n_2-3\ch_0(E_\beta))\left(\frac{\ch_0(E_\alpha)}{\ch_0(E_\beta)}e_{\beta-3} - e_\alpha\right) \\
&+ \left(n_1 - n_2 \frac{\ch_0(E_\alpha)}{\ch_0(E_\beta)} + 3\ch_0(E_\alpha)\right)e_{\beta-3} - (3\ch_0(E_\beta)e_\alpha -e_\gamma).
\end{align*}

Note that the character $\frac{\ch_0(E_\alpha)}{\ch_0(E_\beta)}e_{\beta-3} - e_\alpha$ in the first term is proportional to $e_\gamma(-3)-e_\gamma$ by a positive scalar, and the coefficient $n_2-3\ch_0(E_\beta)$ is non-negative. We denote the rest term as 
\[v':= \left(n_1 - n_2 \frac{\ch_0(E_\alpha)}{\ch_0(E_\beta)} + 3\ch_0(E_\alpha)\right)e_{\beta-3} - (3\ch_0(E_\beta)e_\alpha -e_\gamma).\]
By Lemma \ref{lemma: rk of exc} and the assumption, $\ch_0(v')=\ch_0(P-w)>0$. In particular, $n_1 - n_2 \frac{\ch_0(E_\alpha)}{\ch_0(E_\beta)} + 3\ch_0(E_\alpha)>0$.
Since $\ch_0(w)>0$, we have the inequality:
\[n_1 - n_2 \frac{\ch_0(E_\alpha)}{\ch_0(E_\beta)} + 3\ch_0(E_\alpha) < 3\ch_0(E_\alpha).\]

Due to a similar computation as in case \emph{i}),
\begin{align*}
\chi(v',v') < & \left(n_1 - n_2 \frac{\ch_0(E_\alpha)}{\ch_0(E_\beta)} + 3\ch_0(E_\alpha)\right)^2+1- \\
 & \left(n_1 - n_2 \frac{\ch_0(E_\alpha)}{\ch_0(E_\beta)} + 3\ch_0(E_\alpha)\right)\chi(e_{\beta-3},3\ch_0(E_\beta)e_\alpha -e_\gamma)\\
<& (3\ch_0(E_\alpha))\left(\left(n_1 - n_2 \frac{\ch_0(E_\alpha)}{\ch_0(E_\beta)} + 3\ch_0(E_\alpha)\right)-3\ch_0(E_\alpha)\right)+1\\
< & 0.
\end{align*}
Note that $v'$ is on the line $L_{e_\alpha(-3)^+e_\alpha(-3)^r}$, $v'$ is not above C$_{LP}$. Since $v'$ is to the left of $E_\beta(-3)$, after moving along the direction $e_\gamma(-3)-e_\gamma$, $v'+a(e_\gamma(-3)-e_\gamma)$ is still not above C$_{LP}$ or $L_{e_\alpha(-3)^+e_\alpha(-3)^r}$. Therefore, $P-w$ is not above those two curves. It is not in TR$_{w E}$ for any $E$ to the left of $E_\alpha(-3)$.

For any exceptional $e$ between $L_{e_\alpha(-3)\pm}$ and $L_{e_\gamma(-3)\pm}$, by the assumption, $\ch_0(P-w)\leq \ch_0(e_\gamma)< \ch_0(e)$. By Lemma \ref{lem:highrkinstrip}, $P-w$ is not in the area between $L_{e^+e^r}$ and $L_{ee(3)}$. Since $w$ is above $L_{ee(3)}$, $P-w$ is not in TR$_{wE}$ for any exceptional $E$ between $L_{e_\alpha(-3)\pm}$ and $L_{e_\gamma(-3)\pm}$. 

The line segment $l_{(P-w)P}$ is above the character $e_\gamma(-3)$, hence above the line segment $l_{e_\gamma(-3)^re_\gamma}$. Since $l_{e_\gamma(-3)^re_\gamma}$ is above any exceptional characters between the vertical rays $L_{e_\alpha(-3)\pm}$ and $L_{e_\gamma\pm}$, the character $P-w$ is not in the triangle TR$_{wE}$ for any such exceptional bundle $E$.

We finish the claim that $w-P$ is outside TR$_{wE}$ for any exceptional bundle $E$. By Theorem \ref{thm: actual wall}, we know that $L_{Pw}$ is an actual wall.\\

The last step is to show that a divisor of $\mathfrak M^s_{L_{Pw}+}(w)$ is contracted at $L_{Pw}$. By Proposition \ref{prop:compcontainsstabobj} and Theorem \ref{prop:lastwall}, for $\sigma\in L_{Pw}$, $\dim \mathfrak M^s_{\sigma +}(w-P) = 1-\chi(P-w, P-w)$, and $\dim \mathfrak M^s_{\sigma +}(P) = 1-\chi(P, P)$. By the previous argument, they are both nonnegative.

By Lemma \ref{lemma: ext2 vanishing for stable factors} and Lemma \ref{lemma:extoftwostabobjsarestab},
\begin{align*}
\dim \mathfrak M^s_{\sigma_+}(w-P, P) &= \dim \mathfrak M^s_{\sigma_+}(w-P) +\dim \mathfrak M^s_{\sigma_+}(P) + \ext^1(w-P, P) -1 \\
&= 1-\chi(w-P, w-P) -\chi(P,P) -\chi(w-P,P) \\
&= 1- \chi(w,w)+\chi(P,w-P)\\
&= \dim \mathfrak M^s_{\sigma_+}(w)+\chi(P,w-P).
\end{align*}
So it suffices to show that $\chi(P, w-P)=-1$.
This is clear in case \emph{ii}):
\[\chi(P,w-P)= \chi(e_\gamma, w-e_\gamma)=-\chi(e_\gamma, e_\gamma)=-1.\]
In case \emph{i}),
\begin{align*}
\chi(P, w-P) &= -\chi(P,P) + \chi(P, w)\\
&=-\chi(P,P)-\chi((3\ch_0(E_\beta)-n_2)e_\alpha, w)\\
&= -\chi(P,P) - (3\ch_0(E_\beta)-n_2) \cdot n_2\chi(e_\alpha, e_\alpha) + (3\ch_0(E_\beta)-n_2) \cdot n_1\chi(e_\alpha, e_{\beta-3}) \\
&=-\chi(e_\gamma, e_\gamma) - \left(3\ch_0(E_\beta) - n_2\right)^2 + \left(3\ch_0(E_\beta)-n_2\right)\left(\chi(e_\alpha, e_\gamma)+ \chi(e_\gamma, e_\alpha)\right) \\
&\ \ \  - (3\ch_0(E_\beta)-n_2)n_2\\
&= -1 - (3\ch_0(E_\beta) - n_2)^2 + (3\ch_0(E_\beta) - n_2)\cdot 3\ch_0(E_\beta) - (3\ch_0(E_\beta) - n_2)n_2\\
&= -1.
\end{align*}
\end{proof}

\subsection{Nef cone}\label{sec4.2}

In this section we study the boundary of the nef cone of the moduli space $\mathfrak M_{GM}^{ss}(w)$. Due to Theorem \ref{left half upper plane's main theorem in
the body}, this is the first actual wall to the left of the vertical wall $L_{w\pm}$. We assume that the character $w$ is primitive, $\ch_0(w)>0$ and $\frac{\ch_1(w)}{\ch_0(w)} \in (-1, 0]$. The following lemma gives an obvious bound for the boundary of the nef cone.

\begin{lemma}
Suppose $\bd(w) \geq 2$, then $L_{\mathcal O(-1)w}$ is an actual wall for $w$.
\end{lemma} \label{lemma: first bound by O}
\begin{proof}
By Corollary \ref{cor:  actual wall} and Theorem \ref{prop:lastwall}, we need to show that $w$ is below the line $L_{\mathcal O(-1)\mathcal O(-1)^r}$.

The point $\mathcal O(-1)^r$ is the intersection of $\bd_{\frac{1}{2}}$ and $L_{\mathcal O\mathcal O(-1)}$, so on the $\cccp$, its coordinate  $\left(1,\frac{\ch_1}{\ch_0},\frac{\ch_2}{\ch_0}\right)=\left(1, \frac {1-\sqrt{5}}{2}, \frac {1-\sqrt{5}}{4}\right)$. Let $P$ be the intersection point of $L_{\mathcal O(-1)\mathcal O(-1)^r}$ and $L_{\mathcal O\pm}$. The function $\bd$ on the line segment $l_{\mathcal O(-1)P}$ reaches its maximum at $P=(1,0, -\frac{1+\sqrt{5}}{2})$, and $\bd_P=\frac{1+\sqrt{5}}{2} < 2$. So $w$ is below the line $L_{\mathcal O(-1)\mathcal O(-1)^r}$.
\end{proof}

By the lemma, when $\frac{\ch_1(w)}{\ch_0(w)} \in (k-1, k]$ for some integer $k$ and $\bd(w)\geq 2$, $L_{\mathcal O(k-1)w}$ is an actual wall.

\begin{lemma}
Suppose $\bd(w)\geq 10$, then the first lower rank wall $L_{vw}$ with $\ch_0(v) \leq \ch_0(w)$ is given by the character $v$ satisfying the following two conditions:
\begin{itemize}
\item $\frac{\ch_1(v)}{\ch_0(v)}$ is the greatest rational number less than $\frac{\ch_1(w)}{\ch_0(w)}$ with $\ch_0(v)\leq \ch_0(w)$;
\item Given the first condition, if $\ch_1(v)$ is even (odd resp.), let $\ch_2(v)$ be the greatest integer ($2\ch_2(v)$ be the greatest odd integer resp.), such that the point $v$ is either an exceptional character or not inside  
$\mathrm{Cone}_{LP}$.
\end{itemize}\label{lemma: first lower rank}
\end{lemma}

\begin{proof}
We may assume that $0<\chy(w)\leq 1$. Note that the slopes of $L_{e^le^+}$ and $L_{e^re^+}$ for any exceptional object with $\chy(e)$ in $[-1,0]$ are at least $-\frac{5}{2}$.

We first show that there is no actual wall between $L_{vw}$ and $L_{w\pm}$. Suppose that there is a character $v'$ with $\ch_0(v')\leq \ch_0(w)$ and $\frac{\ch_1(v')}{\ch_0(v')}< \frac{\ch_1(w)}{\ch_0(w)}$, such that $L_{v'w}$ is an actual wall between $L_{vw}$ and $L_{w\pm}$. By the previous lemma, we may assume that $\frac{\ch_1(v')}{\ch_0(v')} \geq 1$. Since $v'$ is either an exceptional character or below C$_{LP}$, by the assumptions on $v$,
\[\frac{\ch_2(v)}{\ch_0(v)} - \frac{\ch_2(v')}{\ch_0(v')} \geq -\frac{1}{\ch_0(v)} + \frac{5}{2}\left(\frac{\ch_1(v')}{\ch_0(v')}-\frac{\ch_1(v)}{\ch_0(v)}\right) - \frac{1}{\ch_0(v')^2}.\]
The coefficient $\frac{5}{2}$ of second term is with respect to the minimum slope of the Le Potier curve. The last term is for the case that $v'$ is exceptional: $\che(e)-\che(e^+) = \frac{1}{\ch_0(e)^2}$. This inequality holds since otherwise $v-(0,0,1)$ will be below C$_{LP}$ with smaller $\ch_2(v)$.

Write $d_\chi := -\frac{\ch_1(v')}{\ch_0(v')}+\frac{\ch_1(v)}{\ch_0(v)}$ for simplicity, then $\frac{1}{\ch_0(v)\ch_0(v')} \leq d_\chi \leq 1$.
Since $L_{v'w}$ is between $L_{vw}$ and $L_{w\pm}$, in other words, $v$ is below $l_{v'w}$, we have the inequality:
\begin{align*}
\che(v')-\che(w)  &\leq \left(\che(v')-\che(v)\right)\frac{\chy(w)-\chy(v')}{d_\chi}\\
&\leq \left(\frac{1}{\ch_0(v)} + \frac{1}{\ch_0(v')^2} + \frac{5}{2}d_\chi\right) \cdot \left(1+ \frac{1}{d_\chi}\cdot \left(\chy(w)-\chy(v)\right)\right)\\
&\leq 1+ 1+ \frac{5}{2} + \frac{1}{d_\chi}\cdot \left(\chy(w)-\chy(v)\right) \cdot \left(\frac{1}{\ch_0(v)} + \frac{1}{\ch_0(v')^2} + \frac{5}{2}d_\chi\right)\\
&\leq \frac{9}{2} + \frac{1}{\ch_0(w)}\left(\frac{1}{d_\chi}\left(\frac{1}{\ch_0(v)} + \frac{1}{\ch_0(v')}\right) + \frac{5}{2}\right)\\
&\leq \frac{9}{2} + 1 + 1 + \frac{5}{2} = 9.
\end{align*}
Therefore
\begin{align*}
\bd_w =& \che(w)+\frac{1}{2}\left(\che(w)\right)^2\leq -\che(w)+\frac{1}{2}\left(\che(v')\right)^2\\
=& - \che(w)+\che(v') +\bd_{v'}<9 + 1=10,
\end{align*}
which contradicts to our assumption.\\

We next show that $L_{vw}$ is an actual wall. By Corollary \ref{cor: actual wall}, it suffices to prove that $v$ is not in TR$_{wE}$ for any exceptional bundle $E$ such that $-1\leq \chy(E)\leq 0$. Suppose that $v $ is in TR$_{wE}$ for an exceptional bundle $E$, then since $\bd_w \geq 10$, the slope of $L_{wE}$ is less than $-9$. The $\frac{\ch_1}{\ch_0}$-width of TR$_{wE}$ is less than
\[\frac{\text{the length of }l_{ee^+}}{9-\frac{5}{2}} < \frac{1}{6\ch_0(E)^2}.\]
Hence if $v$ is in TR$_{wE}$, then
\[\frac{1}{\ch_0(E)\ch_0(v)} \leq \frac{\ch_1(v)}{\ch_0(v)}-\frac{\ch_1(E)}{\ch_0(E)} < \frac{1}{6\ch_0(E)^2}.\]
In this way, $\ch_0(w)\geq \ch_0(v)>6\ch_0(E)$. In particular,
$\ch_0(E)\leq \ch_0(w)$. Note that $L_{wE}$ becomes a lower rank wall closer to between $L_{wv}$ and $L_{w\pm}$. By Corollary \ref{cor: actual wall}, $L_{wE}$ is an actual wall. But this is not possible by the proof in the first part. Therefore, $v$ is not in TR$_{wE}$ for any exceptional bundle $E$.
\end{proof}

Now we can describe the boundary of the nef cone:

\begin{theorem}
Let $w$ be a primitive character with $\ch_0(w) > 0$ and $\bd_w \geq 10$, the first actual wall for $\mathfrak M_{GM}^{s}(w)$ is given by $L_{vw}$, where $v$ is the character defined in Lemma \ref{lemma: first lower rank}.
\label{thm:nefcone}
\end{theorem}

\begin{proof}
We may assume that $\frac{\ch_1(w)}{\ch_0(w)} \in (-1, 0]$, and by Lemma \ref{lemma: first lower rank}, we only need to show that any higher rank actual wall is not between $L_{vw}$ and $L_{w\pm}$. Let $v'$ be a character satisfying the properties in Theorem \ref{thm: actual wall} with $\ch_0(v') = \ch_0(w) + r$ for some positive integer $r$. 


The slope of $L_{vw}$ is less than
\[(\bd_w-1)\lfs\left(\frac{\ch_1(w)}{\ch_0(w)}-\frac{\ch_1(v)}{\ch_0(v)}\right) < -9\ch_0(w).\]
So the left intersection point of $L_{vw}\cap\bd_0$ has $\frac{\ch_1}{\ch_0}$-coordinate less than $-9\ch_0(w)$. Since $v'-w$ is to the left of this point, we get the inequality
\[\chy(v'-w) < -9\ch_0(w),\]
hence
\[r < \frac{1}{9} \cdot \frac{\ch_1(w)-\ch_1(v')}{\ch_0(w)}.\]
By Lemma \ref{lemma: first bound by O}, we have $\frac{\ch_1(v')}{\ch_0(v') }> -1$, so
\[\ch_1(v') > -\ch_0(w)-r.\]
Therefore,
\[r< \frac{1}{9}\cdot \frac{\ch_1(w)+\ch_0(w)+r}{\ch_0(w)} \leq \frac{1}{9} \cdot\frac{\ch_0(w)+r}{\ch_0(w)}\leq \frac{1}{9} +\frac{r}{9}.\]
This leads to the contradiction since $r<1$ and cannot be a positive integer.
\end{proof}

\subsection{A concrete example}\label{sec4.4}
In this section, we apply the criteria for actual walls and compute the stable base locus walls on the primitive side for the moduli space of stable objects of character $w=(\ch_0,\ch_1,\ch_2)$ $=$ $(4,0,-15)$.

We first determine the last wall of $w$. The equation of $L^{\bd_{\frac{1}{2}}}_w$ is given by
\[\che+\frac{3\sqrt{35}}{14}\chy+\frac{15}{4}=0.\]
The $\chy$ coordinates of the intersection points $L^{\bd_{\frac{1}{2}}}_w\cap \bd_{\frac{1}{2}}$ are $-\frac{\sqrt{35}}{2}\pm\frac{3}{2}$. The larger one is approximately $-1.458$ and the intersection point falls in the segment between $e^l$ and $e^r$, for the exceptional character $e$ given by $E_{\left(\frac{3}{2}\right)}$, which is the cotangent bundle $\Omega$. 

By the Hirzebruch-Riemann-Roch formula in the proof for Lemma \ref{lemma:chiformrho}, 
\[\chi(\Omega,w) = \chi\left((2,-3,\frac{3}{2}),(4,0,-15)\right)=-30 +6+18+8=2>0.\]
Therefore $w$ is above the line $L_{e^le^+}$, and is of Case 1 in the Definition \ref{def:lastwall}. The last wall of $w$ is given by $L_{\Omega w}$ with equation:
\[\che+3\chy+\frac{15}{4}=0.\]
We now look for all the lower rank walls. By Corollary \ref{cor: actual wall}, we only need determine all characters $v\in K(\pp)$ such that
\begin{itemize}
\item $0<\ch_0(v)\leq \ch_0(w)=4$, $\chy(v)<\chy(w)=0$;
\item $v$ is between $L_{w\pm}$ and $\lw_w$;
\item $v$ is exceptional or not inside Cone$_{LP}$;
\item $v$ is not in TR$_{wE}$ for any exceptional character $E$.
\end{itemize}

When $\ch_0(v)$ is $1$, $\ch_1(v)$ can only be $-1$, $v$ is either $(1,-1,\frac{1}{2})$ or $(1,-1,-\frac{1}{2})$.

When $\ch_0(v)$ is $2$, $\ch_1(v)$ can be $-1$ or $-2$, $v$ is one of the characters as follows: $(2,-1,-\frac{1}{2})$; $(2,-1,-\frac{3}{2})$; $(2,-1,-\frac{5}{2})$; $(2,-1,-\frac{7}{2})$; $(2,-2,-1)$.

When $\ch_0(v)$ is $3$, $\ch_1(v)$ can be $-1$, $-2$, $-3$ or $-4$, $v$ is one of the following characters:
\begin{itemize}
\item $(3,-1,-\frac{3}{2})$; $(3,-1,-\frac{5}{2})$; $(3,-1,-\frac{7}{2})$; $\dots$ $(3,-1,-\frac{15}{2})$;
\item $(3,-2,-1)$; $(3,-2,-2)$; $(3,-2,-3)$; $(3,-2,-4)$; $(3,-2,-5)$;

\item $(3,-3,-\frac{3}{2})$; $(3,-4,1)$.
\end{itemize}

When $\ch_0(v)$ is $4$, we have $-5\leq\ch_1(v)\leq -1$,  $v$ is one of the following characters:
\begin{itemize}
\item $(4,-1,-\frac{5}{2})$; $(4,-1,-\frac{7}{2})$; $\dots$ $(4,-1,-\frac{23}{2})$; 
\item $(4,-2,-4)$; $(4,-2,-5)$; $\dots$ $(4,-2,-8)$; 
\item $(4,-3,-\frac{3}{2})$; $(4,-3,-\frac{5}{2})$; $\dots$ $(4,-3,-\frac{11}{2})$;
\item $(4,-4,-2)$; $(4,-5,\frac{1}{2})$.
\end{itemize}

The nef boundary of $\mathfrak M^s_{GM}(w)$ is the wall $L_{w(4,-1,-\frac{5}{2})}$.

We now compute the characters that are contained in TR$_{wE}$ for some exceptional $E$. By Lemma \ref{lem:highrkinstrip}, we only need consider the exceptional bundles $\mathcal O(-1)$ and $\Omega(1)$. The equations for the three edges of TR$_{w\mathcal O(-1)}$ are:
\[\che -\frac{1}{2}\chy=0, \; \che+\frac{5}{2}\chy+3=0,\; \che +\frac{17}{4}\chy+\frac{15}{4}=0.\]
By a direct computation, characters $(3,-2,3)$, $(4,-3,-\frac{5}{2})$, $(4,-3,-\frac{7}{2})$ are in TR$_{w\mathcal O(-1)}$. The equations for the three edges of TR$_{w\Omega(1)}$ are:
\[\che -\chy=0, \; \che+2\chy+\frac{3}{2}=0,\; \che + 7 \chy+\frac{15}{4}=0.\]
The coordinates of vertices are 
$(1,-\frac{1}{2},-\frac{1}{2})$, $(1,-\frac{9}{20},-\frac{3}{5})$, $(1,-\frac{15}{32},-\frac{15}{32})$. Since for any $v$, $\chy(v)$ is not in $(-\frac{1}{2},-\frac{9}{20})$, there is no $v$ in TR$_{w\Omega(1)}$.

To find the higher rank walls, we first determine the bound for $\ch_0(v)$. $\lw_w\cap \bd_{\leq 0}=\left\{\left(1,-3+\sqrt{\frac{3}{2}},\frac{1}{2}\left(-3+\sqrt{\frac{3}{2}}\right)^2\right),\; \left(1,-3-\sqrt{\frac{3}{2}},\frac{1}{2}\left(-3-\sqrt{\frac{3}{2}}\right)^2\right)\right\}$. Let $v\in K(\pp)$ be a character such that
\begin{itemize}
\item $\ch_0(v)>\ch_0(w)=4$, $\chy(v)<\chy(w)=0$;
\item $v$ is between $L_{w\pm}$ and $\lw_w$;
\item $v$ and $v-u$ are exceptional or not inside Cone$_{LP}$;
\item $v$ and $v-u$ are not in TR$_{wE}$ for any exceptional character $E$.
\end{itemize}
Since $v$ and $v-u$ are on the different components of $L_{vw}\cap\bd_{\geq 0}$, we have the inequalities:
\[\chy(v)\geq-3+\sqrt{\frac{3}{2}},\; \chy(v-u)\leq -3 -\sqrt{\frac{3}{2}}.\]
Therefore, \begin{equation}\left(-3+\sqrt{\frac{3}{2}}\right)\ch_0(v)\leq \ch_1(v) \leq \left(-3-\sqrt{\frac{3}{2}}\right)(\ch_0(v)-4). \label{eqch}\end{equation}
We get a bound for $\ch_0(v)$: $\ch_0(v)\leq 2+2\sqrt{6}<7$. When $\ch_0(v)$ is $6$, by (\ref{eqch}), $\ch_1(v)\leq -2\left(-3-\sqrt{\frac{3}{2}}\right)<-8$. Therefore, $\chy(v)\leq -9 =\chy(E_w)$, which is not possible. 

When $\ch_0(v)$ is $5$, by (\ref{eqch}), $\ch_1(v)$ can be $-5$, $-6$ or $-7$. $v$ is one of the following characters:

\[(5,-5,-\frac{7}{2});\; (5,-6,0);\; (5,-7,\frac{5}{2}).\]

These characters $v$ and $w-v$ are not contained in TR$_{wE}$ for any exceptional $E$. Combining Theorem \ref{left half upper plane's main theorem in
the body} and Theorem \ref{thm: actual wall}, we may draw the stable base locus decomposition walls in the divisor cone of $\mathfrak M^s_{GM}(w)$ as follows.

\begin{center}
\begin{tikzpicture}[domain=1:5]

\tikzset{%
    add/.style args={#1 and #2}{
        to path={%
 ($(\tikztostart)!-#1!(\tikztotarget)$)--($(\tikztotarget)!-#2!(\tikztostart)$)%
  \tikztonodes},add/.default={.2 and .2}}
}

\newcommand\XA{0.1}

\coordinate (W) at (0,-3.75);
\node at (W) {$\bullet$};


\coordinate (V) at (-1.5,0.75);
\node [opacity=\XA] at (V) {$\bullet$};
\draw [add= 0 and 1] (W) to (V) node[above]{Eff};

\coordinate (V) at (-1/4,-5/8);
\node [opacity=\XA]at (V) {$\bullet$};
\draw [add= 0 and 2] (W) to (V) node[above]{Nef};

\coordinate (V) at (-1,.5);
\node [opacity=\XA] at (V) {$\bullet$};
\draw [add= 0 and 1.2][opacity=\XA] (W) to (V);

\coordinate (V) at (-1,-.5);
\node [opacity=\XA]at (V) {$\bullet$};
\draw [add= 0 and 1.5][opacity=\XA] (W) to (V);

\coordinate (V) at (-.5,-.25);
\node [opacity=\XA] at (V) {$\bullet$};
\draw [add= 0 and 1.5][opacity=\XA] (W) to (V);

\coordinate (V) at (-.5,-.75);
\node[opacity=\XA]at (V) {$\bullet$};
\draw [opacity=\XA][add= 0 and 1.5] (W) to (V);

\coordinate (V) at (-.5,-1.25);
\node [opacity=\XA]at (V) {$\bullet$};
\draw [add= 0 and 2.5] [opacity=\XA] (W) to (V);

\
\coordinate (V) at (-.5,-1.75);
\node[opacity=\XA] at (V) {$\bullet$};
\draw[opacity=\XA] [add= 0 and 3.5] (W) to (V);

\coordinate (V) at (-1/3,-0.5);
\node[opacity=\XA] at (V) {$\bullet$};
\draw[opacity=\XA] [add= 0 and 1.5] (W) to (V);

\coordinate (V) at (-1/3,-5/6);
\node[opacity=\XA] at (V) {$\bullet$};
\draw[opacity=\XA] [add= 0 and 1.5] (W) to (V);

\coordinate (V) at (-1/3,-7/6);
\node[opacity=\XA] at (V) {$\bullet$};
\draw[opacity=\XA] [add= 0 and 1.5] (W) to (V);

\coordinate (V) at (-1/3,-9/6);
\node[opacity=\XA] at (V) {$\bullet$};
\draw[opacity=\XA] [add= 0 and 1.5] (W) to (V);

\coordinate (V) at (-1/3,-11/6);
\node[opacity=\XA] at (V) {$\bullet$};
\draw[opacity=\XA] [add= 0 and 1.5] (W) to (V);

\coordinate (V) at (-1/3,-13/6);
\node[opacity=\XA] at (V) {$\bullet$};
\draw[opacity=\XA] [add= 0 and 1.5] (W) to (V);

\coordinate (V) at (-1/3,-15/6);
\node[opacity=\XA]  at (V) {$\bullet$};
\draw[opacity=\XA] [add= 0 and 1.5] (W) to (V);

\coordinate (V) at (-2/3,-1/3);
\node [opacity=\XA] at (V) {$\bullet$};
\draw[opacity=\XA] [add= 0 and 1.5] (W) to (V);

\coordinate (V) at (-2/3,-2/3);
\node [][opacity=\XA] at (V) {$\bullet$};
\draw[opacity=\XA] [add= 0 and 1.5] (W) to (V);

\coordinate (V) at (-2/3,-4/3);
\node [opacity=\XA] at (V) {$\bullet$};
\draw[opacity=\XA] [add= 0 and 1.5] (W) to (V);

\coordinate (V) at (-2/3,-5/3);
\node [][opacity=\XA] at (V) {$\bullet$};
\draw[opacity=\XA] [add= 0 and 1.5] (W) to (V);

\coordinate (V) at (-4/3,1/3);
\node[opacity=\XA] at (V) {$\bullet$};
\draw[opacity=\XA] [add= 0 and 1] (W) to (V);


\coordinate (V) at (-1/4,-7/8);
\node[opacity=\XA] at (V) {$\bullet$};
\draw[opacity=\XA][add= 0 and 1.5] (W) to (V);

\coordinate (V) at (-1/4,-9/8);
\node[opacity=\XA] at (V) {$\bullet$};
\draw[opacity=\XA] [add= 0 and 1.5] (W) to (V);
\coordinate (V) at (-1/4,-11/8);
\node[opacity=\XA] at (V) {$\bullet$};
\draw[opacity=\XA] [add= 0 and 1.5] (W) to (V);
\coordinate (V) at (-1/4,-13/8);
\node[opacity=\XA] at (V) {$\bullet$};
\draw[opacity=\XA] [add= 0 and 1.5] (W) to (V);
\coordinate (V) at (-1/4,-15/8);
\node[opacity=\XA] at (V) {$\bullet$};
\draw[opacity=\XA] [add= 0 and 1.5] (W) to (V);
\coordinate (V) at (-1/4,-17/8);
\node[opacity=\XA] at (V) {$\bullet$};
\draw[opacity=\XA] [add= 0 and 1.5] (W) to (V);
\coordinate (V) at (-1/4,-19/8);
\node[opacity=\XA] at (V) {$\bullet$};
\draw[opacity=\XA] [add= 0 and 1.5] (W) to (V);
\coordinate (V) at (-1/4,-21/8);
\node[opacity=\XA] at (V) {$\bullet$};
\draw[opacity=\XA] [add= 0 and 1.5] (W) to (V);
\coordinate (V) at (-1/4,-23/8);
\node[opacity=\XA] at (V) {$\bullet$};
\draw[opacity=\XA] [add= 0 and 1.5] (W) to (V);

\coordinate (V) at (-2/4,-4/4);
\node[opacity=\XA] at (V) {$\bullet$};
\draw[opacity=\XA] [add= 0 and 1.5] (W) to (V);

\coordinate (V) at (-2/4,-6/4);
\node[opacity=\XA] at (V) {$\bullet$};
\draw[opacity=\XA] [add= 0 and 1.5] (W) to (V);

\coordinate (V) at (-2/4,-8/4);
\node[opacity=\XA] at (V) {$\bullet$};
\draw[opacity=\XA] [add= 0 and 1.5] (W) to (V);

\coordinate (V) at (-3/4,-3/8);
\node [opacity=\XA]at (V) {$\bullet$};
\draw[opacity=\XA] [add= 0 and 1.5] (W) to (V);

\coordinate (V) at (-3/4,-9/8);
\node [opacity=\XA] at (V) {};
\draw[opacity=\XA] [add= 0 and 1.5] (W) to (V);

\coordinate (V) at (-3/4,-11/8);
\node [opacity=\XA] at (V) {};
\draw[opacity=\XA] [add= 0 and 1] (W) to (V);

\coordinate (V) at (-5/4,1/8);
\node[opacity=\XA] at (V) {};
\draw[opacity=\XA] [add= 0 and 1] (W) to (V);


\coordinate (V) at (-5/5,-7/10);
\node[opacity=\XA] at (V) {$\bullet$};
\draw[opacity=\XA] [add= 0 and 1] (W) to (V);

\coordinate (V) at (-6/5,0);
\node[opacity=\XA] at (V) {$\bullet$};
\draw[opacity=\XA] [add= 0 and 1] (W) to (V);

\coordinate (V) at (-7/5,5/10);
\node[opacity=\XA] at (V) {$\bullet$};
\draw[opacity=\XA] [add= 0 and 1] (W) to (V);

\draw[->] (-4,-3.75) -- (0,-3.75) node[above right] {$w$}-- (1.5,-3.75) node[above right] {$B$};

\draw[->,opacity =0.3] (-4,0) -- (2.5,0) node[above right] {$\frac{ch_1}{ch_0}$};

\draw[->][] (0,-4.25)-- (0,0) node [above right] {O} --  (0,6) node[right] {$H$};

\draw[->,opacity=0.3] (0,-4.25)-- (0,0) node [above right] {O} --  (0,4) node[right] {$\frac{ch_2}{ch_0}$};

\draw [thick](-3,4.5) parabola bend (0,0) (1.5,1.125);

\end{tikzpicture}
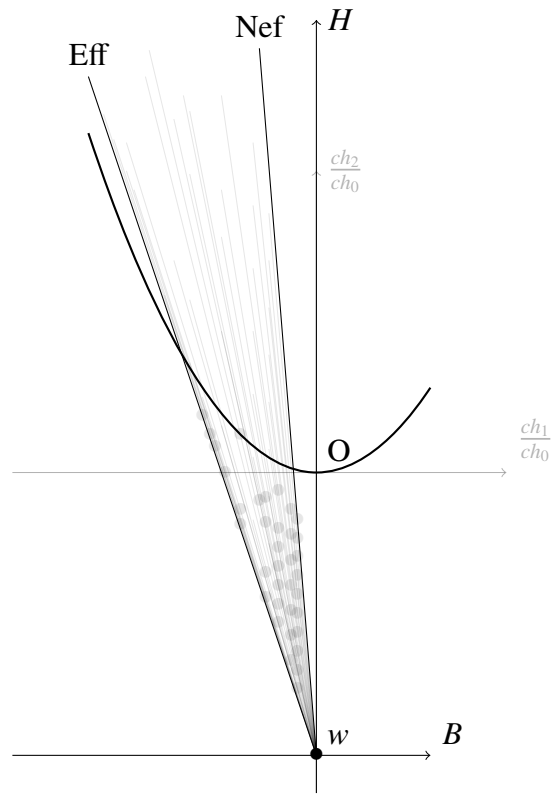
The stable base locus decomposition of the effective cone of $\mathfrak M^s_{GM}(4,0,-15)$
\end{center}


\bibliographystyle{alpha}\bibliography{}

\begin{thebibliography}{10}

\bibitem[ABCH]{ABCH}
D.~Arcara, A.~Bertram, I.~Coskun, and J.~Huizenga.
\newblock The minimal model program for the {H}ilbert scheme of points on
  {$\Bbb{P}^2$} and {B}ridgeland stability.
\newblock {\em Adv. Math.}, 235:580--626, 2013.

\bibitem[AKO]{AKO}
D.~Auroux, L.~Katzarkov, and D.~Orlov.
\newblock Mirror symmetry for del {P}ezzo surfaces: vanishing cycles and
  coherent sheaves.
\newblock {\em Invent. Math.}, 166(3):537--582, 2006.

\bibitem[BM1]{BM}
A.~Bayer and E.~Macr{\`{\i}}.
\newblock The space of stability conditions on the local projective plane.
\newblock {\em Duke Math. J.}, 160(2):263--322, 2011.


\bibitem[BM2]{BM1}
A.~Bayer and E.~Macr{\`{\i}}.
\newblock M{MP} for moduli of sheaves on {K}3s via wall-crossing: nef and
  movable cones, {L}agrangian fibrations.
\newblock {\em Invent. Math.}, 198(3):505--590, 2014.

\bibitem[BM3]{BM2}
A.~Bayer and E.~Macr{\`{\i}}.
\newblock Projectivity and birational geometry of {B}ridgeland moduli spaces.
\newblock {\em J. Amer. Math. Soc.}, 27(3):707--752, 2014.



\bibitem[BMS]{BMS}
A.~Bayer, E.~Macr{\`{\i}}, and P.~Stellari.
\newblock Stability conditions on abelian threefolds and some {C}alabi-{Y}au
  threefolds.
\newblock 2014.


\bibitem[BMW]{BMW}
A.~Bertram, C.~Martinez, and J.~Wang.
\newblock The birational geometry of moduli spaces of sheaves on the projective plane.
\newblock {\em Geom. Dedicata}, 173:37--64, 2014.

\bibitem[Br1]{Bri07}
T.~Bridgeland.
\newblock Stability conditions on triangulated categories.
\newblock {\em Ann. of Math. (2)}, 166(2):317--345, 2007.

\bibitem[Br2]{Bri08}
T.~Bridgeland.
\newblock Stability conditions on {$K3$} surfaces.
\newblock {\em Duke Math. J.}, 141(2):241--291, 2008.

\bibitem[CH1]{CH}
I.~Coskun and J.~Huizenga.
\newblock Interpolation, {B}ridgeland stability and monomial schemes in
              the plane.
\newblock {\em J. Math. Pures Appl. (9)}, 102(5):930--971, 2014.


\bibitem[CH2]{CH2}
I.~Coskun and J.~Huizenga.
\newblock The ample cone of moduli spaces of sheaves on the plane.
\newblock {\em eprint arXiv:1409.5478}.

\bibitem[CH3]{CH3}
I. ~Coskun and J. ~Huizenga.
\newblock The birational geometry of the moduli spaces of sheaves on
              {$\Bbb P^2$}.
\newblock {\em Proceedings of the {G}\"okova {G}eometry-{T}opology
              {C}onference 2014}, 114--155, 2015.

\bibitem[CHW]{CHW}
I.~Coskun, J.~Huizenga and M.~Woolf
\newblock The effective cone of the moduli space of sheaves on the plane.
\newblock {\em eprint arXiv:1401.1613},

\bibitem[De]{SGA}
P.~Deligne.
\newblock {\em Cohomologie \'etale}.
\newblock Lecture Notes in Mathematics, Vol. 569. Springer-Verlag, Berlin,
  1977.
\newblock S{\'e}minaire de G{\'e}om{\'e}trie Alg{\'e}brique du Bois-Marie SGA
  4${1{\o}er 2}$, Avec la collaboration de J. F. Boutot, A. Grothendieck, L.
  Illusie et J. L. Verdier.

\bibitem[DH]{DoHu}
I.~V. Dolgachev and Y.~Hu.
\newblock Variation of geometric invariant theory quotients.
\newblock {\em Inst. Hautes \'Etudes Sci. Publ. Math.}, (87):5--56, 1998.
\newblock With an appendix by Nicolas Ressayre.

\bibitem[DP]{DP}
J.-M. Drezet and J.~Le~Potier.
\newblock Fibr\'es stables et fibr\'es exceptionnels sur {${\bf P}_2$}.
\newblock {\em Ann. Sci. \'Ecole Norm. Sup. (4)}, 18(2):193--243, 1985.

\bibitem[Gi]{Gi}
V.~Ginzburg.
\newblock Lectures on nakajima's quiver varieties.
\newblock {\em Geometric methods in representation theory. {I}}, 24:145--219, 2012.


\bibitem[GR]{GorRu}
A.~L. Gorodentsev and A.~N. Rudakov.
\newblock Exceptional vector bundles on projective spaces.
\newblock {\em Duke Math. J.}, 54(1):115--130, 1987.

\bibitem[Gr]{Gr}
A.~Grothendieck.
\newblock El\'ements de g\'eom\'etrie alg\'ebrique, chapters iii and iv.
\newblock {\em Inst. Hautes \'Etudes Sci. Publ. Math.}, 11, 1961.

\bibitem[Ki]{Ki}
A.~D. King.
\newblock Moduli of representations of finite-dimensional algebras.
\newblock {\em Quart. J. Math. Oxford Ser. (2)}, 45(180):515--530, 1994.

\bibitem[LeP]{LeP}
J.~Le~Potier.
\newblock {\em Lectures on vector bundles}, volume~54 of {\em Cambridge Studies
  in Advanced Mathematics}.
\newblock Cambridge University Press, Cambridge, 1997.
\newblock Translated by A. Maciocia.

\bibitem[LZ]{LZ}
C. ~Li and X. ~Zhao.
\newblock The {MMP} for deformations of Hilbert schemes of points on the projective plane.
\newblock {\em eprint arXiv:1312.1748}.


\bibitem[Ma]{Mac}
E.~Macr{\`{\i}}.
\newblock Stability conditions on curves.
\newblock {\em Math. Res. Lett.}, 14(4):657--672, 2007.

\bibitem[Or]{Or}
D.~O. Orlov.
\newblock Projective bundles, monoidal transformations, and derived categories
  of coherent sheaves.
\newblock {\em Izv. Ross. Akad. Nauk Ser. Mat.}, 56(4):852--862, 1992.



\bibitem[Ta]{Abe}
T. ~Takeshi.
\newblock Strange duality for height zero moduli spaces of sheaves on
              {$\Bbb P^2$}.
\newblock {\em Michigan Math. J.}, 64:569--586, 2015.


\bibitem[Th]{Th}
M.~Thaddeus.
\newblock Geometric invariant theory and flips.
\newblock {\em J. Amer. Math. Soc.}, 9(3):691--723, 1996.









\bibitem[Wo]{Woolf}
M.~Woolf.
\newblock Nef and effective cones on the moduli space of torsion sheaves on the projective plane.
\newblock {\em eprint arXiv:1305.1465}.







\end{thebibliography}

\end{document}